\documentclass[preprint,review]{elsarticle}
\usepackage{amssymb}
\usepackage{latexsym}

\usepackage{url}
\usepackage{xcolor}
\definecolor{newcolor}{rgb}{.8,.349,.1}

\usepackage{lipsum}
\usepackage{amsfonts}
\usepackage{graphicx}
\usepackage{epstopdf}
\usepackage{algorithmic}
\usepackage{mathtools}
\usepackage{amsmath}
\usepackage{bm}
\usepackage{float}

\def\R{\mathbb{R}}
\def\N{\mathbb{N}}

\def\Omegat{\widetilde{\Omega}}
\def\Pt{\widetilde{\bm{P} }}

\def\l{\left\langle}
\def\r{\right\rangle}
\def\1{\chi}

\newcommand{\pf}[2]{\frac{\partial #1 }{\partial #2}}
\usepackage{stmaryrd}
\usepackage{enumerate}
\usepackage{todonotes}
\usepackage{physics}
\usepackage{bibentry}
\usepackage{graphicx}
\usepackage{amssymb}
\usepackage{array}
\usepackage{amsmath}
\usepackage{epstopdf}
\usepackage{epsfig}
\usepackage{subfigure}
\usepackage{pstricks}
\usepackage{fancyheadings}
\usepackage[margin=1in]{geometry}
\usepackage[utf8]{inputenc}
\usepackage[english]{babel}

\usepackage{amsthm}
\newtheorem{theorem}{Theorem}[section]
\newtheorem{lemma}[theorem]{Lemma}
\newtheorem{definition}[theorem]{Definition}

\usepackage{stackengine}

\def\x{\hspace{3ex}}    %BETWEEN TWO 1-DIGIT NUMBERS
\stackMath

\newtheorem{remark}{Remark}
\begin{document}

%\verso{K. Duru, F. Fung, C. Williams}

\begin{frontmatter}

\title{Upwind summation by parts finite difference methods for large scale elastic wave simulations in 3D complex geometries \tnoteref{tnote1}}%
\tnotetext[tnote1]{Upwind SBP FD methods for large scale elastic wave simulations in 3D complex geometries}

\author[1]{Kenneth {Duru}\corref{cor1}}
\cortext[cor1]{Corresponding author: 
  Tel.: +0-000-000-0000;  
  fax: +0-000-000-0000;}
  \ead{kenneth.duru@anu.edu.au}
\author[1]{Frederick  {Fung}}
\ead{shilufred.feng@anu.edu.au}
%\fntext[fn1]{This is author footnote for second author.}  
\author[1]{Christopher {Williams}}
%% Third author's email
\ead{Christoper.Williams@anu.edu.au}
%\author[2]{Given-name4 \snm{Surname4}}

\address[1]{Mathematical Sciences Institute, Australian National University, Australia.}
%\address[2]{Affiliation 2, Address, City and Postal Code, Country}

% \received{1 May 2013}
% \finalform{10 May 2013}
% \accepted{13 May 2013}
% \availableonline{15 May 2013}
% \communicated{S. Sarkar}

\begin{abstract}
High-order accurate summation-by-parts (SBP) finite difference (FD) methods constitute efficient numerical methods for simulating large-scale hyperbolic wave propagation problems. Traditional SBP FD operators that approximate first-order spatial derivatives with central-difference stencils often have spurious unresolved numerical wave-modes in their computed solutions. Recently derived high order accurate upwind SBP operators based non-central (upwind) FD stencils have the potential to suppress these poisonous spurious wave-modes on marginally resolved computational grids. In this paper, we demonstrate that not all high order upwind SBP FD operators are applicable. Numerical dispersion relation analysis shows that odd-order upwind SBP FD operators also support spurious unresolved high-frequencies on marginally resolved meshes. Meanwhile, even-order upwind SBP FD operators (of order $2, 4, 6$) do not support spurious unresolved high frequency wave modes and also have better numerical dispersion properties.

For all the upwind SBP FD operators we discretise the three space dimensional (3D) elastic wave equation on boundary-conforming curvilinear meshes.
Using the energy method we prove that the semi-discrete approximation is stable and energy-conserving.
We derive a priori error estimate and prove the convergence of the numerical error.
Numerical experiments for  the 3D elastic wave equation in complex geometries corroborate the theoretical analysis.
Numerical simulations of the 3D elastic wave equation in heterogeneous media with complex non-planar free surface topography are given, including numerical simulations of community developed seismological benchmark problems. 
Computational results show that even-order upwind SBP FD operators are more efficient, robust and less prone to numerical dispersion errors on marginally resolved meshes when compared to the odd-order upwind and traditional SBP FD operators.
Finally, scaling tests demonstrate nearly perfect strong scaling and verify the efficiency of our parallel implementation of the high order upwind SBP methods.
\end{abstract}

% \begin{keywords}
% Finite difference method \sep high-order accuracy \sep stability \sep boundary treatment \sep high frequency seismic wave propagation \sep complex geometries  \sep high performance computing 
% %\todo{please can you insert keywords here}
% \end{keywords}

\end{frontmatter}

\section{Introduction}
High fidelity numerical simulations of seismic (elastic) waves are prevalent in many applications such as earthquake engineering, natural minerals and energy resources exploration, strong-ground motion analysis, and underground fluid injection monitoring.
Seismic waves emanating from geophysical events can propagate hundreds to thousands of kilometres interacting with geological structure and complicated free-surface topography. 
Exploration seismology and natural earthquake hazard mitigation increasingly rely on multi-scale and high frequency (0--20 Hz)  simulations.
Often, surface and interface waves \cite{Rayleigh1885}  are the largest amplitude wave modes and are arguably the most important wave modes in the medium.
Therefore accurate and efficient numerical simulation of seismic surface and interface waves, and scattering of high-frequency wave modes by complex non-planar free-surface topography, are critical for assessing and quantifying seismic risks and hazards ~\cite{Graves_etal2011}.

In this study, we derive efficient high order accurate numerical methods for large scale numerical simulations of seismic waves in complex geometries. 
We consider the elastic wave equation in the first-order form, where the unknowns are the particle velocity vector and stress fields.
The main classes of numerical methods for the solution of time-dependent PDEs are the spectral method, finite element (FE) method, finite difference (FD) method,  finite volume (FV) method and the discontinuous Galerkin (DG) finite element method. 
They all have different strengths and weaknesses. 
Computational efficiency has continued to make the use of FD methods on structured grids attractive. 
However, the presence of boundary conditions and complex non-planar free-surface topography make the design of stable and accurate FD methods challenging. 
For computational seismology, staggered grids FD methods on  Cartesian meshes are the industry standard because they are efficient and have optimal numerical dispersion properties. 
However, the design of high order accurate and stable staggered FD methods for the elastic wave equation in complex geometries is a challenge. 
Although some progress is being made in this direction \cite{OssianAnders2020}.

For well-posed initial boundary value problems (IBVP), the summation-by-parts (SBP) FD \cite{BStrand1994} with the simultaneous approximation term (SAT) \cite{CarpenterGottliebAbarbanel1994,Mattsson2003,NordstromCarpenter2001} technique for implementing boundary conditions enables the development of stable numerical approximations on smooth geometries. The methods can be extended to complex geometries using curvilinear coordinate transforms and multi-block schemes \cite{DuruandDunham2016,Svard2004}. 
Traditional SBP FD operators are based on central finite difference formula with special one-sided boundary closures, designed such that the operator preserves the integration by parts principle \cite{BStrand1994}.
Often, traditional  SBP FD operators which approximate the spatial derivative suffer from spurious unresolved wave-modes in their numerical solutions. 
For marginally resolved solutions, these spurious wave-modes have the potential to destroy the accuracy of numerical solutions for a first-order hyperbolic partial differential equation, such as the elastic wave equation.

To ensure the accuracy of numerical solutions of elastic wave equations in complex geometries, we  discretise  the  3D  elastic  wave  equation  with  a  pair  of  non-central  (upwind)  FD  stencils  \cite{Mattsson2017}, on  boundary-conforming  curvilinear  meshes. The main benefit for these operators \cite{Mattsson2017,DovgilovichSofronov2015} is that they have the potential to suppress poisonous spurious oscillations from unresolved wave-modes, which can destroy the accuracy of numerical simulations.
However,  these operators are asymmetric and dissipative, can potentially destroy symmetries that exist in the continuum problem.  For example, the linear elastic wave equation with free-surface boundary conditions preserves the mechanical energy for all time.  It is imperative that a stable numerical approximation must preserve these symmetries, by mimicking the corresponding continuous energy estimate at the discrete level.  Otherwise numerical simulations on marginally resolved meshes can be polluted by numerical artefacts, in particular for long time simulations.  A good case in point is that important information in the medium propagated by scattered high frequency surface waves could be corrupted through numerical dissipation.

% through an analysis of the discrete operators' dispersion relation.

% The upwind operators are asymmetric and dissipative, which can potentially destroy these symmetries. 
% For example, the linear elastic wave equation with free-surface boundary conditions preserves the mechanical energy for all time.
% Otherwise numerical simulations on marginally resolved meshes can be polluted by numerical artefacts, in particular for longtime simulations. 
% A good case in point is that important information in the medium propagated by scattered high frequency surface waves could be corrupted through numerical dissipation. 

Our first objective is to carefully combine the upwind SBP operator pairs which have good dispersion relation properties so that we preserve the discrete anti-symmetric property and invariants of the underlying IBVP. 
We then achieve energy stability by imposing boundary conditions weakly with penalty terms, in a manner that leads to bounded discrete energy estimate. In particular, if the IBVP is energy conserving, the numerical approximation is also energy conserving. We derive a priori error estimates and prove the convergence of the numerical error.

The second objective of this study is to reveal the numerical dispersion properties of the discrete upwind SBP FD operators \cite{Mattsson2017}. By the analysis of the numerical dispersion relation  we show that not all the upwind SBP FD operators derived in \cite{Mattsson2017} are capable of suppressing poisonous spurious unresolved wave modes for the elastic wave equation. In particular, we show that odd-order upwind SBP FD operators also support spurious unresolved high frequency wave modes on marginally resolved meshes. Even-order upwind SBP FD operators (of order $2, 4, 6$) do not support spurious unresolved wave modes and have better numerical dispersion properties than traditional SBP FD operators and odd-order upwind SBP FD operators.   The numerical dispersion relation reaches near optimal properties for the 6th order accurate upwind SBP FD operator. Beyond the 6th order accurate upwind SBP FD operator, higher order accuracy does not improve the dispersion properties of the upwind SBP operators. Numerical experiments for the 3D elastic wave equation in complex geometries verify the analysis.

% We find that operators of increasing odd-order have poorer dispersion properties, and can support spurious wave modes. 
% Furthermore, high order even accuracy operators also suffer from poor dispersion properties. 
% We identify that orders $2,4,6$ do not support spurious wave modes in their numerical solution and can be used in place of traditional SBP operators to improve the efficiency of large scale wave propagation.
% For these upwind operators with better dispersion relation properties, we design upwind SBP numerical schemes for the 3D  elastic wave equation.
% It is imperative that a stable numerical approximation must preserve the symmetries that exist in the continuous problem by mimicking them at the discrete level. 

Numerical solutions are integrated in time using the fourth-order accurate low-storage Runge-Kutta method \cite{CarpenterKennedy1994}. 
The numerical method is implemented in WaveQLab \cite{DuruandDunham2016}, a petascale elastic wave solver. 
Simulations of elastic waves in heterogeneous media with free surface topography are presented, including the numerical simulation of community developed seismological benchmark problems. 
 Our results show that even-order upwind SBP FD operators are more robust and less prone to numerical dispersion on marginally resolved meshes when compared to traditional SBP operators \cite{BStrand1994,DuruandDunham2016}.
Consequently, more accurate numerical solutions on marginally refined meshes can be achieved with less computational effort by carefully using certain ($6$th order acurate) upwind SBP operators in a way that mimic the continuous problem. 

We have used the perfectly matched layer (PML) \cite{DuruKozdonKreiss2016} to enable efficient domain truncation and prevent artificial numerical reflections, from the computational boundaries, from contaminating the numerical simulations.  The effective absorption properties of the PML allows us to sufficiently limit the computational domain  with only a few grid points around the computational boundaries where the PML is active. Thus saving as much as $\% 97.9017$ of  the required computational resources for the 3D seismological benchmark problem \cite{Seismowine}.
Because of the asymmetric properties of the PML and the upwind SBP operators, a stable implementation of the PML for the 3D IBVP elastic wave equation  using the upwind SBP operators is a non-trivial task. The details of the numerical treatment of the PML using upwind SBP operators will be reported in a forthcoming paper.

For large scale numerical simulations, it is important that the parallel numerical software is efficient and scalable. We have performed  strong  scaling  tests, demonstrating nearly perfect scaling and   verifying  the  efficiency  of  our  parallel implementation  of  high  order  upwind  SBP  operators  for large scale elastic  wave  simulations  in  3D  complex geometries.

 The structure of the paper is as follows. 
In the next section, we present the elastic wave equation in general curvilinear coordinates and prove anti-symmetric properties that must be preserved. 
In section 3, we introduce general linear well-posed boundary conditions and derive continuous energy estimates. 
In section 4, we discretise in space, introduce upwind SBP operators, and approximate the elastic wave equation in space. 
Numerical boundary conditions are derived in section 5 as well as the derivation of  semi-discrete energy estimates, proving stability. 
In section 6, we provide an error analysis for our discretised schemes. 
We also provide a dispersion relation analysis, showing that even-order upwind SBP operators have dispersion relations which better mimic the continuous problem when compared to their odd-order and traditional counterparts. 
In section 7, we present numerical simulations verifying the accuracy and demonstrating the efficacy of the numerical method in complex geometry, with geologically constrained non-planar free-surface topography. 
We also provide scaling tests that indicate our code achieves near-perfect strong scaling. 
 In section 8, we summarise our findings and speculate on the directions for future work.

\section{Preliminary}
In this section we introduce  the elastic wave equation in general curvilinear coordinates and derive invariants and anti-symmetric properties that must be preserved.
\subsection{Physical model}

    Let $\Omega$ be a connected compact subset of $\R^3$ with a piecewise-smooth boundary $\Gamma$, $t>0$ be the time variable, $\bm{\sigma} \coloneqq (\sigma_{xx},\sigma_{yy},\sigma_{zz}\sigma_{xy},\sigma_{xz},\sigma_{yz})^T$ be the vector of stresses, and $\mathbf{v} \coloneqq (v_x,v_y,v_z)^T$ be the particle velocities.
    Throughout,  we assume that the velocities and stresses are bounded functions of space and time. 
    The first order time-dependent elastic wave equations in a source free, heterogeneous medium are
    \begin{align}\label{eq:elastic_sys}
        \begin{pmatrix}
        \rho \frac{\partial v_x}{\partial t }\\
        \rho \frac{\partial v_y}{\partial t }\\
        \rho \frac{\partial v_z}{\partial t } \vspace{0.2em} \\
        \mathbf{S} 
        \begin{pmatrix}
        \frac{\partial \sigma_{xx} }{\partial t }\\
        \frac{\partial \sigma_{yy} }{\partial t }\\
        \frac{\partial \sigma_{zz} }{\partial t }\\
        \frac{\partial \sigma_{xy} }{\partial t }\\
        \frac{\partial \sigma_{xz} }{\partial t }\\
        \frac{\partial \sigma_{yz} }{\partial t }\\
        \end{pmatrix}
        \end{pmatrix}
        = 
        \begin{pmatrix}
         \frac{\partial \sigma_{xx} }{\partial x } +  \frac{\partial \sigma_{xy} }{\partial y } +  \frac{\partial \sigma_{xz} }{\partial z }\\
         \frac{\partial \sigma_{xy} }{\partial x } +  \frac{\partial \sigma_{yy} }{\partial y } +  \frac{\partial \sigma_{yz} }{\partial z }\\
         \frac{\partial \sigma_{xz} }{\partial x } +  \frac{\partial \sigma_{yz} }{\partial y } +  \frac{\partial \sigma_{zz} }{\partial z }\\
         \frac{\partial v_{x} }{\partial x }\\
         \frac{\partial v_{y} }{\partial y }\\
         \frac{\partial v_{z} }{\partial z }\\
         \frac{\partial v_{x} }{\partial y } + \frac{\partial v_{y} }{\partial x } \\
         \frac{\partial v_{x} }{\partial z } + \frac{\partial v_{z} }{\partial x } \\
         \frac{\partial v_{y} }{\partial z } + \frac{\partial v_{z} }{\partial y } 
        \end{pmatrix},
    \end{align}
    where $\mathbf{S}= \mathbf{S}^T > 0$ is the compliance matrix given in Equation \eqref{eq:S matrix} and $\rho: \Omega \mapsto \R_+$ is the density of the medium. 
    The first three equations in \eqref{eq:elastic_sys} describe the conservation of momentum, and the latter six encode time derivative of Hookes law in three space dimensions. \\
    It is convenient to work with Equation \eqref{eq:elastic_sys} in its general basis form. Let $\bm{e} \coloneqq \{\bm{e}_{x}, \bm{e}_{y},\bm{e}_{z}\}$ be a basis for $\R^3$, then we may recast this PDE system into its conservative and non-conservative components by introducing the anti-symmetric form \cite{Duru_exhype_2_2019}
    \begin{align}\label{eq:anti_sys}
        \bm{P}^{-1} \pf{\bm{Q}}{t} = \nabla \cdot \bm{F} (\bm{Q} ) + \sum_{ \xi \in \{x,y,z \} } \bm{B}_{\xi} (\nabla \bm{Q}), 
    \end{align}
    where $\bm{Q} \coloneqq (\bm{v}, \bm{\sigma} )^T$, 
    \begin{align}\label{eq:S matrix}
        \bm{P} = \begin{pmatrix} \rho^{-1} \bm{1} & \bm{0} \\ \bm{0}^T & \bm{C} \end{pmatrix}, && \bm{S}^{-1} \coloneqq \bm{C} \coloneqq  \begin{pmatrix} 
        2 \mu + \lambda & \lambda & \lambda & 0 & 0 & 0 \\
         \lambda & 2 \mu + \lambda & \lambda & 0 & 0 & 0 \\
         \lambda & \lambda & 2 \mu + \lambda & 0 & 0 & 0 \\
        0 & 0 & 0 & \mu & 0 & 0 \\
        0 & 0 & 0 & 0 & \mu & 0 \\
        0 & 0 & 0 & 0 & 0 & \mu 
        \end{pmatrix}
        ,% && \rho^{-1} \coloneqq \frac{1}{\rho}
    \end{align}
    with $\mu, \lambda \in \R$ are independent Lam\'e parameters which describe an isotropic medium and 
    \begin{align}\label{eq:split_form}
         \bm{F}_{\xi } (\bm{Q} )  \coloneqq 
         \begin{pmatrix}
            e_{\xi x} \sigma_{xx} + e_{\xi y} \sigma_{xy} + e_{\xi z} \sigma_{xz}\\
            e_{\xi x} \sigma_{xy} + e_{\xi y} \sigma_{yy} + e_{\xi z} \sigma_{yz}\\
            e_{\xi x} \sigma_{xz} + e_{\xi y} \sigma_{yz} + e_{\xi z} \sigma_{zz}\\
            0 \\
            0 \\
            0 \\
            0 \\
            0 \\
            0 
         \end{pmatrix},
         &&
         \bm{B}_{\xi} (\nabla \bm{Q}) :=
         \begin{pmatrix}
            0 \\
            0 \\
            0 \\
            e_{\xi x} \pf{v_x}{\xi} \\
            e_{\xi y} \pf{v_y}{\xi}\\
            e_{\xi z} \pf{v_z}{\xi}\\
            e_{\xi y} \pf{v_x}{\xi} + e_{\xi x} \pf{v_y}{\xi} \\
            e_{\xi z} \pf{v_x}{\xi} + e_{\xi x} \pf{v_z}{\xi} \\
            e_{\xi z} \pf{v_y}{\xi} + e_{\xi y} \pf{v_z}{\xi} 
         \end{pmatrix},
    \end{align}
    with $\bm{F} \coloneqq (\bm{F}_{x}, \bm{F}_{y}, \bm{F}_{z})^T$ and the basis vectors $\bm{e}_{\xi} = \left(e_{\xi x}, e_{\xi y}, e_{\xi z}\right)^T$.  
    For instance in the standard Cartesian co-ordinate systems we have 
    \begin{align}\label{eq:canonical_basis}
        \bm{e}_{x} = (1,0,0)^T, &&  \bm{e}_{y} = (0,1,0)^T, && \bm{e}_{z} = (0,0,1)^T,
    \end{align}
    which recovers Equation \ref{eq:elastic_sys}. 
    % \begin{align}
    % \bm{F} (Q) = 
    % \begin{pmatrix}
    % \sigma_{xx} & \sigma_{xy} & \sigma_{xz} & 0 & 0 & 0 & 0 & 0 & 0 \\
    % \sigma_{xy} & \sigma_{yy} & \sigma_{yz} & 0 & 0 & 0 & 0 & 0 & 0 \\
    % \sigma_{xz} & \sigma_{yz} & \sigma_{zz} & 0 & 0 & 0 & 0 & 0 & 0 
    % \end{pmatrix}
    % \end{align}
    
    \begin{lemma}\label{lem:anti_sym_prop}
        Consider the anti-symmetric form given in Equation \eqref{eq:anti_sys}. 
        For any basis $\bm{e}$ that spans $\Omega$ we have
        \begin{align*}
            \left(\left(\frac{\partial \bm{Q}}{\partial \xi}\right)^T \bm{F}_{\xi}\left(\bm{Q}\right)- \bm{Q}^T\bm{B}_{\xi} (\nabla \bm{Q}) \right) = 0.
        \end{align*}
    \end{lemma}
    \begin{proof}
        Expanding the matrix multiplication and simplifying yields the result.
    \end{proof}
    We will find Lemma \ref{lem:anti_sym_prop} useful in proving stability in our discretisation scheme in complex geometries. 
    It will often be the case that our basis will be defined locally, through finding a set of functions whose partial derivatives evaluated at each point in $\Omega$ span $\R^3$. 
    For instance, consider the functions $\text{proj}_{\xi}: \Omega \mapsto \R$ for $\xi \in \{x,y,z\}$ given through
    \begin{align}\label{Eq:indicator}
        \text{proj}_{\xi}(x,y,z) \coloneqq \sum_{\eta \in \{x,y,z\}} \eta \1_{ \{ \xi \}} (\eta),  && \1_{B} (v) \coloneqq 
        \begin{cases}
            1 & \text{if } v \in B, \\
            0 & \text{otherwise.}
        \end{cases}
    \end{align}
    These functions form the (local) basis vectors
    \begin{align}
        \bm{e}_{\xi}|_{(x_0,y_0,z_0)} =  \left( \pf{}{x} \text{proj}_{\xi} , \pf{}{y} \text{proj}_{\xi}, \pf{}{z} \text{proj}_{\xi} \right) \Big|_{(x_0,y_0,z_0)},
    \end{align}
    so $\bm{e}_{\xi}$  are the contravariant basis defined in \eqref{eq:canonical_basis}. 
    % \begin{comment}
    % Let us introduce the constant coefficient matrices $\bm{A}_{\xi}$ for $\xi \in \{ x,y,z\} $ given by
    % \begin{align}\label{eq:A_xi}
    %     && \bm{A}_{\xi} \coloneqq \begin{pmatrix} \bm{0}_3 && \bm{a}_{\xi} \\ \bm{a}_{\xi}^T && \bm{0}_6 \end{pmatrix}, && \\
    %     \bm{a}_x \coloneqq 
    %     \begin{pmatrix}
    %     1 & 0 & 0 & 0 & 0 & 0 \\
    %     0 & 0 & 0 & 1 & 0 & 0 \\
    %     0 & 0 & 0 & 0 & 1 & 0 
    %     \end{pmatrix},
    %     &&
    %     \bm{a}_y \coloneqq 
    %     \begin{pmatrix}
    %     0 & 0 & 0 & 1 & 0 & 0 \\
    %     1 & 0 & 0 & 0 & 0 & 0 \\
    %     0 & 0 & 0 & 0 & 0 & 1 
    %     \end{pmatrix},
    %     &&
    %     \bm{a}_x \coloneqq 
    %     \begin{pmatrix}
    %     0 & 0 & 0 & 0 & 1 & 0 \\
    %     0 & 0 & 0 & 0 & 0 & 1 \\
    %     0 & 0 & 1 & 0 & 0 & 0 
    %     \end{pmatrix},
    % \end{align}
    % where $\bm{0}_n$ denotes the $n \times n$ zero matrix. 
    % Further define 
    % \begin{align*}
    %     \bm{P} = \begin{pmatrix} \rho^{-1} \bm{1} & \bm{0} \\ \bm{0}^T & \bm{C} \end{pmatrix},
    % \end{align*}
    % where $\bm{C} \coloneqq \bm{S}^{-1}$. 
    % Equation \ref{eq:elastic sys} can be recast as the linear system
    % \begin{align}
    %     \bm{P}^{-1} \pf{\bm{Q}}{t} = \sum_{\xi \in \{x,y,z \}  } \bm{A}_{\xi} \pf{\bm{Q}}{\xi}, 
    % \end{align}
    % with $\bm{Q} \coloneqq (\bm{v}, \bm{sigma} )^T$.
    % \end{comment}
    
\subsection{Curvilinear coordinates}
Assume $\Omega \subset \R^3$ to be sufficiently smooth such that it can be mapped to the unit cube $\Omegat \coloneqq [0,1]^3$. See also Figure \ref{fig:my_label}. 
If $\Omega$ is piece-wise smooth we can partition it into locally smooth sub-blocks and map each sub-block to the unit cube.  
For simplicity, we will only consider one sub-block here. 
Let $\Phi : \Omega \mapsto \Omegat$ be a diffeomorphism and adopt the notation
\begin{align*}
    \Phi (x,y,z) \coloneqq (q(x,y,z), r(x,y,z), s(x,y,z)),
\end{align*}
where ${\xi}: \Omega \mapsto \R$ for $\xi \in \{q,r,s \}$. 
Define for $S \subset {\Omega}$ 
\begin{align*}
	\Phi (S) \coloneqq \{ \Phi(s) \in \Omegat \ | \   s \in S  \}.
\end{align*}
Assume that $\Phi({\Omega}) = \Omegat$, and furthermore the boundary interaction 
\begin{align*}
	\Phi^{-1}(\{0\} \times [0,1]^2  ) =  \bigcup_{y ,z}  \left(\widehat{X}(y,z),y,z\right) , %\Phi^{-1}( \emptyset \times [0,1]^2
\end{align*}
for a smooth function $\widehat{X}:\R^2 \mapsto \R $. For example $\widehat{X}(y,z)$ could describe a complex free-surface topography.

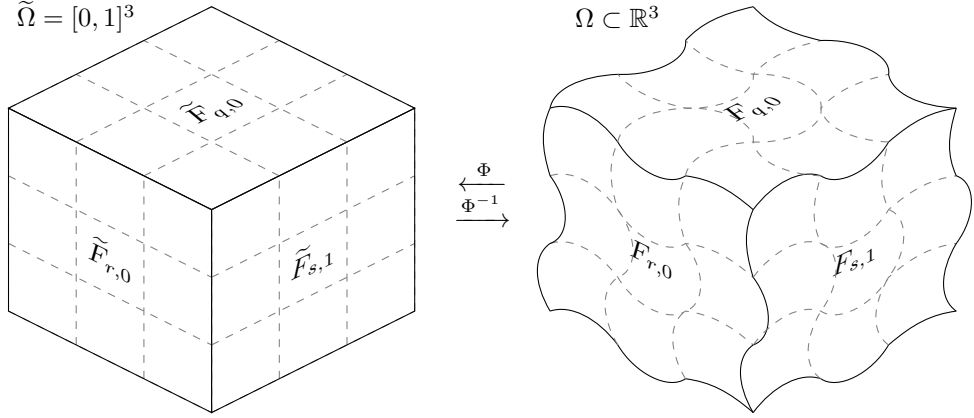
\begin{figure}[H]
    \centering
    \begin{tikzpicture}[every node/.style={minimum size=1cm},on grid, scale = 0.9]
    
%% original cube    
\begin{scope}[every node/.append style={yslant=-0.5},yslant=-0.5]
  \node at (0.5,2.5) {};
  \node at (1.5,2.5) {};
  \node at (2.5,2.5) {};
  \node at (0.5,1.5) {};
  \node at (1.5,1.5) {$\widetilde{F}_{r,0}$};
  \node at (2.5,1.5) {};
  \node at (0.5,0.5) {};
  \node at (1.5,0.5) {};
  \node at (2.5,0.5) {};
  \draw[dashed, gray] (0,0) grid (3,3);
  \draw (0,0) rectangle (3,3);
\end{scope}
\begin{scope}[every node/.append style={yslant=0.5},yslant=0.5]
  \node at (3.5,-0.5) {};
  \node at (4.5,-0.5) {};
  \node at (5.5,-0.5) {};
  \node at (3.5,-1.5) {};
  \node at (4.5,-1.5) {$\widetilde{F}_{s,1}$};
  \node at (5.5,-1.5) {};
  \node at (3.5,-2.5) {};
  \node at (4.5,-2.5) {};
  \node at (5.5,-2.5) {};
  \draw[dashed, gray] (3,-3) grid (6,0);
  \draw (3,-3) rectangle (6,0);
\end{scope}
\begin{scope}[every node/.append style={
    yslant=0.5,xslant=-1},yslant=0.5,xslant=-1
  ]
  \node at (3.5,2.5) {};
  \node at (3.5,1.5) {};
  \node at (3.5,0.5) {};
  \node at (4.5,2.5) {};
  \node at (4.5,1.5) {$\widetilde{F}_{q,0}$};
  \node at (4.5,0.5) {};
  \node at (5.5,2.5) {};
  \node at (5.5,1.5) {};
  \node at (5.5,0.5) {};
  \draw[dashed, gray] (3,0) grid (6,3);
  \draw (3,0) rectangle (6,3);
\end{scope}

\node at (1,4.35) {$\Omegat = [0,1]^3$};

\node at (7,2) {$\xleftarrow[]{ \ \Phi \  } $};
\node at (7,1.5) {$\xrightarrow[]{\Phi^{-1} }$};

%%curvey boi
\begin{scope}[shift = {(8,0)}]
\node at (1,4.35) {$\Omega \subset \R^3$};
\begin{scope}[every node/.append style={yslant=-0.5},yslant=-0.5]

    \coordinate (A1) at (0,0);
    \coordinate (A2) at (1,0);
    \coordinate (A3) at (2,0);
    \coordinate (A4) at (3,0);
    
    \coordinate (B1) at (0,1);
    \coordinate (B2) at (1,1);
    \coordinate (B3) at (2,1);
    \coordinate (B4) at (3,1);
    
    \coordinate (C1) at (0,2);
    \coordinate (C2) at (1,2);
    \coordinate (C3) at (2,2);
    \coordinate (C4) at (3,2);
    
    \coordinate (D1) at (0,3);
    \coordinate (D2) at (1,3);
    \coordinate (D3) at (2,3);
    \coordinate (D4) at (3,3);

  \node at (0.5,2.5) {};
  \node at (1.5,2.5) {};
  \node at (2.5,2.5) {};
  \node at (0.5,1.5) {};
  \node at (1.5,1.5) {$F_{r,0}$};
  \node at (2.5,1.5) {};
  \node at (0.5,0.5) {};
  \node at (1.5,0.5) {};
  \node at (2.5,0.5) {};
  %\draw (0,0) grid (3,3);
  %\draw [cyan, shift = {(0,0)}] plot [smooth, tension=2] coordinates { (0,0) (0.167,0.6) (0.33,0) (0.5,-0.6) (0.66,0) (3,0)};
  %\draw [cyan, shift = {(0,1)}] plot [smooth, tension=2] coordinates { (0,0) (0.33,0) (0.66,0) (3,0)};
  %\draw [cyan, shift = {(0,2)}] plot [smooth, tension=2] coordinates { (0,0) (0.33,0) (0.66,0) (3,0)};
  %\draw [cyan, shift = {(0,3)}] plot [smooth, tension=2] coordinates { (0,0) (0.33,0) (0.66,0) (3,0)};
  \draw[color=black] (A1) to [bend left=40] (A2) to [bend right=50] (A3) to [bend left=20] (A4) ; 
  \draw[dashed, gray] (B1) to [bend left=40] (B2) to [bend right=50] (B3) to [bend left=20] (B4) ; 
  \draw[dashed, gray] (C1) to [bend left=40] (C2) to [bend right=50] (C3) to [bend left=20] (C4) ; 
  \draw[color=black] (D1) to [bend left=40] (D2) to [bend right=50] (D3) to [bend left=20] (D4) ; 
  
  \draw[color=black] (A1) to [bend left=40] (B1) to [bend right=50] (C1) to [bend left=20] (D1) ; 
  \draw[dashed, gray] (A2) to [bend left=40] (B2) to [bend right=50] (C2) to [bend left=20] (D2) ; 
  \draw[dashed, gray] (A3) to [bend left=40] (B3) to [bend right=50] (C3) to [bend left=20] (D3) ; 
  \draw[color=black] (A4) to [bend left=40] (B4) to [bend right=50] (C4) to [bend left=20] (D4) ; 
\end{scope}
\begin{scope}[every node/.append style={yslant=0.5},yslant=0.5, shift = {(3,-3)}]
    \coordinate (A1) at (0,0) + (-3,3);
    \coordinate (A2) at (1,0);
    \coordinate (A3) at (2,0);
    \coordinate (A4) at (3,0);
    
    \coordinate (B1) at (0,1);
    \coordinate (B2) at (1,1);
    \coordinate (B3) at (2,1);
    \coordinate (B4) at (3,1);
    
    \coordinate (C1) at (0,2);
    \coordinate (C2) at (1,2);
    \coordinate (C3) at (2,2);
    \coordinate (C4) at (3,2);
    
    \coordinate (D1) at (0,3);
    \coordinate (D2) at (1,3);
    \coordinate (D3) at (2,3);
    \coordinate (D4) at (3,3);

  \node at (0.5,2.5) {};
  \node at (1.5,2.5) {};
  \node at (2.5,2.5) {};
  \node at (0.5,1.5) {};
  \node at (1.5,1.5) {$F_{s,1}$};
  \node at (2.5,1.5) {};
  \node at (0.5,0.5) {};
  \node at (1.5,0.5) {};
  \node at (2.5,0.5) {};
  %\draw (0,0) grid (3,3);
  
  \draw[color=black] (A1) to [bend left=40] (A2) to [bend right=50] (A3) to [bend left=20] (A4) ; 
  \draw[dashed, gray] (B1) to [bend left=40] (B2) to [bend right=50] (B3) to [bend left=20] (B4) ; 
  \draw[dashed, gray] (C1) to [bend left=40] (C2) to [bend right=50] (C3) to [bend left=20] (C4) ; 
  \draw[color=black] (D1) to [bend left=40] (D2) to [bend right=50] (D3) to [bend left=20] (D4) ; 
  
  %\draw[color=black] (A1) to [bend left=40] (B1) to [bend right=50] (C1) to [bend left=20] (D1) ; 
  \draw[dashed, gray] (A2) to [bend left=40] (B2) to [bend right=50] (C2) to [bend left=20] (D2) ; 
  \draw[dashed, gray] (A3) to [bend left=40] (B3) to [bend right=50] (C3) to [bend left=20] (D3) ; 
  \draw[color=black] (A4) to [bend left=40] (B4) to [bend right=50] (C4) to [bend left=20] (D4) ; 
\end{scope}
\begin{scope}[every node/.append style={
    yslant=0.5,xslant=-1},yslant=0.5,xslant=-1,
    shift = {(3,0)}
  ]
    \coordinate (A1) at (0,0);
    \coordinate (A2) at (1,0);
    \coordinate (A3) at (2,0);
    \coordinate (A4) at (3,0);
    
    \coordinate (B1) at (0,1);
    \coordinate (B2) at (1,1);
    \coordinate (B3) at (2,1);
    \coordinate (B4) at (3,1);
    
    \coordinate (C1) at (0,2);
    \coordinate (C2) at (1,2);
    \coordinate (C3) at (2,2);
    \coordinate (C4) at (3,2);
    
    \coordinate (D1) at (0,3);
    \coordinate (D2) at (1,3);
    \coordinate (D3) at (2,3);
    \coordinate (D4) at (3,3);

  \node at (0.5,2.5) {};
  \node at (1.5,2.5) {};
  \node at (2.5,2.5) {};
  \node at (0.5,1.5) {};
  \node at (1.5,1.5) {$F_{q,0}$};
  \node at (2.5,1.5) {};
  \node at (0.5,0.5) {};
  \node at (1.5,0.5) {};
  \node at (2.5,0.5) {};
  %\draw (0,0) grid (3,3);
  
  %\draw[color=black] (A1) to [bend left=40] (A2) to [bend right=50] (A3) to [bend left=20] (A4) ; 
  \draw[dashed, gray] (B1) to [bend left=40] (B2) to [bend right=50] (B3) to [bend left=20] (B4) ; 
  \draw[dashed, gray] (C1) to [bend left=40] (C2) to [bend right=50] (C3) to [bend left=20] (C4) ; 
  \draw[color=black] (D1) to [bend left=40] (D2) to [bend right=50] (D3) to [bend left=20] (D4) ; 
  
  %\draw[color=black] (A1) to [bend left=40] (B1) to [bend right=50] (C1) to [bend left=20] (D1) ; 
  \draw[dashed, gray] (A2) to [bend left=40] (B2) to [bend right=50] (C2) to [bend left=20] (D2) ; 
  \draw[dashed, gray] (A3) to [bend left=40] (B3) to [bend right=50] (C3) to [bend left=20] (D3) ; 
  \draw[color=black] (A4) to [bend left=40] (B4) to [bend right=50] (C4) to [bend left=20] (D4) ; 
\end{scope}
\end{scope}

\end{tikzpicture}
    \caption{Curvilinear coordinate transform and boundary faces of the computational space $\Omegat$ and modelling space $\Omega$. }
    \label{fig:my_label}
\end{figure}

The Jacobian determinant for $\Phi^{-1}$ can be written as
\begin{align*}
	J = x_q (y_r z_s - z_r y_s) - y_q(x_r z_s - z_r x_s) + z_q(x_r y_s - y_r x_s).
\end{align*}
Here $\xi_{\eta}$ is the partial derivative ${\partial \xi}/{\partial \eta}$ for $\xi,\eta \in \{x,y,z,q,r,s\}$.
%Through Cramers rule we can reconstruct the Jacobian matrix entries for $\Phi_i$ as
Similarly, 
\begin{align*}
	q_x = \frac{1}{J} (y_r z_s - z_r y_s) && r_x = \frac{1}{J} (z_q y_s - y_q z_s) && s_x = \frac{1}{J} (y_q z_r - z_q y_r), \\
	q_y = \frac{1}{J} (z_r x_s - x_r z_s) && r_y = \frac{1}{J} (x_q z_s - z_q x_s) && s_y = \frac{1}{J} (z_q x_r - x_q z_r), \\
	q_z = \frac{1}{J} (z_r y_s - y_r x_s) && r_z = \frac{1}{J} (y_q x_s - y_s x_q) && s_z = \frac{1}{J} (x_q y_r - x_r y_q) .
\end{align*}
The spatial derivatives in the transformed coordinates have: the conservative form
\begin{align}\label{eq:conservative}
	J \frac{\partial v}{\partial x} =  \frac{\partial}{\partial q} (Jq_{x} v) + \frac{\partial}{\partial r} (J r_x v) + \frac{\partial}{\partial s} (J s_x v) ,
\end{align}
and the non-conservative form 
 \begin{align}\label{eq:non-conservative}
\frac{\partial v}{\partial x} = q_x\frac{\partial v}{\partial q} + r_x\frac{\partial v}{\partial r}  + s_x\frac{\partial v}{\partial s}.
\end{align}
Although the conservative \eqref{eq:conservative} and non-conservative \eqref{eq:non-conservative} transformations of the  spatial derivatives are mathematically equivalent, when discretised they give different approximations.
Specifically,  in the discrete setting, the conservative form \eqref{eq:conservative} approximated with an SBP operator  preserves the divergence theorem. 

For each $(x_0,y_0,z_0) \in \Omega$ and function $\xi \in \{q,r,s\}$ let $\xi_\eta = \pf{\xi}{\eta}$ for $\xi \in \{q,r,s\}, \ \eta \in \{x,y,z\}$, and choose the co-ordinate basis vectors
\begin{align}
    \bm{e}_{\xi} = J(\xi_x, \xi_y, \xi_{z})\Big|_{(x_0,y_0,z_0)},
\end{align}
%%%%%%%%%%%%%%%%%%%%%%%%
%%%%%%%%%%%%%%%%%%%%%%%%
so Equation \eqref{eq:elastic_sys} is transformed to the curvilinear coordinates $(q, r, s)$, with the gradient operator redefined as $\nabla \coloneqq (\pf{}{q},\pf{}{r},\pf{}{s})$, to
\begin{align}\label{eq:transformedEQ}
    \Pt^{-1} \pf{}{t} \bm{Q} = \nabla \cdot \bm{F} (\bm{Q}) + \sum_{\xi \in \{q,r,s\} } \bm{B}_{\xi} (\nabla \bm{Q} ),
\end{align}
where $\Pt = J^{-1} \bm{P}$ and 
 \begin{align}
         \bm{F}_{\xi} (\bm{Q} )  \coloneqq 
         \begin{pmatrix}
            J( { \xi_x} \sigma_{xx} +  { \xi_y} \sigma_{xy} +  { \xi_z} \sigma_{xz})\\
            J( { \xi_x} \sigma_{xx} +  { \xi_y} \sigma_{xy} +  { \xi_z} \sigma_{xz})\\
            J( { \xi_x} \sigma_{xx} +  { \xi_y} \sigma_{xy} +  { \xi_z} \sigma_{xz})\\
            0 \\
            0 \\
            0 \\
            0 \\
            0 \\
            0 
         \end{pmatrix},
         &&
         \bm{B}_{\xi} (\nabla \bm{Q}) \coloneqq
         \begin{pmatrix}
            0 \\
            0 \\
            0 \\
            J  { \xi_x} \pf{v_x}{\xi} \\
            J  { \xi_y} \pf{v_y}{\xi}\\
            J  { \xi_z} \pf{v_z}{\xi}\\
            J( { \xi_y} \pf{v_x}{\xi} +  { \xi_x} \pf{v_y}{\xi}) \\
            J( { \xi_z} \pf{v_x}{\xi} +  { \xi_x} \pf{v_z}{\xi}) \\
            J( { \xi_z} \pf{v_y}{\xi} +  { \xi_y} \pf{v_z}{\xi}) 
         \end{pmatrix}
         .
    \end{align}
    %%%
    %%%
    Note that with the basis vectors  $\bm{e}_{\xi} = \left(J  { \xi_x}, J  { \xi_y}, J  { \xi_z}\right)^T$ we get the anti-symmetric form \eqref{eq:split_form}.
    %%%
    %%%
    
    \begin{remark}
    The coordinate transformation in Equation \eqref{eq:transformedEQ} is structure preserving, that is Lemma \ref{lem:anti_sym_prop} holds  and 
         we have
        \begin{align*}
            \left(\left(\frac{\partial \bm{Q}}{\partial \xi}\right)^T \bm{F}_{\xi}\left(\bm{Q}\right)- \bm{Q}^T\bm{B}_{\xi} (\nabla \bm{Q}) \right) = 0,
        \end{align*}
        for all $\xi \in \{ q, r, s \} $. This will be crucial in deriving high order accurate, structure preserving and provably energy stable scheme for the elastic wave equation in complex geometries.
    \end{remark}
    %For our use, the conservative form will be used on the stress terms of our PDE, whilst the non-conservative form will be used on the velocities. 
\section{Boundary Conditions}
%For $\xi \in \{ q,r,s\} $, and define 

% Let $\Xi \coloneqq \{Q,R,S \}$ and
% \begin{align*}
%     \widetilde{\Gamma} \coloneqq \bigcup_{\xi \in \Xi } \bigcup_{\eta \in \Xi \setminus \{ \xi \} } \{ \xi \in \{ 0,1\}, \nu \in [0,1] \ | \ (q,r,s) \in \Omegat \} ,
% \end{align*}
% it follows that $\Phi^{-1} (\widetilde{\Gamma} ) = \Gamma$. 
% The complex geometry in $\Omega$ is translated to $\Omegat$ through the warped inner product $\l \cdot , \cdot \r_{\Phi} \coloneqq \l \Phi^{-1}(\cdot) , \Phi^{-1}(\cdot) \r$. 
% It follows that the unit normal vectors at a point $\gamma$ on the boundary $\widetilde{\Gamma}$ are given through
% \begin{align*}
%     \bm{n}_{\gamma} = \frac{\Phi \circ \nabla \circ \Phi^{-1} (\gamma)}{\|\Phi \circ \nabla \circ \Phi^{-1} (\gamma) \|_2 }.
% \end{align*}
% \todo{this is not correct, need a cross product.}\\
% Try 
% \begin{align*}
%     \nabla \l e_{\xi} , \Phi (x,y,z) \r \coloneqq (\xi_x,\xi_y,\xi_z) . 
% \end{align*}

% Try this,\\
% Let $(q,r,s) \in \tilde{\Gamma}$ so at least one of $q,r,s$ is equal to $0$ or $1$, then we look at $\Phi^{-1}(q,r,s) \in \Gamma$.
% Let $\xi \in \{ q ,r, s \} $ so that $\xi = 0$ or $\xi = 1$.  \\

% We are happy now\\
In this section, we formulate linear well-posed  boundary conditions in complex geometries. 
As shown in Figure \ref{fig:my_label}, define the faces of the boundary $\Gamma$ as 
%%%
\begin{align*}
    F_{\xi, i } \coloneqq \{ (x_0,y_0,z_0) \in \Omega  \ | \ \xi(x_0,y_0,z_0) = i \}
\end{align*}
for $\xi \in \{q,r,s\}$ and $i \in \{0,1\}$. 
Each of the $F_{\xi, i }$ are Lebesgue-almost disjoint and form the boundary of $\Omega$, that is 
\begin{align}
    \Gamma = \bigcup_{\xi , i } F_{\xi, i }. 
\end{align}
Similarly the faces of the computational boundary $\widetilde{\Gamma}$ are given through 
\begin{align}\label{eq:boundary_faces}
    \widetilde{F}_{\xi, i} \coloneqq \Phi( F_{\xi, i}) = \{ (q,r,s)|_{(x_0,y_0,z_0)} \ | \ \xi(x_0,y_0,z_0) = i , (x_0,y_0,z_0) \in F_{\xi, i} \},
\end{align}
and the boundary $\widetilde{\Gamma}$ of $\Omegat$ is made from the level sets of the functions $q,r,s$,
\begin{align}
    \widetilde{\Gamma} =  \Phi (\Gamma) = \bigcup_{\xi \in \{q,r,s\} } \{ (q,r,s)|_{(x_0,y_0,z_0)} \ | \ \xi(x_0,y_0,z_0) \in \{ 0 , 1 \}, (x_0,y_0,z_0) \in \Gamma \} .
\end{align}

For a point $(x_0,y_0,z_0) \in F_{\xi, i }$, the unit normal vector to the surface $F_{\xi, i }$ is given by 
\begin{align}\label{eq:normal_vector}
    \bm{n} (x_0,y_0,z_0) =  \frac{1}{\sqrt{\xi_x^2+\xi_y^2+\xi_z^2 }} \begin{pmatrix} \xi_x \\ \xi_y \\ \xi_z \end{pmatrix} \Big|_{(x_0,y_0,z_0)},
\end{align}
where $\xi_x, \xi_y, \xi_z$ are the partial derivatives of $\xi$ with respect to $x,y,z$ respectively. 
% For a particular $\xi \in \{q,r,s \}$, the unit normal vector at a point $v = (x_0,y_0,z_0) \in \Gamma$ such that $\xi(x_0,y_0,z_0) = 0$ or $1$ is given by 
% \begin{align}
%     \bm{n} (x_0, y_0, z_0) = \frac{1}{\sqrt{\xi_x(v)^2+\xi_y(v)^2+\xi_z(v)^2 }} \begin{pmatrix} \xi_x(v) \\ \xi_y(v) \\ \xi_z(v) \end{pmatrix},
% \end{align}
% where $\xi_x, \xi_y, \xi_z$ are the partial derivatives of $\xi$ with respect to $x,y,z$ respectively. 

% Now see 
% \begin{align}
%     \bm{Q}^T \bm{F}_{\xi} (\bm{Q} ) = J \sqrt{\xi_x^2+\xi_y^2+\xi_z^2 }  \bm{v}^T \bm{T},
% \end{align}
% so the boundary terms can be given through
% \begin{align}
%     BTs (\bm{v},\bm{T}) \coloneqq \oint_{\Gamma} \bm{v}^T \bm{T} dS = \sum_{\substack{\xi \in \{q,r,s\} \\ i \in \{0, 1 \} } } \int_{\widetilde{F}_{\xi, i}} (-1)^{i+1} J \sqrt{\xi_x^2+\xi_y^2+\xi_z^2 }  \bm{v}^T \bm{T} \frac{dq dr ds}{d \xi}.
% \end{align}
For each normal vector $\bm{n}(x_0,y_0,z_0)$, we can form a locally spanning orthonormal basis with the vectors $\bm{m}(x_0,y_0,z_0)$ and $\bm{l}(x_0,y_0,z_0)$ given through a change of variable to the computational space
%%%%%
\begin{align*}
    \bm{m}\left(x_0,y_0,z_0\right) \coloneqq 
     \frac{\bm{m_0} - \l \bm{n} , \bm{m}_0 \r \bm{n} }{|\bm{m_0} - \l \bm{n} , \bm{m}_0 \r \bm{n}|} \Big|_{\left(x_0,y_0,z_0\right)}, 
    \quad 
    \bm{l} \left(x_0,y_0,z_0\right) \coloneqq  \bm{n} \times \bm{m}\Big|_{\left(x_0,y_0,z_0\right)},
\end{align*}
%%%%
where $\bm{m}_0$ is a vector not in the span of $\bm{n}$. 
For brevity, we often drop the evaluation point when this is clear from context.

Denote the  local impedances by $Z_\eta$ for $\eta \in \{ l, m,n \}$,
%%%%%%%%%%%%%%%%%%%
%%%%%%%%%%%%%%%%%%%
where $Z_n = \rho c_n$ is the p--wave impedance and $Z_m =\rho c_m$, $Z_l= \rho c_l$ are   the s--wave impedances. Here, $c_n, c_m, c_l$ are the corresponding effective wave speeds defined 
%%%%%%%%%%%%%%%%%%%
We consider specifically an isotropic medium  the effective wavespeeds  are given by
%%%%%%%%%%%%%%%%%%%%
%%%%%%%%%%%%%%%%%%%%
$
c_n = c_p, \quad c_m =c_l = c_s.
$
%%%%%%%%%%%%%%%%%%%
%%%%%%%%%%%%%%%%%%%

On the boundary surface, we extract the particle velocity vector and the traction vector, and the local rotation matrix  
%%%%%%%%%%%%%%%%%%%
%%%%%%%%%%%%%%%%%%%
\begin{align}\label{eq:velocity_traction_rotation}
\mathbf{v} = \begin{pmatrix}
v_{x}\\
v_{y}\\
v_{z}
\end{pmatrix},
%%%%%%%%%%%%%%%%%%%
 \quad
%%%%%%%%%%%%%%%%%%%
\mathbf{T} = \begin{pmatrix}
T_{x}\\
T_{y}\\
T_{z}
\end{pmatrix} = \begin{pmatrix}
 \sigma_{xx}&\sigma_{xy} & \sigma_{xz} \\
 \sigma_{xy}&\sigma_{yy} & \sigma_{yz} \\
 \sigma_{xz}&\sigma_{yz} & \sigma_{zz} \\
\end{pmatrix}\begin{pmatrix}
 n_x   \\
n_y\\
n_z 
\end{pmatrix}, \quad \mathbf{R} = \begin{pmatrix}
 \bm{n}^T  \\
\bm{m}^T\\
\bm{l}^T 
\end{pmatrix},
\end{align}
%%%%%%%%%%%%%%%%%%%
%%%%%%%%%%%%%%%%%%%
where $\det(\mathbf{R}) \ne 0$ and $ \mathbf{R}^{-1} = \mathbf{R}^T$. 
%%%%%%%%%%%%%%%%%%%
%%%%%%%%%%%%%%%%%%%

%%%%%%%%%%%%%%%%%%%
%%%%%%%%%%%%%%%%%%%
Next, rotate the particle velocity and traction vectors into the local orthonormal basis, $\mathbf{l}$ ,  $\mathbf{m}$  and  $\mathbf{n}$, having 
%%%%%%%%%%%%%%%%%%%
%%%%%%%%%%%%%%%%%%%
\begin{align}\label{eq:local_velocity_tractions}
v_\eta  = \left(\mathbf{R}\mathbf{v}\right)_\eta , \quad T_\eta  = \left(\mathbf{R}\mathbf{T}\right)_\eta , \quad \eta \in \{ l, m, n \}.
\end{align}
%%%%%%%%%%%%%%%%%%%
%%%%%%%%%%%%%%%%%%%
Plane p--waves and plane s--waves propagating along the normal vector $\mathbf{n}$ on the boundary are given by
%%%%%%%%%%%%%%%%%%%
%%%%%%%%%%%%%%%%%%%
 \begin{align}\label{eq:characteristics}
{q}_\eta  = \frac{1}{2}\left({Z}_\eta {v}_\eta  + {T}_\eta \right), \quad {p}_\eta  = \frac{1}{2}\left({Z}_\eta {v}_\eta  - {T}_\eta \right), \quad Z_\eta  > 0.
\end{align}
%%%%%%%%%%%%%%%%%%%
%%%%%%%%%%%%%%%%%%%
%Here, the characteristics defined in \eqref{eq:characteristics} are . 
%%%%%%%%%%%%%%%%%%%

%%%%%%%%%%%%%%%%%%%
%At the boundary $\xi = 1$ $(\xi = 0)$, if $Z_\eta  > 0$ then ${q}_\eta $ (${p}_\eta $)  are the characteristics  going into the domain  and  ${p_\eta }$ (${q_\eta }$) the  characteristics going out of the domain.
%%%%%%%%%%%%%%%%%%%
%The number of boundary conditions must correspond to the number of characteristics going into the domain, see ~\cite{DuruandDunham2016, GustafssonKreissOliger1995}.
At the boundary faces ${F}_{\xi, i}$ defined in \eqref{eq:boundary_faces} we consider the linear boundary conditions, 
%%%%%%%%%%%%%%%%%%%
%%%%%%%%%%%%%%%%%%%
{
\begin{equation}\label{eq:BC_General2}
\begin{split}
& \frac{Z_{\eta}}{2}\left({1-\gamma_\eta }\right){v}_\eta  -\frac{1+\gamma_\eta }{2} {T}_\eta  = 0,  \quad (x, y, z) \in {F}_{\xi, 0}, \\
& \frac{Z_{\eta}}{2} \left({1-\gamma_\eta }\right){v}_\eta  + \frac{1+\gamma_\eta }{2}{T}_\eta  = 0,  \quad (x, y, z) \in {F}_{\xi, 1}.
 \end{split}
\end{equation}
}
%%%%%%%%%%%%%%%%%%%%%%%
%%%%%%%%%%%%%%%%%%%%%%%
%%%%%%%%%%%%%%%%%%%
Here  $\gamma_\eta $ are real parameters with $ 0 \le |\gamma_\eta |\le 1$. 
%
%The boundary conditions \eqref{eq:BC_General2} specify the ingoing characteristics on the boundary in terms of the outgoing characteristics. 
%
%In an elastic medium, we have $Z_\eta  > 0$ for all $\eta = l, m, n$, and there are three boundary conditions at each boundaries $\xi = 1$, ($\xi = 0$).
%%%%%%%%%%%%%%%%%%%
%%%%%%%%%%%%%%%%%%%
%%%%%%%%%%%%%%%%%%%
 The boundary condition \eqref{eq:BC_General2}, can describe several physical situations.
 We have a {free-surface boundary condition}  if $\gamma_\eta  = 1$, an {absorbing boundary condition}  if $\gamma_\eta  = 0$ and a {clamped boundary condition}  if $\gamma_\eta  = -1$.
 %Note that
%%%%%%%%%%%%%%%%%%%
%%%%%%%%%%%%%%%%%%%
Note that the boundary condition \eqref{eq:BC_General2} satisfy the inequalities
  \begin{align}\label{eq:simplify_3}
&  v_\eta T_\eta  > 0, \quad \forall |\gamma_\eta | < 1, \quad \text{and} \quad  v_\eta T_\eta  = 0, \quad \forall |\gamma_\eta | = 1, \quad \xi \equiv 0, \nonumber
 \\
 & v_\eta T_\eta   < 0, \quad \forall |\gamma_\eta | < 1, \quad \text{and} \quad  v_\eta T_\eta  = 0, \quad \forall |\gamma_\eta | = 1, \quad \xi \equiv 1.
  \end{align}
%%%%%%%%%%%%%%%%%%%
 
We introduce the boundary terms which are surface integrals encoding the work done by the traction force on the boundary  
%%%%%%%%%%%%%%%%%%%
\begin{align}\label{eq:boundaryterm_101}
    BTs (\bm{v},\bm{T}) \coloneqq \oint_{\Gamma} \bm{v}^T \bm{T} dS = \sum_{\substack{\xi \in \{q,r,s\} \\ i \in \{0, 1 \} } } (-1)^{i+1} \int_0^1 \int_0^1 J \sqrt{\xi_x^2+\xi_y^2+\xi_z^2 }  \bm{v}^T \bm{T} \frac{dq dr ds}{d \xi}.
\end{align}
 
%%%%%%%%%%%%%%%%%
\begin{lemma}\label{Lem:BTs}
Consider the well-posed boundary conditions \eqref{eq:BC_General2} with $|\gamma_\eta | \le 1$. The boundary term $\mathrm{BTs}$ defined in \eqref{eq:boundaryterm_101}   is negative semi-definite, $\mathrm{BTs} \le 0$, for all $Z_\eta > 0$.
\end{lemma}
%%%%%%%%%%%%%%%%%
\begin{proof}
 %%%%%%%%%%%%%%%%%%%
 Consider the boundary term $ \mathrm{BTs}\left(v , T \right)$ defined in \eqref{eq:boundaryterm_101}.
With $\mathbf{v}^T\mathbf{T}  = \left(\mathbf{R}\mathbf{v}\right)^T\left(\mathbf{R}\mathbf{T}\right) = \sum_{\eta \in \{ l,m,n \} }{v_\eta T_\eta }$,  we have
  \begin{align}\label{eq:boundaryterm_2}
%%%%%%%%%%%%%%%%%
%%%%%%%%%%%%%%%%%
 \mathrm{BTs}\left(\bm{v} , \bm{T} \right)  %&\sum_{\xi = q, r, s}\left(\int_{\widetilde{\Gamma}} \left(J\sqrt{\xi_x^2 + \xi_y^2 + \xi_z^2}\right)\mathbf{v}^T\mathbf{T} \Big|_{\xi = 1}\frac{dqdrds}{d\xi} - \int_{\widetilde{\Gamma}} \left(J\sqrt{\xi_x^2 + \xi_y^2 + \xi_z^2}\right)\mathbf{v}^T\mathbf{T} \Big|_{\xi = 0} \frac{dqdrds}{d\xi} \right)\\
  =&\int_0^1 \int_0^1 \left(\left(J\sqrt{\xi_x^2 + \xi_y^2 + \xi_z^2}\right) \sum_{\eta \in \{ l,m,n \} } v_\eta T_\eta \right) \bigg|_{\xi = 1}\frac{dqdrds}{d\xi} \nonumber \\
  -& \sum_{\xi \in  \{ q, r, s \} } \int_0^1 \int_0^1 \left(\left(J\sqrt{\xi_x^2 + \xi_y^2 + \xi_z^2}\right)\sum_{\eta \in \{ l,m,n \} } v_\eta T_\eta \right)\bigg|_{\xi = 0} \frac{dqdrds}{d\xi} .
 %
 %\mathrm{BTs}\left(v , T \right) =\sum_{\xi = q, r, s}\int_{\widetilde{\Gamma}}\sum_{\eta = l,m,n}\left(\left(\sqrt{\xi_x^2 + \xi_y^2 + \xi_z^2}J{v_\eta T_\eta }\right)\Big|_{\xi = 1} -\left(\sqrt{\xi_x^2 + \xi_y^2 + \xi_z^2}J{v_\eta T_\eta }\right)\Big|_{\xi = 0}\right)\frac{dqdrds}{d\xi}.
%%%%%%%%%%%%%%%%%%
%%%%%%%%%%%%%%%%%%
  \end{align}
%%%%%%%%%%%%%%%%%
 Finally, the identity  \eqref{eq:simplify_3}  completes the proof of the lemma. 
%%%%%%%%%%%%%%%%%
% %%%%%%%%%%%%%%%%
\end{proof}
%%%%%%%%%%%%%%%%%%

In the continuous setting, we can show that our PDE has finite energy controlled by the boundary terms, $\mathrm{BTs}$. 
To begin, for real functions we introduce the  $L^2$ inner product,
%%%%%%%%%%%%%%%%%%%
%%%%%%%%%%%%%%%%%%%
%%%%%%%%%%%%%%%%%%%
\begin{align}\label{eq:weighted_scalar_product}
\l \mathbf{Q}, \mathbf{F} \r \coloneqq \int_{\Omega}{\left(\mathbf{Q}^T\mathbf{F}\right) dxdydz},
\end{align}
%%%%%%%%%%%%%%%%%%%
%%%%%%%%%%%%%%%%%%%
and the corresponding energy-norm
\begin{align}\label{eq:physical_energy}
\|\mathbf{Q}\left(\cdot, \cdot, \cdot, t\right)\|_P^2 = \l\mathbf{Q}, \frac{1}{2} \mathbf{P}^{-1} \mathbf{Q}\r = \int_{\Omega}\left(\sum_{\eta \in \{x, y, z\} } \frac{\rho}{2} v_\eta^2  + \frac{1}{2}\boldsymbol{\sigma}^T{S}\boldsymbol{\sigma}\right)dxdydz. 
\end{align}
%%%%%%%%%%%%%%%%%%%
%%%%%%%%%%%%%%%%%%%
The weighted $L^2$-norm $\|\mathbf{Q}\left(\cdot, \cdot, \cdot, t\right)\|_P^2$ is the mechanical energy, which is the sum of the kinetic energy and the strain energy.
%As we assume that our interaction surface can be modelled by a smooth boundary, we may readily calculate the boundary terms through Equation **. 
\begin{theorem}\label{theo:energy_estimate_cont}
    The transformed elastic wave equation \eqref{eq:transformedEQ} in curvilinear coordinates subject to the boundary conditions  \eqref{eq:BC_General2}
    satisfies the energy equation
    \begin{align*}
        \frac{d}{dt} \|\mathbf{Q}\left(\cdot, \cdot, \cdot, t\right)\|_P^2 = BTs \le 0.
    \end{align*}
    \begin{proof}
    Consider
    \begin{align}
        \frac{d}{dt} \|\mathbf{Q}\left(\cdot, \cdot, \cdot, t\right)\|_P^2 = \l \bm{Q} , P^{-1} \pf{}{t}  \bm{Q} \r  = \l \bm{Q}, \nabla \cdot \bm{F} (\bm{Q}) + \sum_{\xi\in \{q,r,s \} } \bm{B}_{\xi } (\nabla \bm{Q} ) \r.
    \end{align}
    Expanding the right hand side and applying integration by parts yields
{\small
    \begin{align}
       \sum_{\xi \in \{q,r,s \}} \l \bm{Q}, \pf{}{\xi} \bm{F}_{\xi} (\bm{Q}) \r  + \l \bm{Q}, \bm{B}_{\xi} (\nabla \bm{Q} ) \r &= \sum_{\xi \in \{q,r,s \} }  \left( \l\bm{Q},  \bm{B}_{\xi} (\nabla \bm{Q} ) \r- \l \pf{}{\xi} \bm{Q}, \bm{F}_{\xi} (\bm{Q}) \r \right) \\
       &\ + BTs(\bm{v}, \bm{T} ),
    \end{align}
    }
    which from Lemma \ref{lem:anti_sym_prop} and \ref{Lem:BTs} gives the result. 
    \end{proof}
\end{theorem}
In the next section, we will derive a numerical approximation of the transformed elastic wave equation \eqref{eq:transformedEQ} in curvilinear coordinates subject to the boundary conditions  \eqref{eq:BC_General2}. We will approximate spatial derivatives using upwind SBP operators and imposed boundary conditions using penalties. To guarantee numerical stability we will prove numerical results analogous to Theorem \ref{theo:energy_estimate_cont}. 
\section{Discretisation}\label{sec:spatial_approximation}
%%%%
In this section, stable discrete numerical approximation of the transformed elastic wave equation \eqref{eq:transformedEQ} in curvilinear coordinates subject to the boundary conditions  \eqref{eq:BC_General2} is derived. 
We use upwind SBP operators \cite{Mattsson2017} to discretise the spatial derivatives, and keep the time variable continuous. 
The upwind SBP operators come in pairs, the forward difference operator $D_{+}$ and the backward difference operator $D_{-}$. 
Our ultimate goal is to carefully combine this pair of upwind SBP operators to preserve, in the discrete setting, the anti-symmetric property given in Lemma \ref{lem:anti_sym_prop}, and derive a conservative scheme. 
Boundary conditions are then  enforced weakly using penalties in a manner that leads to energy stability.

\subsection{Discrete spatial derivative}
%%%%
    We will use a reference mesh that is uniform across each of the axes to discretise the reference computational cube $\Omegat = [0,1]^3$. 
    For each $\xi \in \{q,r,s\}$, consider the uniform discretisation of the unit interval $\xi \in [0, 1]$ 
    \begin{align}
        \xi_{i} \coloneqq i/n_{\xi} && i \in \{0,\dots, n_{\xi} \}, 
    \end{align}
    where $n_{\xi} + 1$ is the number of grid-points on the $\xi$-axis. 
    
    We will use upwind SBP operators introduced in \cite{Mattsson2017} to approximate the spatial derivatives, $\pf{}{\xi}$. By combining the operators $D_{-},D_{+}$ we can respect the integration by parts formula:
    \begin{align}\label{eq:IBP_cts}
		 \int_0^1 \pf{}{\xi} (f) g d \xi + \int_0^1 f \pf{}{\xi} (g) d \xi = f(1)g(1) - f(0)g(0).
	\end{align} 
	This will be critical in deriving a stable and conservative numerical approximation of the elastic wave equation in complex geometries.

	For each $\xi \in \{q,r,s\}$ define $H_{\xi} \coloneqq \text{diag}\left(h_0^{(\xi)} , \dots, h_{n_{\xi}}^{(\xi)}\right)$, with $h_j^{(\xi)} >0$ for all $j \in \{ 0, 1, \dots, n_\xi \}$.
	We mimic the integration by parts property through finding the dual pairing of linear operators $D_{+\xi},D_{-\xi} : \R^{n_{\xi} +1 } \mapsto \R^{n_{\xi} +1 }$ so that
	\begin{align}\label{eq:upw_SBP}
        (D_{+\xi} \bm{f} )^T H_{\xi} \bm{g} +  \bm{f}^T H_{\xi} (D_{-\xi }\bm{g}) = f(\xi_{n_{\xi}})g(\xi_{n_{\xi}}) - f(\xi_0)g(\xi_0),
    \end{align}
    for vectors $\bm{f} = (f(\xi_0), \dots , f(\xi_{n_{\xi}}) )^T$, $\bm{g} = (g(\xi_0), \dots , g(\xi_{n_{\xi}}) )^T$ sampled from weakly differentiable functions of the $\xi$ variable. 
    Furthermore, the matrix $D_{+\xi} + D_{+\xi}^T$ (or $D_{-\xi} + D_{-\xi}^T$) is negative semi-definite to introduce efficient numerical suppression of unresolved high frequency wave modes, for more details see \cite{Mattsson2017,DovgilovichSofronov2015}.
    
    We make the discussion more formal.
    \begin{definition}
		Let $D_{-\xi}$, $D_{+\xi} : \R^{n_\xi} \mapsto \R^{n_\xi}$ be linear operators that solve Equation \ref{eq:upw_SBP} for a diagonal weight matrix $H_\xi \in \R^{n_\xi \times n_\xi}$. 
		If the matrix $S_{+} \coloneqq {D_{+\xi} + D_{+\xi}^T}$ or $S_{-} \coloneqq {D_{-\xi} + D_{-\xi}^T}$ is also negative semi-definite, then the 3-tuple $(H_\xi,D_{-\xi},D_{+\xi})$ is called an upwind diagonal-norm dual-pair SBP operator.  
	\end{definition}
	%Given upwind SBP operators 
    %\begin{align}
    %    (H_{q}, D_{q,+},D_{q,-}), && (H_{r}, D_{r,+},D_{r,-}), && (H_{s},D_{s,+},D_{s,-}),
    %\end{align}
    We call $(H_{\xi},D_{-\xi},D_{+\xi})$ an upwind diagonal-norm dual-pair SBP operator of order $m$ if the accuracy conditions
	\begin{align} \label{eq:acc}
		    D_{\eta\xi}( \bm{\xi}^i) = i \bm{\xi}^{i-1}
	\end{align}
	are satisfied for all $i \in \{0,\dots,m\}$ and $\eta \in \{-,+\}$ where $\bm{\xi}^i \coloneqq (\xi_0^i, \dots, \xi_{n_\xi}^i)^T$. 
	
	Throughout this study, we will use diagonal diagonal-norm dual-pair SBP operators, and they will be simply referred to as upwind SBP operators. 
	
	Similar to traditional SBP operators, upwind SBP operators have higher accuracy in the interior, away from the boundaries. The accuracy of the operators is lowered close to the boundaries where special boundary closures are used. Upwind SBP operators with even-order $\left(2p\right)$-th accuracy in the interior are closed with $p$-th order accurate stencils close to boundaries. Upwind SBP operators with odd-order $\left(2p+1\right)$-th accuracy in the interior are closed with $p$-th order accurate stencils close to boundaries. These operators can yield $\left(p+1\right)$-th global order of accuracy.

	The 1D SBP operators can be extended to higher space dimensions using tensor products $\otimes$. Let $f(q,r,s)$ denote a 3D scalar function, and $f_{ijk} \coloneqq {f}(q_i,r_j, s_k)$ denote the corresponding 3D grid function.
	The  3D scalar grid function $f_{ijk}$ is rearranged row-wise as a vector $\bm{f}$ of length $n_qn_rn_s$. For $\xi \in \{q,r,s\}$ and $\eta \in \{-,+\}$ define:
    \begin{align}
        \bm{D}_{\eta\xi} \coloneqq  \bigotimes_{k \in \{q,r,s\} } ( \1_{k = \xi} D_{\eta k} + \1_{k \neq \xi} I_{n_k}), && \bm{H} \coloneqq  \bigotimes_{k \in \{q,r,s\} } H_{k},
    \end{align}
    where $I_{n_{\xi}}$ is the identity matrix of size $n_{\xi} \times n_{\xi}$, and we take $\1_{ k = \xi } \coloneqq \1_{ \{ \xi \} } (k)$ and $\1_{ k \neq \xi } \coloneqq 1 -\1_{ k = \xi }$. 
    So $\bm{D}_{\xi\eta}$  approximates the partial derivative operator in the $\xi$ direction.
    An inner product on $\R^{n_{q} + 1} \times \R^{n_{r} + 1} \times \R^{n_{s} + 1}$ is induced by $\bm{H}$ through 
     \begin{align}
        \l \bm{g} , \bm{f} \r_{\bm{H}} \coloneqq \bm{g}^T \bm{H} \bm{f} = \sum_{i=0}^{n_q} \sum_{j=0}^{n_r} \sum_{k=0}^{n_s}f_{ijk}g_{ijk} h_i^{(q)} h_j^{(r)} h_k^{(s)}, \end{align}
        where  the corresponding discrete energy-norm is given by
        %%%%%%%%%%%%%%%%%%%
     \begin{align}\label{eq:physical_energy_discrete}
         \|\mathbf{Q}\left(\cdot, \cdot, \cdot, t\right)\|_{HP}^2 = \l\mathbf{Q}, \frac{1}{2} \mathbf{P}^{-1} \mathbf{Q}\r_{\bm{H}}. 
     \end{align}
%%%%%%%%%%%%%%%%%%%
    Further, we have the multi-dimensional SBP property 
    \begin{align}\label{eq:multiDSBP}
        \sum_{\xi \in \{q,r,s\} } \left( \l \bm{D}_{-\xi} (\bm{f}), \bm{g}  \r_{\bm{H}} + \l \bm{f} , \bm{D}_{+\xi} (\bm{g}) \r_{\bm{H}} \right) =  \sum_{\xi \in \{q,r,s\} }S_{\xi}\left(\bm{f}\bm{g}\right), %\oint_{\widetilde{\Gamma} } fg d \mu_n,
    \end{align}
    where $S_{\xi}\left(\bm{f}, \bm{g}\right)$ in the right hand side is the surface cubature, defined by
    \begin{align}\label{eq:bt_q}
    S_{q}\left(\bm{f}\bm{g}\right) = \sum_{i \in \{ 0, n_q \} } \left(-1\right)^{q_i+1}\sum_{j=0}^{n_r} \sum_{k=0}^{n_s}f_{ijk}g_{ijk}  h_j^{(r)} h_k^{(s)},
	\end{align}
	%%%%
	\begin{align}\label{eq:bt_r}
    S_{r}\left(\bm{f} \bm{g}\right) = \sum_{j \in \{ 0, n_r \} } \left(-1\right)^{r_j+1}\sum_{i=0}^{n_q} \sum_{k=0}^{n_s}f_{ijk}g_{ijk}  h_i^{(q)} h_k^{(s)},
	\end{align}
	%%%%%
	%%%%
	\begin{align}\label{eq:bt_s}
    S_{s}\left(\bm{f} \bm{g}\right) = \sum_{k \in \{ 0, n_s \} } \left(-1\right)^{s_k+1}\sum_{i=0}^{n_q} \sum_{j=0}^{n_r}f_{ijk}g_{ijk}  h_i^{(q)} h_j^{(r)} .
	\end{align}
	%%%%%
	Note that $\xi_0 = 0$ and $\xi_{n_\xi} = 1$, for all $\xi \in \{ q, r, s \}$.

    \subsection{Numerical approximation in space}
    Consider the transformed elastic wave equation \eqref{eq:transformedEQ} in curvilinear coordinates, and approximate the spatial operators using the upwind SBP operators. Note that every 3D scalar grid function is rearranged row-wise as a vector of length $n_qn_rn_s$. 
    Therefore the unknown vector field $\mathbf{Q}$ is a vector of length $9 n_qn_rn_s$.
    
    The semi-discrete approximation reads 
    \begin{align}\label{eq:gen_hyp_transformed_discrete}
    \widetilde{\mathbf{P}}^{-1} \frac{d }{d t} {\mathbf{Q}} = \grad_{D_{-}}\bullet {\mathbf{F} \left({\mathbf{Q}} \right)} + \sum_{\xi \in \{q, r, s\} }\mathbf{B}_\xi\left(\grad_{D_{+}}{\mathbf{Q}}\right),
    \end{align}
where the discrete operator  $\grad_{D_{\eta}} = \left(\mathbf{D}_{\eta q}, \mathbf{D}_{\eta r}, \mathbf{D}_{\eta s}\right)^T$, with $\eta \in \{+, -\}$, is analogous to the continuous gradient operator $\grad = \left(\partial/\partial q, \partial/\partial r, \partial/\partial s\right)^T$. 
In  $\grad_{D_{\eta}}$ we have replaced the continuous derivative operators in $\grad$ with their discrete counterparts. 

\begin{remark}
The backward difference operator $D_{-}$ is used to approximate the spatial derivative for the conservative flux term, whilst the forward difference operator $D_{+}$ is an approximant for the non-conservative product term. 
This combination of upwind operators and the specific choice of the anti-symmetric form \eqref{eq:transformedEQ} is critical to deriving a conservative and energy stable scheme for the elastic wave equation in complex geometries.
\end{remark}
%%%%%%%%%%%%%%%%%%%
Note that we have not imposed boundary conditions yet. 
Numerical boundary treatment will be discussed in the next subsection. 
%%%%%%%%%%%%%%%%%%%

We will now prove the discrete equivalence of Lemma \ref{lem:anti_sym_prop}.

 \begin{lemma}\label{lem:anti_sym_prop_disc}
        Consider the semi-discrete approximation given in Equation \ref{eq:gen_hyp_transformed_discrete}. We have the discrete anti-symmetric form
       {\small
        \begin{align*}
         \left(\left(\left(I_9\otimes\bm{D}_{+\xi}\right)\bm{Q}\right)^T \bm{F}_{\xi}\left(\bm{Q}\right)
            - \bm{Q}^T\mathbf{B}_\xi\left(\grad_{D_{+}}\bar{\mathbf{Q}}\right)\right)= 0.
        \end{align*}
        }
    \end{lemma}
    \begin{proof}
        As before, expanding the matrix multiplication and simplifying yields the result.
    \end{proof}
    
    Further, for a 3D scalar field $f_{ijk} = f(x_i, y_j, z_k)$ we also introduce the surface cubature
\begin{align}\label{eq:surface_cubature_1D}
    &\mathbb{I}_{q_i}\left(\mathbf{f}\right) = \sum_{j=0}^{n_r} \sum_{k=0}^{n_s} \left(J_{ijk}\sqrt{q_{xijk}^2+ q_{yijk}^2 + q_{zijk}^2} f_{ijk}\right)  h_j^{(r)} h_k^{(s)}, \\
    & \mathbb{I}_{r_j}\left(\mathbf{f}\right) = \sum_{i=0}^{n_q} \sum_{k=0}^{n_s}\left(J_{ijk}\sqrt{r_{xijk}^2+ r_{yijk}^2 + r_{zijk}^2} f_{ijk}\right)   h_i^{(q)} h_k^{(s)}, \\
    &\mathbb{I}_{s_k}\left(\mathbf{f}\right) = \sum_{i=0}^{n_q} \sum_{j=0}^{n_r}\left(J_{ijk}\sqrt{s_{xijk}^2+ s_{yijk}^2 + s_{zijk}^2} f_{ijk}\right)   h_i^{(q)} h_j^{(r)}
	\end{align}
	%%%%
	and
	\begin{align}\label{eq:surface_cubature_3D}
   \mathbb{I}\left(\mathbf{f}\right) = \sum_{\xi \in \{ q, r, s \} }\sum_{i \in \{ 0, n_\xi \} }\left(-1\right)^{\xi_i} \mathbb{I}_{\xi_i}\left(\mathbf{f}\right).
	\end{align}
	Therefore we have
	\begin{align}\label{eq:disc_boundary_term}
	    \mathbb{I}\left(\bm{v}^T\bm{T}\right) = \sum_{\xi \in \{q,r,s\} }S_{\xi}\left(\bm{J}\sqrt{\bm{\xi}_x^2+\bm{\xi}_y^2+\bm{\xi}_z^2 },  \bm{v}^T \bm{T}\right),
	\end{align}
	where the surface cubature $S_{\xi}$ is defined in \eqref{eq:bt_q}--\eqref{eq:bt_s}.
	Here $S_{\xi}$ approximates integrals on faces of $\Omegat$, whilst $\mathbb{I}_{\xi_\eta}$ approximates integrals over $\Omega$ along the slice $\xi_\eta$. 
	Thus the boundary surface term $\mathbb{I}\left(\bm{v}^T\bm{T}\right)$ is a numerical approximation of continuous counterpart $BTs\left(\bm{v}, \bm{T}\right)$ defined in \eqref{eq:boundaryterm_101}.
    %%%%
    \begin{theorem}\label{theo_sbp_sem_discrete}
    Consider the semi-discrete approximation \eqref{eq:gen_hyp_transformed_discrete} of the elastic wave equation. 
    We have
    \begin{align*}
        \frac{d}{dt} \|\mathbf{Q}\left(\cdot, \cdot, \cdot, t\right)\|_{HP}^2 =  \mathbb{I}\left(\bm{v}^T\bm{T}\right),
    \end{align*}
    where $\mathbb{I}\left(\bm{v}^T\bm{T}\right)$ is the surface term defined in \eqref{eq:disc_boundary_term}.
    \begin{proof}
    Consider
    {\small
    \begin{align}
        \frac{d}{dt} \|\mathbf{Q}\left(\cdot, \cdot, \cdot, t\right)\|_{HP}^2 = \l \bm{Q} , P^{-1} \pf{}{t}  \bm{Q} \r_{H}  = \l \bm{Q}, \grad_{D_{-}}\bullet {\mathbf{F} \left({\mathbf{Q}} \right)} + \sum_{\xi \in \{ q, r, s \} }\mathbf{B}_\xi\left(\grad_{D_{+}}{\mathbf{Q}}\right) \r_{H}.
    \end{align}
    }
    Expanding the right hand side and applying the multi-dimensional SBP property \eqref{eq:multiDSBP} yields
    \begin{align*}
       &\sum_{\xi \in \{q,r,s\}} \left(\l \bm{Q},  \left(I_9\otimes\bm{D}_{-\xi}\right)\bm{F}_{\xi} \left(\bm{Q}) \r_{H}  + \l \bm{Q},  \mathbf{B}_\xi\left(\grad_{D_{+}}\bar{\mathbf{Q}}\right)\right) \r_{H}\right) \\
       &=  \mathbb{I}\left(\bm{v}^T\bm{T}\right) + \sum_{\xi \in \{q,r,s\}} \left( \l\bm{Q},  \bm{B}_{\xi} (\grad_{D_{+}} \bm{Q} ) \r_{H}- \l \left(I_9\otimes\bm{D}_{+\xi}\right) \bm{Q}, \bm{F}_{\xi} (\bm{Q}) \r_{H}\right) ,
    \end{align*}
    which from Lemma \ref{lem:anti_sym_prop_disc} gives the result. 
    \end{proof}
\end{theorem}
%%%%%%%%%%%%%%%%%%%

\section{Boundary treatment}

% Given the discrete approximant measure $\mu_n$, we call a discrete differentiation operator conservative if 
% \begin{align*}
%     \int \pf{f}{\xi} d \mu_n = BT(f).
% \end{align*}
% In the continuous setting, Equation ** is the conservative form of the derivative whilst ** is the non-conservative form. 
We will now implement the boundary conditions \eqref{eq:BC_General2} weakly using penalties. 
The idea is to impose the boundary conditions as Simultaneous Approximating Terms (SAT) in \eqref{eq:gen_hyp_transformed_discrete} with appropriate penalty parameters such that the numerical boundary terms do not permit energy growth.
We will choose penalty parameters such that a discrete energy estimate is derived.

The semi-discrete approximation with weak enforcement of boundary conditions is
  \begin{align}\label{eq:gen_hyp_transformed_discrete_SAT}
\widetilde{\mathbf{P}}^{-1} \frac{d }{d t} {\mathbf{Q}} = \grad_{D_{-}}\bullet {\mathbf{F} \left({\mathbf{Q}} \right)} + \sum_{\xi \in \{ q, r, s \} }\mathbf{B}_\xi\left(\grad_{D_{+}}{\mathbf{Q}}\right) + \sum_{\substack{\xi \in \{q,r,s\} \\ i \in \{0, n_\xi \} } }\mathbf{SAT}_{\xi, i}\left({\mathbf{Q}}\right),
\end{align}
where $\mathbf{SAT}_{\xi, i}$ are penalty terms added to the discrete equation \eqref{eq:gen_hyp_transformed_discrete} at the boundaries to enforce the boundary conditions \eqref{eq:BC_General2}. The SAT penalty terms are not unique, they are designed such that the boundary procedure is consistent and the discrete approximation is energy stable. We will consider first the case of a free-surface boundary condition, and proceed later to the general case. 

\subsection{SAT term for the free-surface boundary condition}
We consider specifically the free-surface boundary condition at all boundary surfaces, $F_{\xi,0}, \ F_{\xi,1}$ for all $\xi \in \{q,r,s\}$.   
With the free-surface boundary condition, at $F_{\xi,0}, \ F_{\xi,1}$, the traction vector vanishes $\left(\bm{T}_x, \bm{T}_y, \bm{T}_z\right) = 0$. We set the SAT terms
{
\begin{align}\label{eq:SAT_free_surface}
  \mathbf{SAT}_{\xi, 0} \ &= \ \ \bm{H}_\xi^{-1}\mathbf{e}_{0_\xi}\bm{J}\sqrt{\bm{\xi}_x^2 + \bm{\xi}_y^2 + \bm{\xi}_z^2}\left(\bm{T}_x, \bm{T}_y, \bm{T}_z, \bm{0}, \bm{0}, \bm{0}, \bm{0}, \bm{0}, \bm{0}\right)^T,  \\
  \nonumber
  \mathbf{SAT}_{\xi, n_\xi}  &= -\bm{H}_\xi^{-1}\mathbf{e}_{n_\xi}\bm{J}\sqrt{\bm{\xi}_x^2 + \bm{\xi}_y^2 + \bm{\xi}_z^2}\left(\bm{T}_x, \bm{T}_y, \bm{T}_z, \bm{0}, \bm{0}, \bm{0}, \bm{0}, \bm{0}, \bm{0}\right)^T,
\end{align}
}
where
\begin{align*}
\centering
&\bm{H}_q = \left(I_9 \otimes H_q \otimes I_{n_r} \otimes I_{n_s}\right), \quad
\bm{H}_r = \left(I_9 \otimes I_{n_q} \otimes H_r \otimes I_{n_s}\right), \\ 
&\bm{H}_s = \left(I_9 \otimes I_{n_q} \otimes I_{n_r} \otimes H_s\right),
\end{align*}
%%%
\begin{align*}
    \mathbf{e}_{0_q} = (I_9 \otimes e_{0_q}e_{0_q}^T \otimes I_{n_r} \otimes I_{n_s}), && \mathbf{e}_{n_q} = (I_9 \otimes e_{n_q}e_{n_q}^T \otimes I_{n_r} \otimes I_{n_s}), \\
    \mathbf{e}_{0_r} = (I_9 \otimes I_{n_q} \otimes e_{0_r}e_{0_r}^T \otimes I_{n_s}), && 
    \mathbf{e}_{n_r} = (I_9 \otimes I_{n_q} \otimes e_{n_r}e_{n_r}^T \otimes I_{n_s}), \\
    \mathbf{e}_{0_s} = (I_9 \otimes I_{n_q}  \otimes I_{n_r} \otimes e_{0_s}e_{0_s}^T), &&
    \mathbf{e}_{n_s} = (I_9 \otimes I_{n_q} \otimes I_{n_r} \otimes e_{n_s}e_{n_s}^T),
\end{align*}
% $$
% \mathbf{e}_{0_q} = \left(I_9 \otimes e_{0_q}e_{0_q}^T \otimes I_{n_r} \otimes I_{n_s}\right), \quad
% \mathbf{e}_{n_q} = \left(I_9 \otimes e_{n_q}e_{n_q}^T \otimes I_{n_r} \otimes I_{n_s}\right),
% $$
% $$
% \mathbf{e}_{0_r} = \left(I_9 \otimes I_{n_q} \otimes e_{0_r}e_{0_r}^T \otimes I_{n_s}\right), \quad
% \mathbf{e}_{n_r} = \left(I_9 \otimes I_{n_q} \otimes e_{n_r}e_{n_r}^T \otimes I_{n_s}\right),
% $$
% $$
% \mathbf{e}_{0_s} = \left(I_9 \otimes I_{n_q}  \otimes I_{n_r} \otimes e_{0_s}e_{0_s}^T\right), \quad
% \mathbf{e}_{n_s} = \left(I_9 \otimes I_{n_q} \otimes I_{n_r} \otimes e_{n_s}e_{n_s}^T\right),
% $$
$$
e_{0_\xi} = \left(1, 0, 0, \cdots, 0\right)^T, \ 
e_{n_\xi} = \left(0, 0, 0, \cdots, 1\right)^T.
$$
Here $I_9$ and $I_{n_\xi}$ are identity matrices of size  $9\times 9$ and $n_\xi \times n_\xi$, respectively, and $\mathbf{e}_{0_\xi}$, $\mathbf{e}_{n_\xi}$ are boundary projection operators.

We state the first main result
\begin{theorem}\label{theo:free_surface_bc_procedure}
    Consider the semi-discrete approximation \eqref{eq:gen_hyp_transformed_discrete_SAT} of the elastic wave equation with the SAT terms $\mathbf{SAT}_{\xi, i}$ defined in \eqref{eq:SAT_free_surface}.
    We have
    \begin{align*}
        \frac{d}{dt} \|\mathbf{Q}\left(\cdot, \cdot, \cdot, t\right)\|_{HP}^2 = 0.
    \end{align*}
    \begin{proof}
    Consider
    {\small
    \begin{align*}
        \frac{d}{dt} \|\mathbf{Q}\left(\cdot, \cdot, \cdot, t\right)\|_{HP}^2 = \l \bm{Q} , P^{-1} \pf{}{t}  \bm{Q} \r_{H}  &= \l \bm{Q}, \grad_{D_{-}}\bullet {\mathbf{F} \left({\mathbf{Q}} \right)} + \sum_{\xi \in \{ q, r, s \} }\mathbf{B}_\xi\left(\grad_{D_{+}}{\mathbf{Q}}\right) \r_{H} \\
        &+  \l \bm{Q},\sum_{\substack{\xi \in \{q,r,s\} \\ i \in \{0, 1 \} } }\mathbf{SAT}_{\xi, i}\left({\mathbf{Q}}\right)\r_{H}.
    \end{align*}
    }
    By Theorem \ref{theo_sbp_sem_discrete} we have
    \begin{align}
        \frac{d}{dt} \|\mathbf{Q}\left(\cdot, \cdot, \cdot, t\right)\|_{HP}^2  =  \mathbb{I}\left(\bm{v}^T\bm{T}\right) +   \sum_{\substack{\xi \in \{q,r,s\} \\ i \in \{0, n_{\xi} \} } }\l \bm{Q},\mathbf{SAT}_{\xi, i}\left({\mathbf{Q}}\right)\r_{H},
    \end{align}
    with
    {
    $$
    \sum_{i \in \{ 0,n_{\xi} \} } \l \bm{Q}, \mathbf{SAT}_{\xi, i} \r_{H} = -\sum_{i\in \{ 0,n_{\xi} \} } (-1)^{\xi_i+1} \mathbb{I}_{\xi_i}\left(\mathbf{v}^T\mathbf{T}\right) 
    $$
    }
    where $\xi_0 =0$ and $\xi_{n_\xi} = 1$.
    We therefore have
    $$
    \sum_{\substack{\xi \in \{q,r,s\} \\ i \in \{0, n_{\xi} \} } }\l \bm{Q},\mathbf{SAT}_{\xi, i}\left({\mathbf{Q}}\right)\r_{H}= - \mathbb{I}\left(\bm{v}^T\bm{T}\right),
    $$
    which  completes the proof. 
    \end{proof}
\end{theorem}

\subsection{SAT terms for general boundary conditions}
%%%%%%%%%%%%%%%%%%%%%%%%%
We will now construct the SAT terms for the general boundary condition \eqref{eq:BC_General2}. Similar to the DG framework \cite{DuruGabrielIgel2017,ExaHyPE2019,Duru_exhype_2_2019}, a weak boundary procedure can be derived by constructing boundary data, $\widehat{v}_\eta, \widehat{T}_\eta$, which are the solution of a Riemann-like problem constrained to satisfy the boundary condition \eqref{eq:BC_General2} exactly. SAT penalty terms are constructed by penalizing data, that is $\widehat{v}_\eta, \widehat{T}_\eta$, against the in-going waves only. The construction of boundary data, $\widehat{v}_\eta, \widehat{T}_\eta$, can be found in Appendix \ref{sec:hat_variables}. We also refer the reader to  ~\cite{DuruGabrielIgel2017,ExaHyPE2019} for more detailed discussions. %Once the hat-variables are available, we construct physics based numerical flux fluctuations by penalizing data against the incoming characteristics \eqref{eq:characteristics} at the element faces.
  %%%%%
  
 Introduce the penalty terms
   %%%%%%%%%%%%%%%%%%%%%%%%%
    %%%%%%%%%%%%%%%%%%%%%%%%%
    {
    \begin{align}\label{eq:penalty_terms}
     %%%%%%%%%%%%%%%%%%%%%%%%%
      %%%%%%%%%%%%%%%%%%%%%%%%%
  &{G}_\eta = \frac{1}{2} {Z}_\eta \left({v}_\eta - \widehat{{v}}_\eta \right)- \frac{1}{2}\left({T}_\eta  - \widehat{{T}}_\eta \right)\Big|_{\xi = 0},   \quad \widetilde{G}_\eta \coloneqq \frac{1}{{Z}_\eta}{G}_\eta  ,
  \\
  \nonumber
  &{G}_\eta  = \frac{1}{2} {Z}_\eta\left({v}_\eta  - \widehat{{v}}_\eta \right)+ \frac{1}{2}\left({T}_\eta  - \widehat{{T}}_\eta \right)\Big|_{\xi = 1}, \quad \widetilde{G}_\eta\coloneqq \frac{1}{{Z}_\eta }{G}_\eta .
   %%%%%%%%%%%%%%%%%%%%%%%%%
    %%%%%%%%%%%%%%%%%%%%%%%%%
  \end{align}
  %%%%%%%%%%%%%%%%%%%%%%%%%
  }
  %%%%%%%%%%%%%%%%%%%%%%%%%
   %Else if $Z_\eta =0$, then ${G}_\eta = \widetilde{G}_\eta = 0$.
   %%%%%%%%%%%%%%%%%%%%%%%%%
   %%%%%%%%%%%%%%%%%%%%%%%%%
%    \begin{align*}
%    %%%%%%%%%%%%%%%%%%%%%%%%%
%  &{G}_\eta ^{\pm} = \widetilde{G}_\eta ^{\pm}  = \frac{1}{2} \left({v}_\eta ^{\pm} - \widehat{{v}}^{\pm}_\eta \right) \mp \frac{1}{2}\left({T}^{\pm}_\eta  - \widehat{{T}}^{\pm}_\eta \right) \equiv 0 .
%  %%%%%%%%%%%%%%%%%%%%%%%%%
%  \end{align*}
  %%%%%%%%%%%%%%%%%%%%%%%%%
  The penalty terms are computed in the transformed coordinates $l,m,n$. We will now  rotate them to the physical coordinates $x,y,z$, we  have
  %%%%%%%%%%%%%%%%%%%%%%%%%
  \begin{align}\label{eq:rotate_back_forth}
  {\mathbf{G}} := \begin{pmatrix}
{G}_{x} \\
{G}_{y} \\
{G}_{z}
\end{pmatrix}
 = \mathbf{R}^T\begin{pmatrix}
{G}_{n} \\
{G}_{m} \\
{G}_{l}
\end{pmatrix},
%%%%%%%%%%%%%%%%%%%
%%%%%%%%%%%%%%%%%%%
\quad 
%%%%%%%%%%%%%%%%%%%
%%%%%%%%%%%%%%%%%%%
\widetilde{\mathbf{G}}:= 
  \begin{pmatrix}
\widetilde{G}_{x} \\
\widetilde{G}_{y} \\
\widetilde{G}_{z}
\end{pmatrix} = \mathbf{R}^T\begin{pmatrix}
\widetilde{G}_{n} \\
\widetilde{G}_{m} \\
\widetilde{G}_{l}
\end{pmatrix}.
\end{align}
  %%%%%%%%%%%%%%%%%%%%%%%%%
  Note that
%%%%%%%%%%%%%%%%%%%%%%%
%%%%%%%%%%%%%%%%%%%%%%%
%%%%%%%%%%%%%%%%%%%%%%%
%%%%%%%%%%%%%%%%%%%%%%%
\begin{equation}\label{eq:identity_pen}
\begin{split}
%%%%%%%%%%%%%%%%%%%%%%%
%%%%%%%%%%%%%%%%%%%%%%%
 \left(\mathbf{v}^T \mathbf{G} - \mathbf{T}^T \widetilde{\mathbf{G}} + \mathbf{v}^T\mathbf{T}\right)&\Big|_{\xi = 0} = \sum_{\eta \in \{ l,m,n \} } \left( v_\eta  G_\eta    - \frac{1}{Z_\eta  }T_\eta   G_\eta + v_\eta  T_\eta \right)\Big|_{\xi = 0} \\
 %%%%%%%%%%%%%%%%%%%%%%%
%%%%%%%%%%%%%%%%%%%%%%% 
 &= \sum_{\eta \in \{ l,m,n \} } \frac{1}{Z_\eta  }\left(|G_\eta |^2 + p^2_\eta \left(v_\eta , T_\eta , Z_\eta \right) - {q}^2_\eta \left(\widehat{v}_\eta , \widehat{T}_\eta , Z_\eta \right)\right)\Big|_{\xi = 0}\\
%%%%%%%%%%%%%%%%%%%%%%%
&= \sum_{\eta \in \{ l,m,n \} } \left(\frac{1}{Z_\eta  }|G_\eta |^2  + \widehat{T}_\eta \widehat{v}_\eta \right)\Big|_{\xi = 0}, \\
%%%%%%%%%%%%%%%%%%%%%%% 
\left(\mathbf{v}^T \mathbf{G} + \mathbf{T}^T \widetilde{\mathbf{G}} - \mathbf{v}^T\mathbf{T}\right)&\Big|_{\xi = 1} = \sum_{\eta \in \{ l,m,n \} } \left(v_\eta  G_\eta   + \frac{1}{Z_\eta  }T_\eta   G_\eta - v_\eta  T_\eta \right)\Big|_{\xi = 1} \\
%%%%%%%%%%%%%%%%%%%%%%%
%%%%%%%%%%%%%%%%%%%%%%% 
&= \sum_{\eta \in \{ l,m,n \} } \frac{1}{Z_\eta  }\left(|G_\eta |^2 + q^2_\eta \left(v_\eta , T_\eta , Z_\eta \right) - {p}^2_\eta \left(\widehat{v}_\eta , \widehat{T}_\eta , Z_\eta \right)\right)\Big|_{\xi = 1}\\
%%%%%%%%%%%%%%%%%%%%%%%
&= \sum_{\eta \in \{ l,m,n \} } \left(\frac{1}{Z_\eta  }|G_\eta |^2  - \widehat{T}_\eta \widehat{v}_\eta \right)\Big|_{\xi = 1} .
%%%%%%%%%%%%%%%%%%%%%%%
\end{split}
\end{equation}
%%%%%%%%%%%%%%%%%%%%%%%
%%%%%%%%%%%%%%%%%%%%%%%
%
We introduce the SAT vector that matches the eigen--structure of the elastic wave equation 
  %%%%% 
 {
%\footnotesize
  %%%%%
 \begin{align}
% &\mathbf{SAT}_{0} =\\
% \nonumber
% &\left[{G}_x, {G}_y,  {G}_z, -{n_x}\widetilde{{G}}_x, -{n_y}\widetilde{{G}}_y , -{n_z}\widetilde{{G}}_z, -\left({n_y}\widetilde{{G}}_x + {n_x}\widetilde{{G}}_y\right), -\left({n_z}\widetilde{{G}}_x + {n_x}\widetilde{{G}}_z\right), -\left({n_z}\widetilde{{G}}_y + {n_y}\widetilde{{G}}_z\right)   \right]^T,\\ \nonumber
%  %%%
% &\mathbf{SAT}_{n_\xi} = \left[{G}_x, {G}_y,  {G}_z,         {n_x}\widetilde{{G}}_x, {n_y}\widetilde{{G}}_y , {n_z}\widetilde{{G}}_z, \left({n_y}\widetilde{{G}}_x + {n_x}\widetilde{{G}}_y\right), \left({n_z}\widetilde{{G}}_x + {n_x}\widetilde{{G}}_z\right), \left({n_z}\widetilde{{G}}_y + {n_y}\widetilde{{G}}_z\right)   \right]^T.
\mathbf{SAT}_{0} = 
         \begin{pmatrix}
            {G}_x\\
            {G}_y\\
            {G}_z\\
            -{n_x}\widetilde{{G}}_x, \\
            -{n_y}\widetilde{{G}}_y \\
            -{n_z}\widetilde{{G}}_z, \\
            -\left({n_y}\widetilde{{G}}_x + {n_x}\widetilde{{G}}_y\right)\\
            -\left({n_z}\widetilde{{G}}_x + {n_x}\widetilde{{G}}_z\right)\\ -\left({n_z}\widetilde{{G}}_y + {n_y}\widetilde{{G}}_z\right)\\
         \end{pmatrix},
         \quad
\mathbf{SAT}_{n_\xi} = 
         \begin{pmatrix}
            {G}_x\\
            {G}_y\\
            {G}_z\\
            {n_x}\widetilde{{G}}_x, \\
            {n_y}\widetilde{{G}}_y \\
            {n_z}\widetilde{{G}}_z, \\
            \left({n_y}\widetilde{{G}}_x + {n_x}\widetilde{{G}}_y\right)\\
            \left({n_z}\widetilde{{G}}_x + {n_x}\widetilde{{G}}_z\right)\\ \left({n_z}\widetilde{{G}}_y + {n_y}\widetilde{{G}}_z\right)\\
         \end{pmatrix}.
\end{align}
      }
  %%%%%
  Here, $\mathbf{n} = (n_x, n_y, n_z)^T$ is the unit normal vector on the boundary defined in \eqref{eq:normal_vector}.
  %%%%% 
  Note that
  \begin{equation}\label{eq:scalar_product_flux}
  \mathbf{Q}^T\mathbf{SAT}_{0} = \mathbf{v}^T \mathbf{G} - \mathbf{T}^T \widetilde{\mathbf{G}} , \quad \mathbf{Q}^T\mathbf{SAT}_{ n_\xi} = \mathbf{v}^T \mathbf{G} + \mathbf{T}^T \widetilde{\mathbf{G}}.
  \end{equation}
  %%%%%
  %%%%% 
  %%%%%
  %%%%%
Finally, the SAT terms for the general boundary conditions are defined as follows
 %%%%%%%%%%%%%%%%%
\begin{align}\label{eq:SAT_General}
\mathbf{SAT}_{\xi, i} = -\bm{H}_{\xi}^{-1}\mathbf{e}_{\xi,i}\bm{J}\sqrt{\bm{\xi}_x^2 + \bm{\xi}_y^2 + \bm{\xi}_z^2}\mathbf{SAT}_i.
\end{align}
%%%%%%%%%%%%%%%%
Introduce the fluctuation term
\begin{align}\label{eq:fluctuation_term}
   {F}_{luc}\left({\bm{G}},\mathbf{Z}\right) \coloneqq  -\sum_{\xi \in  \{ q, r, s \} }\sum_{i \in \{ 0, n_\xi \} } \mathbb{I}_{\xi_i}\left(\sum_{\eta = l,m,n} \frac{1}{Z_\eta  }|G_\eta |^2\right) \le 0, 
\end{align}
%%%%%%%%%%%%%%%%
%%%%%%%%%%%%%%%%
and  discrete boundary surface terms $\mathbb{I} \left(\widehat{\mathbf{v}}^T \widehat{\mathbf{T}}\right)$. Note that 
    \begin{align}\label{eq:disc_boundary_term_hat}
    \mathbb{I} \left(\widehat{\mathbf{v}}^T \widehat{\mathbf{T}}\right)=  \sum_{\xi \in \{q,r,s\} }S_{\xi}\left(\bm{J}\sqrt{\bm{\xi}_x^2+\bm{\xi}_y^2+\bm{\xi}_z^2 },  \widehat{\mathbf{v}}^T\widehat{\mathbf{T}}\right),
\end{align}
where the surface cubature $S_{\xi}$ is defined in \eqref{eq:bt_q}--\eqref{eq:bt_s}. Note also that by \eqref{eq:identity_3_bc} the boundary term is never positive, $\mathbb{I} \left(\widehat{\mathbf{v}}^T \widehat{\mathbf{T}}\right) \le 0$ for all $|\gamma_\eta|\le 1$, and by \eqref{eq:fluctuation_term} the  fluctuation term is never positive, ${F}_{luc}\left({\bm{G}},\mathbf{Z}\right) \le 0.$
We state the second main result
    \begin{theorem}\label{theo:gneral_bc_procedure}
    Consider the semi-discrete approximation \eqref{eq:gen_hyp_transformed_discrete_SAT} of the elastic wave equation with the SAT-terms  $\mathbf{SAT}_{\xi, i}$ defined in \eqref{eq:SAT_General}.
    We have
    \begin{align*}
        \frac{d}{dt} \|\mathbf{Q}\left(\cdot, \cdot, \cdot, t\right)\|_{HP}^2 = {F}_{luc}\left({\bm{G}},\mathbf{Z}\right)+ \mathbb{I} \left(\widehat{\mathbf{v}}^T \widehat{\mathbf{T}}\right)   \le 0.
    \end{align*}
    \end{theorem}
%\subsection{Proof of Theorem \ref{theo:gneral_bc_procedure}}
    \begin{proof}
    Consider
    {\small
    \begin{align*}
        \frac{d}{dt} \|\mathbf{Q}\left(\cdot, \cdot, \cdot, t\right)\|_{HP}^2 = \l \bm{Q} , P^{-1} \pf{}{t}  \bm{Q} \r_{H}  &= \l \bm{Q}, \grad_{D_{-}}\bullet {\mathbf{F} \left({\mathbf{Q}} \right)} + \sum_{\xi= q, r, s}\mathbf{B}_\xi\left(\grad_{D_{+}}{\mathbf{Q}}\right) \r_{H} \\
        &+  \l \bm{Q},\sum_{\substack{\xi \in \{q,r,s\} \\ i \in \{0, 1 \} } }\mathbf{SAT}_{\xi, i}\left({\mathbf{Q}}\right)\r_{H}.
    \end{align*}
    }
    By Theorem \ref{theo_sbp_sem_discrete} we have
    \begin{align}
        \frac{d}{dt} \|\mathbf{Q}\left(\cdot, \cdot, \cdot, t\right)\|_{HP}^2  = \mathbb{I} \left({\mathbf{v}}^T {\mathbf{T}}\right)+   \sum_{\substack{\xi \in \{q,r,s\} \\ i \in \{0, n_{\xi} \} } }\l \bm{Q},\mathbf{SAT}_{\xi, i}\left({\mathbf{Q}}\right)\r_{H},
    \end{align}
    with
    $$
    \l \bm{Q}, \mathbf{SAT}_{\xi, 0} \r_{H} = -\mathbb{I}_{\xi_0}\left(\mathbf{v}^T \mathbf{G} - \mathbf{T}^T \widetilde{\mathbf{G}}\right), \quad \l \bm{Q}, \mathbf{SAT}_{\xi, n_{\xi}} \r_{H} = -\mathbb{I}_{\xi_{n_\xi}}\left(\mathbf{v}^T \mathbf{G} + \mathbf{T}^T \widetilde{\mathbf{G}}\right),
    $$
    and
    $$
    \mathbb{I} \left({\mathbf{v}}^T {\mathbf{T}}\right) = \sum_{\substack{\xi \in \{q,r,s\} \\ i \in \{0, 1 \} } } \left(-1\right)^{\xi_i + 1}\mathbb{I}_{\xi_i}\left({\mathbf{v}}^T{\mathbf{T}}\right).
    $$
    By using the identity \eqref{eq:identity_pen} we have
    \begin{equation}
    \begin{split}
        &\frac{d}{dt} \|\mathbf{Q}\left(\cdot, \cdot, \cdot, t\right)\|_{HP}^2  = \\
        &-\sum_{\xi \in  \{ q,r,s \} }\left(\mathbb{I}_{\xi_0}\left(\mathbf{v}^T \mathbf{G} - \mathbf{T}^T \widetilde{\mathbf{G}} + \mathbf{v}^T\mathbf{T}\right) + \mathbb{I}_{\xi_{n_\xi}}\left(\mathbf{v}^T \mathbf{G} + \mathbf{T}^T \widetilde{\mathbf{G}} - \mathbf{v}^T\mathbf{T}\right)\right)\\
        &= -\sum_{\xi \in \{ q,r,s \} }\left(\mathbb{I}_{\xi_0}\left(\sum_{\eta \in \{ l,m,n \} } \frac{1}{Z_\eta  }|G_\eta |^2 + \widehat{\mathbf{v}}^T\widehat{\mathbf{T}} \right) + \mathbb{I}_{\xi_{n_\xi}}\left(\sum_{\eta \in \{ l,m,n \} } \frac{1}{Z_\eta  }|G_\eta |^2 -\widehat{\mathbf{v}}^T\widehat{\mathbf{T}} \right)\right)\\
        & =  {F}_{luc}\left({\bm{G}},\mathbf{Z}\right) + \mathbb{I} \left(\widehat{\mathbf{v}}^T \widehat{\mathbf{T}}\right)   \le 0.
        \end{split}
    \end{equation}
    The proof is complete. 
    \end{proof}
%%%%%%%%%%%%%%%%%%%
The fluctuation term ${F}_{luc}\left({\bm{G}},\mathbf{Z}\right) \le 0 $ adds a little numerical dissipation on the boundary. However, in the limit of mesh refinement the fluctuation term vanishes, that is ${F}_{luc}\left({\bm{G}},\mathbf{Z}\right) \to 0^+$ as $h \to 0^+$, and we have  $\mathbb{I} \left(\widehat{\mathbf{v}}^T \widehat{\mathbf{T}}\right) \to BTs\left(\bm{v}, \bm{T}\right)$. Thus the discrete main results, Theorems \eqref{theo:free_surface_bc_procedure} and \eqref{theo:gneral_bc_procedure}, are completely analogous to the continuous counterpart, Theorem \ref{theo:energy_estimate_cont}.

\section{Numerical error analysis}
In this section we will analyse the numerical errors for the semi-discrete approximation \eqref{eq:gen_hyp_transformed_discrete_SAT}.   
We will derive a priori error estimate and prove convergence of the error as $h \to 0^+$. Next, we will discuss numerical dispersion errors that are peculiar to wave propagation problems, and  which are most prominent at high frequencies. 
%%%
%%%
\subsection{A priori error estimate}
%%%
Let $\boldsymbol{\mathcal{Q}}$ denote the exact solution of the IBVP, and $\boldsymbol{\mathcal{Q}}(q_i, r_j, s_k, t)$ denote the restriction of the exact solution on the grid $(q_i, r_j, s_k)$ at time $t$. 
We introduce the numerical error on the grid
%%%%%%
%%%%%%
\begin{align}
  \boldsymbol{\mathcal{E}}_{ijk}(t) \coloneqq \mathbf{Q}_{ijk}(t)-\boldsymbol{\mathcal{Q}}(q_i, r_j, s_k, t). 
\end{align}
%%%%%%
%%%%%%
The error $\boldsymbol{\mathcal{E}} \in \R^{9 n_q n_r n_s}$ satisfies the error equation
%%%%%%
\begin{align}\label{eq:gen_hyp_transformed_discrete_SAT_error}
\widetilde{\mathbf{P}}^{-1} \frac{d }{d t} \boldsymbol{\mathcal{E}} = \grad_{D_{-}}\bullet {\mathbf{F} \left(\boldsymbol{\mathcal{E}} \right)} + \sum_{\xi \in \{ q, r, s \} }\mathbf{B}_\xi\left(\grad_{D_{+}}\boldsymbol{\mathcal{E}}\right) 
+\sum_{\substack{\xi \in \{q,r,s\} \\ i \in \{0, n_\xi \} } }\mathbf{SAT}_{\xi, i}\left(\boldsymbol{\mathcal{E}}\right) + \mathbb{T},
%+ \mathbf{SAT}_{q}\left({\mathbf{Q}}\right),
\end{align}
 %%%   
%%
where $\mathbb{T}\in \R^{9 n_q n_r n_s}$ is the truncation error of the SBP FD operator. Note that the truncation error $\mathbb{T}$ is a 3D vector field.  However, it has a structure which is similar in all spatial directions. 
In particular, for grid points $\xi_j = j h_{\xi}$ in the spatial direction $\xi \in \{q, r, s\}$ the truncation error is of the form
%%%%%%%%%%%%%% 
\begin{equation}\label{eq:truncation_error}
\begin{split}
& \mathbb{T}_{\xi,j}  = \left \{
\begin{array}{rl}
 h_{\xi}^{\gamma} \beta_j\frac{\partial^{\gamma+1} \mathcal{Q}}{\partial \xi^{\gamma +1}}\Big|_{\xi_j},  & \text{if boundary}  ,\\
 h_{\xi}^{\nu} \beta_j\frac{\partial^{\nu+1} \mathcal{Q}}{\partial \xi^{\nu+1}}\Big|_{\xi_j},
 & \text{if interior}.
\end{array} \right\} ,
\end{split}
\end{equation}
where $\mathbb{T}_{ijk} = \mathbb{T}_{q,i} + \mathbb{T}_{r,j} + \mathbb{T}_{s,k}$.
%%%%%%%%%%%%%%
Here, $\beta_j$, are mesh independent constants,  $h_{\xi}>0$ is the grid spacing,  $\gamma \in \{1, 2, \cdots \}$ is the order of accuracy of the SBP FD stencils close to the boundary and $\nu \in \{1, 2, \cdots \}$ is the order of accuracy of the SBP FD stencils in the interior, away from the boundaries.
For traditional SBP operators based on central difference stencils the interior accuracy is always even, and  we have $(\gamma, \nu) = (p, 2p)$, for $p \in \N $. For upwind SBP operators the interior order of accuracy can be odd or even. As discussed in section \ref{sec:spatial_approximation}, see also \cite{Mattsson2017}, upwind SBP operators with even-order $\left(2p\right)$-th accuracy in the interior are closed with $p$-th order accurate stencils close to boundaries, and we also have $(\gamma, \nu) = (p, 2p)$. Upwind SBP operators with odd-order $\left(2p+1\right)$-th accuracy in the interior are closed with $p$-th order accurate stencils close to boundaries, giving $(\gamma, \nu) = (p, 2p+1)$. 
The traditional SBP operators and upwind operators can yield $\left(p+1\right)$-th global order of accuracy, for smooth solutions.

Note that at the initial time the numerical error is zero $\boldsymbol{\mathcal{E}} (0) =0$. Application of Theorem \ref{theo:free_surface_bc_procedure} or Theorem \ref{theo:gneral_bc_procedure} to the error equation \eqref{eq:gen_hyp_transformed_discrete_SAT_error} gives the error estimate
    \begin{theorem}\label{theo:error_estimate}
    Consider the semi-discrete error equation \eqref{eq:gen_hyp_transformed_discrete_SAT_error}, with the numerical error $\boldsymbol{\mathcal{E}}(t) $  and the truncation error $\boldsymbol{\mathbb{T}}$.
    If the SAT-terms are chosen such that Theorem \ref{theo:free_surface_bc_procedure} or Theorem \ref{theo:gneral_bc_procedure} holds, then we have  
    \begin{align*}
         \|\boldsymbol{\mathcal{E}}\left(t\right)\|_{HP} 
        \le
        \int_0^t\|\boldsymbol{\mathbb{T}}\left(\tau\right)\|_{HP^{-1}} d\tau.
    \end{align*}
\end{theorem}
%%%
 Theorem \ref{theo:error_estimate} proves that the numerical error $\boldsymbol{\mathcal{E}}$ is bounded by the truncation error $\boldsymbol{\mathbb{T}}$, and will converge to zero if $\boldsymbol{\mathbb{T}}$ is square integrable.  If $\boldsymbol{\mathcal{Q}}$ is sufficiently smooth such that the highest derivatives in \eqref{eq:truncation_error} are continuous then the numerical error will converge to zero optimally, $\boldsymbol{\mathcal{E}} = O(h^{\gamma+1})$.
 
  Note also that Theorem \ref{theo:error_estimate}  hold for traditional and upwind SBP operators. Thus traditional   and even-order upwind SBP operators with $(\gamma, \nu) = (p, 2p)$ will have the same asymptotic error $\boldsymbol{\mathcal{E}} = O(h^{p+1})$.
  However, as we will see from numerical experiments performed in the next section, on a marginally resolved mesh the numerical errors could completely differ, where for example the upwind SBP operators yield optimal numerical errors than the traditional SBP operator. This can be explained by analysing the numerical dispersion properties of the operators.

\subsection{Numerical dispersion relation analysis}\label{sec:dsipersion_relation}
In order to understand the numerical dispersion properties of the SBP operators we consider the 1D shear plane wave propagating along the $x$-axis
$$
Q(x,t) 
=
\begin{pmatrix}
v_0\\
\sigma_0
\end{pmatrix}
\exp \left(-i (\omega t - kx\right)), \qquad i = \sqrt{-1},
$$
%%%
with the shear wave speed $c_s = \sqrt{{\mu}/{\rho}} >0$, where $\mu > 0$ is the shear modulus and $\rho >0$ is the density of the medium.
Here, $Q_0 = \left(v_0, \sigma_0\right)^T$ is the constant polarisation vector, $\omega \in \mathbb{R}$ is the temporal frequency and $k\in \mathbb{R}$ is spatial wave number.
%%%
The plane wave $Q(x,t)$ solves the 1D elastic wave equation 
\begin{align}\label{eq:1D_elastic_wave_equation}
    \rho \frac{\partial v}{\partial t} = \frac{\partial \sigma}{\partial x}, 
    \qquad
    \frac{1}{\mu} \frac{\partial \sigma}{\partial t} = \frac{\partial v}{\partial x},
\end{align}
 subject to the solvability condition called the dispersion relation
\begin{align}\label{eq:1D_elastic_wave_equation_dispersion}
     \omega  = c_s k. 
\end{align}
%%%
We introduce the phase velocity $V_p$ and the group velocity $V_g$ defined by
%%%
$$
V_p:=\frac{\omega}{k} = c_s>0, \qquad V_g:=\frac{\partial \omega}{\partial k} = c_s>0.
$$
%%%
For the simple 1D model \eqref{eq:1D_elastic_wave_equation} the dispersion relation \eqref{eq:1D_elastic_wave_equation_dispersion} is linear, the phase and group velocities are constant $c_s>0$, and  are independent of the frequency and the wave number. 

Now let us consider the semi-discrete counterpart
\begin{align}\label{eq:1D_wave_disc}
     \rho \frac{d \mathbf{v}}{dt} = D_{+} \boldsymbol{\sigma}, \qquad \frac{1}{\mu} \frac{d \boldsymbol{\sigma}}{dt} =D_{-} \mathbf{v},
\end{align}
where $D_+$ and $D_-$ are the upwind SBP operators.
We consider the interior stencils only, thus
$$
\left(D_{+} \mathbf{v}\right)_j = \frac{1}{h}\sum_{j=-r}^{q} \alpha_j v_j, \qquad
\left(D_{-} \mathbf{v}\right)_j = \frac{1}{h}\sum_{j=-q}^{r} \beta_j v_j, \qquad
    0\le r < q,
$$
with the consistency requirements
$$
\sum_{j=-r}^{q} \alpha_j = \sum_{j=-r}^{q} \beta_j = 0, \qquad \sum_{j=-r}^{q} j\alpha_j = \sum_{j=-q}^{r} j\beta_j = 1.
$$
Here, $h>0$ is the uniform grid spacing, $\alpha_j, \beta_j$ are the non-dimensional constant coefficients defining the upwind finite difference stencils. Note that $\beta_j = -\alpha_{-j}$ for $j = -q, -(q-1), \cdots r$.
The Upwind  SBP finite difference coefficients $\alpha_j$ for the interior stencils with  order of accuracy $2, 3, 4, 5, 6, 7, 8, 9$ are given in Table \ref{tab:upwind}.
  \begin{table}[h]%{0.25\textwidth}
        \centering
       \begin{tabular}{||ccc||c|c|c|c|c|c|c|c|c|c||}
       \hline
order & $r$ & $q$ &  $\alpha_{-4}$& $\alpha_{-3}$ & $\alpha_{-2}$ & $\alpha_{-1}$ & $\alpha_{0}$ &    $\alpha_{1}$ & $\alpha_{2}$ & $\alpha_{3}$ & $\alpha_{4}$ & $\alpha_{5}$ \\ 
\hline
2   & 0         & 2  &  - & - & - & - & $-3/2$ & $2$ & $-1/2$ & - & - & -                                   \\ \hline
3   & 1         & 2  &  - & - & - & $-1/3$ & $-1/2$ &   $1$ &  $-1/6$ & - & - & -                           \\ \hline
4   & 1         & 3  &  - & - & - & $-1/4$& $-5/6$ &    $3/2$& $-1/2$ & $1/12$ & - & -                      \\ \hline
5   & 2         & 3  &  - & - & $1/20$ &  $-1/2$& $-1/3$& $1$& $-1/4$& $1/30$  & - & -                       \\ \hline
6   & 2         & 4  &  - & - & $1/30$ & $-2/5$ & $ -7/12$ & $4/3$ & $-1/2$ & $2/15$ & $-1/60$  & -             \\ \hline
7   & 3         & 4  &  - & $-1/105$ &  $1/10$ & $-3/5$ & $-1/4$ & $1$ & $-3/10$ &    $1/15$ & $-1/140$    & -  \\ \hline
8   & 3         & 5  &  - & $-1/168$ & $1/14$ & $-1/2$ & $-9/20$ & $5/4$ & $-1/2$ & $1/6$ & $1/28$ & $1/280$          \\ \hline
9   & 4         & 5  & $1/504$& $-1/42$ & $1/7$ & $-2/3$ & $-1/5$ & $1$ & $-1/3$ & $2.0/21$&  $-1/56$ & $1/630$                    \\ \hline
% \hline
\end{tabular}
\caption{Upwind forward SBP finite difference coefficients for the interior stencils with interior order of accuracy $2, 3, 4, 5, 6, 7, 8, 9.$ }
\label{tab:upwind}
\end{table}

For the traditional SBP finite difference operator we have
%%%%
%%%%
$$
\left(D \mathbf{v}\right)_j=\left(D_{+} \mathbf{v}\right)_j = \left(D_{-} \mathbf{v}\right)_j= \frac{1}{h}\sum_{j=-q}^{q} \gamma_j v_j, \qquad
     q\ge 1,
$$
where $\gamma_j$ are the non-dimensional constant coefficients of the finite difference operator, with $\gamma_{-j} = -\gamma_j$, $\gamma_0 =0$, and satisfying the consistency requirements
$$
\sum_{j=-q}^{q} \gamma_j  = 0, \qquad \sum_{j=1}^{q} 2j\gamma_j = 1.
$$
The traditional  SBP finite difference coefficients $\gamma_j$ for the interior stencils with  order of accuracy $2, 4, 6, 8$ are given in Table \ref{tab:traditional}.
\begin{table}[h]%{0.25\textwidth}
        \centering
       \begin{tabular}{||cc||c|c|c|c||}
       \hline
order  & $q$ &  $\gamma_{1}$& $\gamma_{2}$ & $\gamma_{3}$ & $\gamma_{4}$      \\ 
\hline
2          & 1  &  $1/2$ & - & - & -                                   \\ \hline
4          & 2  &  $2/3$ & $-1/12$ & - & -                     \\ \hline
6           & 3  &  $3/4$ & $-3/20$ & $1/60$ & -              \\ \hline
8          & 4  &  4/5 & $-1/5$ & $4/105$ & $-1/280$           \\ \hline
% \hline
\end{tabular}
\caption{Traditional (central) SBP finite difference coefficients for the interior stencils with interior order of accuracy $2,4, 6, 8.$ }
\label{tab:traditional}
\end{table}
%%%%%

Inserting $Q(x,t)$ in \eqref{eq:1D_wave_disc} we have the numerical dispersion relation for the upwind finite difference SBP operator
% \begin{align}
%       \widetilde{\omega}^2 = -c^2 k_{+} \ k_{-}
% \end{align}
% with
% \begin{align}
%     k_{+} \coloneqq  \sum_{j=-r}^{p} \alpha_j \exp(ij \widetilde{k}), \qquad 
%     k_{-} \coloneqq  -\sum_{j=-r}^{p} \alpha_j \exp(-ij\widetilde{k}), 
% \end{align}
% where 
% $$
% \widetilde{\omega} = h\omega, \qquad \widetilde{k} = hk
% $$
% are the numerical frequency and numerical wave number, respectively.
% Note that 
% $$
% k_{+} = \sum_{j=-r}^{p} \alpha_j \cos(j \widetilde{k}) + i \sum_{j=-r}^{p} \alpha_j \sin(j \widetilde{k}), \qquad
% k_{-} = -\left(\sum_{j=-r}^{p} \alpha_j \cos(j \widetilde{k}) - i \sum_{j=-r}^{p} \alpha_j \sin(j \widetilde{k})\right)
% $$
% and
% $$
% k_{+}k_{-} = -\left(\left(\sum_{j=-r}^{p} \alpha_j \cos(j \widetilde{k})\right)^2 +  \left(\sum_{j=-r}^{p} \alpha_j \sin(j \widetilde{k})\right)^2\right), 
% $$
\begin{align}\label{eq:dispersion_relation_upwind}
      \widetilde{\omega} = c_s \sqrt{\left(\sum_{j=-r}^{q} \alpha_j \cos(j \widetilde{k})\right)^2 +  \left(\sum_{j=-r}^{q} \alpha_j \sin(j \widetilde{k})\right)^2},
\end{align}
 where we have taken the positive square root, and 
 $
 \widetilde{\omega} = h\omega, \quad \widetilde{k} = hk,
 $
 are the numerical frequency and numerical wave number, respectively. 
 
 Similarly, for the traditional SBP operator we have the numerical dispersion relation
 \begin{align}\label{eq:dispersion_relation_traditional}
      \widetilde{\omega} = c_s   \Big|\sum_{j=1}^{q} 2\gamma_j \sin(j \widetilde{k})\Big|.
\end{align}

For sufficiently small wave number $0\le \widetilde{k} \ll 1$,  the upwind SBP dispersion relation \eqref{eq:dispersion_relation_upwind} and the traditional SBP dispersion relation \eqref{eq:dispersion_relation_traditional} will sufficiently approximate the continuous linear dispersion relation
$\widetilde{\omega} \approx c_s\widetilde{k}$. 
However, in general the numerical dispersion relations \eqref{eq:dispersion_relation_upwind}--\eqref{eq:dispersion_relation_traditional} are nonlinear functions of the wave number $\widetilde{k}$. We consider specifically the $2\pi$-periodic interval $\widetilde{k} \in [-\pi, \pi]$. There are always unresolved numerical modes, which are most prominent at high frequencies,  present in the solution. For the traditional SBP operator, from \eqref{eq:dispersion_relation_traditional}, note in particular that $\widetilde{\omega}(0)=\widetilde{\omega}(\pm \pi) = 0$.  However, for the upwind SBP operators, from \eqref{eq:dispersion_relation_upwind}, we have $\widetilde{\omega}(0)= 0$, $\widetilde{\omega}(\pm \pi) \ne 0$. There are spurious unresolved wave  modes present in the solution. In general, a mode $\widetilde{\omega}(\widetilde{k}_1)$ is spurious if ${\partial \widetilde{\omega}(\widetilde{k})}/{\partial \widetilde{k}}\Big|_{\widetilde{k} = \widetilde{k}_1}<0 $. The negative group velocity  will propagate energy in the opposite (wrong) direction. This also  implies that there are $(\widetilde{k}_0, \widetilde{k}_1)$ with $0\le\widetilde{k}_0 < \widetilde{k}_1$ such that $\omega(\widetilde{k_0}) = \omega(\widetilde{k}_1)$. Note that $\omega(\widetilde{k_0})$ would correspond to a resolved wave mode, and $\omega(\widetilde{k_1})$ is an unresolved spurious wave mode which can be poisonous to numerical simulations.  
%%%

 For simplicity we set the shear wave speed to $c_s = 1$, and because of symmetry we plot the dispersion relations in the positive sub-interval $\widetilde{k} \in [0, \pi]$. The numerical dispersion relations are displayed in Figure \ref{fig:dispersion_all} for upwind operators with interior order of accuracy $2, 3, 4, 5, 6, 7, 8, 9$ and for traditional centered difference SBP operators with interior order of accuracy $2,  4,  6,  8$.
 
 We summarise the significant observations below:
\begin{enumerate}
    \item  In general upwind SBP operators have better dispersion properties than traditional SBP operators based on centered finite difference stencils.
% \item  even-order  ($(2p)$-th) accurate  upwind SBP operators have better dispersion properties than odd-order ($(2p+1)$-th) accurate upwind SBP operators.
 \item  The properties of the numerical dispersion relation  may improve with increasing accuracy.
 \item  The numerical dispersion relation reaches near optimal properties for 6th order accurate upwind SBP operator.
 \item  Beyond the 6th order accurate upwind SBP operator, higher order accuracy does not improve the dispersion properties of the upwind SBP operators.
 \item  even-order upwind SBP operators of order $2, 4, 6$ do support not spurious unresolved wave mode, since ${\partial\widetilde{\omega}(\widetilde{k})}/{\partial \widetilde{k}}\gtrsim 0$ for all $\widetilde{k}$.   All odd-order upwind SBP operator including upwind SBP operator of order  $8$ support spurious unresolved high frequency wave modes, ${\partial\widetilde{\omega}(\widetilde{k})}/{\partial \widetilde{k}}<0$ for some $\widetilde{k}$.
 \item For all traditional SBP operators almost half of the spectrum is dominated by spurious unresolved wave modes with ${\partial\widetilde{\omega}(\widetilde{k})}/{\partial \widetilde{k}}<0$. 
\end{enumerate}
\begin{figure}[H]
    \centering
    \includegraphics[width=0.85\textwidth]{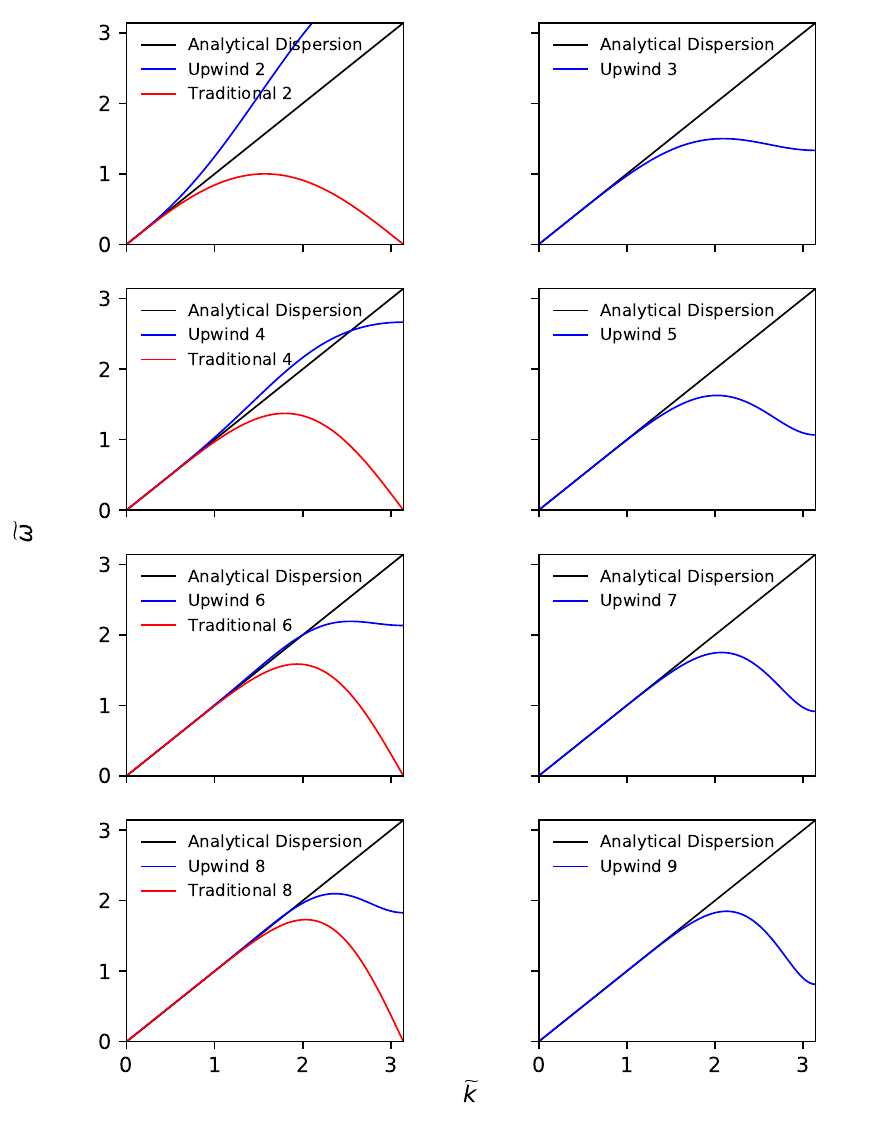}
     \vspace{-1em}
    \caption{Numerical dispersion relations for traditional operators of order $2,4,6,8$ and upwind operators of order $2,3,4,5,6,7,8,9$ compared with the analytical (exact) dispersion relation of the continuous 1D operator. }
    \label{fig:dispersion_all}
\end{figure}
%
% The traditional operators dispersion relations improve with higher order accuracy. 
% As they all have a root at zero, we compare how close the functions are with an $L^2$ norm for varying orders. 
% In contrast, the upwind operators improve the $L^2$ and $L^{\infty}$ norm of the dispersion relations for the internal propagating wave. 
%

\section{Numerical simulations in 3D}
%%%%%
%%%%%
In this section, we present numerical simulations in 3D. The experiments are designed to evaluate accuracy and demonstrate the efficiency of the upwind SBP operators in resolving scattered high frequency waves from complex geometries on marginally resolved meshes. We will also demonstrate parallel efficiency and perfect scaling of our parallel implementation and numerical simulation on Gadi\footnote{Gadi contains a total of 155,000 CPU cores, 567 Terabytes of memory and 640 GPUs.

    3,074 nodes each containing two 24-core Intel Xeon Scalable ‘Cascade Lake’ processors and 192 Gigabytes of memory.
        including 50 nodes each offering 1.5 Terabytes of Intel Optane DC Persistent memory.
    160 nodes each containing four Nvidia V100 GPUs and two 24-core Intel Xeon Scalable 'Cascade Lake' processors.
    Linking the storage systems and Gadi is Mellanox Technologies' latest generation HDR InfiniBand technology in a Dragonfly+ topology, capable of transferring data at up to 200 Gb/s.
    The storage sub-systems are NetApp enterprise class storage arrays, linked together in a DDN Lustre parallel file system.
    Altair’s PBSPro software optimises job scheduling and workload management.
    Gadi uses the latest version of the CentOS 8 operating system.
}, Australia's newest supercomputer. We will use WaveQLab3D's default implementation \cite{DuruandDunham2016} of traditional SBP operator of interior order $6$ and upwind operators \cite{Mattsson2017} of interior order $2,3,4,5,6,7,8,9$. Then we will compare the numerical accuracy of traditional SBP operator of interior order $6$ and upwind operators of interior order $6$. These SBP operators are closed with 3rd order accurate stencils have 4th order global accuracy. The solutions are integrated in time using $4$-th order acurate low-storage Runge-Kutta time stepping scheme \cite{CarpenterKennedy1994}. %\todo{need to fill in the citation}

To verify accuracy, we compute numerical solution of the 3D benchmark problem LOH1 \cite{Kristekova_etal2009, Kristekova_etal2006}, which has a semi-analytic-solution and compare results. We then show the potential of the upwind scheme by simulating a large scale 3D wave propagation problem with complex geometry at different grid resolutions. For the $6$-th order accurate operators, the upwind scheme provides comparable results on a marginally resolved grid to the traditional SBP operator on a finer grid, thereby improving computational efficiency for 3D numerical seismic wave simulations. 

In all simulations, we have used the PML \cite{DuruKozdonKreiss2016} to prevent artificial numerical reflections, from the computational boundaries, from contaminating the numerical simulations. A stable implementation of the PML for the 3D  elastic wave equation  IBVP using the upwind SBP operators is a non-trivial task. The details of the numerical treatment of the PML using upwind SBP operators will be reported in a forthcoming paper.

%Two different systems, Gadi in NCI and Supermuc-NG in LRZ, are used for the two numerical test problems, respectively.
%Our test system for the Zugspitze Simulation is Supermuc-NG, a (Lenovo NextScale) system with an overall peak performance of 26.9 PetaFlop/s, listed in position 13   of the TOP500\footnote{See https://www.top500.org.} list as of June 2020. 
%Each of the 6,336 thin nodes is loaded with 48 cores and 96 GByte memory, the same number of cores is loaded on fat nodes but associated with 768 GByte memory. The Waveqlab3D\todo{Feng: not seen mentioned in previous sections, do we need to mention about it?} software is compiled by Intel C/C++ compiler \todo{found this online but not sure about the version} and is linked to \todo{Which mpi flavour is used in supermuc-ng? MPICH2?}

%\the\textwidth
%\the\textheight
%370.38374

\subsection{Layer over  homogeneous half-space (LOH1) 3D benchmark problem}
%%%
%%%
To verify and assess the numerical accuracy of our upwind method, we choose the 3D seismological benchmark problem,  Layer Over Homogeneous Half-space (LOH1) \cite{Seismowine, Kristekova_etal2009, Kristekova_etal2006} benchmark problem, a Seismic wave Propagation and Imaging in Complex media (SPICE) validation code.
The LOH1 benchmark has a planar free surface and an internal interface between a thin low velocity (soft) upper-layer and high velocity (hard) lower crust,  see  Figure \ref{fig:loh1_setup}.
The material properties for the soft upper-layer and hard lower-half-space are 
\begin{align*}
    \rho = 2600  \1_{ \{ (x_0,y_0,z_0) \ | \ x_0 \leq 1 \} } + 2700 \1_{ \{ (x_0,y_0,z_0) \ | \ x_0 > 1 \} },\\
    c_p = 4000  \1_{ \{ (x_0,y_0,z_0) \ | \ x_0 \leq 1 \} } + 6000 \1_{ \{ (x_0,y_0,z_0) \ | \ x_0 > 1 \} },\\
    c_s = 2000  \1_{ \{ (x_0,y_0,z_0) \ | \ x_0 \leq 1 \} } + 3343 \1_{ \{ (x_0,y_0,z_0) \ | \ x_0 > 1 \} }.
\end{align*}
The wave-speeds have units $ m/s$ and the density $\rho$ has units $kg/m^3$. Note that $\rho$, $c_p$ and $c_s$ are discontinuous in the medium.
%%%
%%%
The benchmark considers homogeneous initial conditions on the solution $\bm{Q}$ with the double-couple moment tensor point source 
%%%
%%%
\begin{align}\label{eq:momententor_pointsource}
    \bm{f} (x,y,z,t) = \bm{M} \boldsymbol{\delta}_{(x_p,y_p,z_p)} (x,y,z) g(t), && g(t) \coloneqq \frac{t}{T^2} \exp(-t/T) , \quad T = 0.1 ~\ s,
\end{align}
%%%
%%%
located  $2$ km at depth $(x_p,y_p,z_p) = (2,0,0)$, where $\boldsymbol{\delta}$ is the 3D Dirac distribution  and $\bm{M} = \left( 0, 0, 0 , 0,0,0,0,0,M_0\right)$ where $M_0 = 10^{18} $ Nm is  the moment magnitude.
%%%%
Note that the moment tensor source \eqref{eq:momententor_pointsource} is spatially singular. Our numerical implementation approximates the singular source to high order accuracy as in \cite{Petersson_etal2016}.

%%%%
In the $z$ and $y$ directions, the domain of the problem is unbounded. 
In the positive $x$ direction (in-towards the Earth), the domain is also unbounded with the Earth's surface $x = 0$ having the free surface, traction-free, boundary condition $\left(\bm{T}_x, \bm{T}_y, \bm{T}_z\right) = 0$.
%  \begin{align}
%      \left(\bm{T}_x, \bm{T}_y, \bm{T}_z\right) = 0, && \eta, \xi \in \{x,y,z \} , \ (x_0,y_0,z_0) \in F_{q,0}. 
% \end{align}

The SPICE code validation project \cite{Seismowine} has suggested to use large enough computational model, namely $\Omega_{L} =[0, 34]\times [-26,32]^2$, so as the seismograms in the receivers do not contain waves, which are due to artificial boundaries of the model.
This would correspond to the computational domain of volume  114376~km$^3$.
%%%%
%%%%
To deal with the unbounded domain, we use the PML \cite{DuruKozdonKreiss2016} to absorb outgoing waves and prevent artificial reflections from the bounded computational domain. 
The PML allows us to sufficiently limit the modelling space to be $ \Omega = [0,6] \times [-5,15]^2$ with only a few grid points around the computational boundaries where the PML is active. Please see also  Figure \ref{fig:loh1_setup}.  Our computational domain $\Omega$ is only $2400$~km$^3$ in volume, and amounts to $\% 2.0983$ of the suggested large domain $\Omega_{L}$, thus saving as much as $\% 97.9017$ of  the required computational resources. Although the PML involves auxiliary variables and equations to be stored and evolved, however, the extra computational cost for evolving the auxiliary variables is very insignificant since they are only active inside the thin PML absorbing layer.

To deal with the discontinuity of the medium at $x= 1$~km  we decompose the domain into two sub-blocks, with block1: $0\le x\le 1 $~km and block2: $1\le x\le 6$~km, discretise each sub-block and couple the solutions across the interface weakly using penalties. In the thin low velocity (soft) upper-layer, block1, we use the grid size $h_x = 62.5$ m, in the $x$-direction and set the uniform grid size $h_y = h_z = h= 100$ m in the $y$- and $z$-direction. In the the hard lower-half-space we use the  uniform grid size $h_x=h_y = h_z = h = 100$ m in all directions. At this mesh resolution, the discritisation generates about 25 million degrees of freedom (DoF) for the evolving unknown vector field, and about 39 million DoF needed to store the mesh, material parameters, the Jacobian and metric parameters. 

%We consider $10$ points-per-$km$ to be a fully-resolved computational grid, whilst the sparse resolution of $5$ points-per-$km$ is a marginally-resolved grid. 

\newcommand{\ytk}{5.8}
\newcommand{\xtk}{5.8}
\newcommand{\ztk}{3.4}

\newcommand{\xrl}{1.7}
\newcommand{\yrl}{2.1}
\newcommand{\xrr}{3.7}
\newcommand{\yrr}{4.1}

\newcommand{ \lpml}{0.24}

\begin{figure}[H]
    \centering

\begin{tikzpicture}

%big blue box
\coordinate (O) at (0,0,0);
\coordinate (A) at (0,\ztk,0);
\coordinate (B) at (0,\ztk,\xtk);
\coordinate (C) at (0,0,\xtk);
\coordinate (D) at (\ytk,0,0);
\coordinate (E) at (\ytk,\ztk,0);
\coordinate (F) at (\ytk,\ztk,\xtk);
\coordinate (G) at (\ytk,0,\xtk);

%yellow layer
\coordinate (A2) at (0,3.3,0);
\coordinate (B2) at (0,3.3,\xtk);
\coordinate (E2) at (\ytk,3.3,0);
\coordinate (F2) at (\ytk,3.3,\xtk);

%red box top
\coordinate (Ar) at (\yrl,\ztk,\xrl);
\coordinate (Br) at (\yrl,\ztk,\xrr);
\coordinate (Er) at (\yrr,\ztk,\xrl);
\coordinate (Fr) at (\yrr,\ztk,\xrr);

%red box bottom
\coordinate (Arb) at (\yrl,2.8,\xrl);
\coordinate (Brb) at (\yrl,2.8,\xrr);
\coordinate (Erb) at (\yrr,2.8,\xrl);
\coordinate (Frb) at (\yrr,2.8,\xrr);

%red box separating layer
\coordinate (Arl) at (\yrl,3.3,\xrl);
\coordinate (Brl) at (\yrl,3.3,\xrr);
\coordinate (Erl) at (\yrr,3.3,\xrl);
\coordinate (Frl) at (\yrr,3.3,\xrr);

 \draw[blue,fill=blue!10,opacity=0.4] (O) -- (C) -- (G) -- (D) -- cycle;% Bottom Face
 \draw[blue,fill=blue!10,opacity=0.4] (O) -- (A) -- (E) -- (D) -- cycle;% Back Face
 \draw[blue,fill=blue!10,opacity=0.4] (O) -- (A) -- (B) -- (C) -- cycle;% Left Face
\draw[blue] (D) -- (E) -- (F) -- (G) -- cycle;% Right Face
 \draw[blue] (C) -- (B) -- (F) -- (G) -- cycle;% Front Face
 \draw[blue] (A) -- (B) -- (F) -- (E) -- cycle;% Top Face

\draw[blue,fill=yellow!40,opacity=0.5] (A2) -- (B2) -- (F2) -- (E2) -- cycle;% Top Face
\draw[blue,fill=yellow!40,opacity=0.7] (Arl) -- (Brl) -- (Frl) -- (Erl) -- cycle;% Top Face

\draw[blue,fill=red!40,opacity=0.7] (Ar) -- (Br) -- (Fr) -- (Er) -- cycle;% Top Face
\draw[blue,fill=red!40,opacity=0.7] (Arb) -- (Brb) -- (Frb) -- (Erb) -- cycle;% Top Face

\draw[blue] (Ar) -- (Arb) ; 
\draw[blue] (Br) -- (Brb) ; 
\draw[blue] (Fr) -- (Frb) ; 
\draw[blue] (Er) -- (Erb) ; 

\node[circle, draw] at (2.6,3.2,3.2) () {S};

\node[draw] at (2.6+0.7348,3.4,3.2- 0.7348) () {6};
%\node[draw] at (2.6+0.5764,3.4,3.2-0.8647 ) () {6};

\draw[black, -stealth] (2.6,3.4,3.2) -- (2.6,3.4,0) ; 
\draw[black, -stealth] (2.6,3.4,3.2) -- (5.8,3.4,3.2) ; 
\draw[black, -stealth] (2.6,3.4,3.2) -- (5.8,3.4,0) ;

%% Following is for debugging purposes so you can see where the points are
%% These are last so that they show up on top
% \foreach \xy in {O, A, B, C, D, E, F, G}{
%     \node at (\xy) {\xy};
% }

\begin{scope}[shift = {(8,1,0)}, scale = 2.8]
\begin{scope}[shift = {(-\yrl,-2.8,-\xrl)}]
\coordinate (O) at (0,0,0);
\coordinate (A) at (0,\ztk,0);
\coordinate (B) at (0,\ztk,\xtk);
\coordinate (C) at (0,0,\xtk);
\coordinate (D) at (\ytk,0,0);
\coordinate (E) at (\ytk,\ztk,0);
\coordinate (F) at (\ytk,\ztk,\xtk);
\coordinate (G) at (\ytk,0,\xtk);

\coordinate (A2) at (0,3.3,0);
\coordinate (B2) at (0,3.3,\xtk);
\coordinate (E2) at (\ytk,3.3,0);
\coordinate (F2) at (\ytk,3.3,\xtk);

\coordinate (sAr) at (\yrl,\ztk,\xrl);
\coordinate (sBr) at (\yrl,\ztk,\xrr);
\coordinate (sEr) at (\yrr,\ztk,\xrl);
\coordinate (sFr) at (\yrr,\ztk,\xrr);

\coordinate (sArb) at (\yrl,2.8,\xrl);
\coordinate (sBrb) at (\yrl,2.8,\xrr);
\coordinate (sErb) at (\yrr,2.8,\xrl);
\coordinate (sFrb) at (\yrr,2.8,\xrr);

\coordinate (sArl) at (\yrl,3.3,\xrl);
\coordinate (sBrl) at (\yrl,3.3,\xrr);
\coordinate (sErl) at (\yrr,3.3,\xrl);
\coordinate (sFrl) at (\yrr,3.3,\xrr);

%  \draw[blue,fill=red!40,opacity=0.7] (O) -- (C) -- (G) -- (D) -- cycle;% Bottom Face
%  \draw[blue,fill=red!40,opacity=0.7] (O) -- (A) -- (E) -- (D) -- cycle;% Back Face
%  \draw[blue,fill=red!40,opacity=0.7] (O) -- (A) -- (B) -- (C) -- cycle;% Left Face
% \draw[blue] (D) -- (E) -- (F) -- (G) -- cycle;% Right Face
%  \draw[blue] (C) -- (B) -- (F) -- (G) -- cycle;% Front Face
%  \draw[blue] (A) -- (B) -- (F) -- (E) -- cycle;% Top Face

% \draw[blue,fill=yellow!40,opacity=0.5] (A2) -- (B2) -- (F2) -- (E2) -- cycle;% Top Face
 \draw[blue,fill=yellow!40,opacity=0.7] (sArl) -- (sBrl) -- (sFrl) -- (sErl) -- cycle;% Top Face

\draw[blue,fill=red!40,opacity=0.7] (sAr) -- (sBr) -- (sFr) -- (sEr) -- cycle;% Top Face
\draw[blue,fill=red!40,opacity=0.7] (sArb) -- (sBrb) -- (sFrb) -- (sErb) -- cycle;% Top Face

\draw[blue] (sAr) -- (sArb) ; 
\draw[blue] (sBr) -- (sBrb) ; 
\draw[blue] (sFr) -- (sFrb) ; 
\draw[blue] (sEr) -- (sErb) ;

% %\node[draw] at (2.6+0.5764,3.4,3.2-0.8647 ) () {6};

\draw[black, -stealth,opacity=0.7] (2.6,3.4,3.2) -- (2.6,3.4,\xrl) ; 
\draw[black, -stealth,opacity=0.7] (2.6,3.4,3.2) -- (\yrr,3.4,3.2) ; 
\draw[black, -stealth,opacity=0.7] (2.6,3.4,3.2) -- (\yrr,3.4,\xrl) ;

\draw[black, dashed,opacity=0.7] (\yrl,3.4,\xrr - \lpml) -- (\yrr,3.4,\xrr - \lpml) ; 
\draw[black, dashed,opacity=0.7] (\yrl,3.4,\xrl + \lpml) -- (\yrr,3.4,\xrl + \lpml) ;

\draw[black, dashed,opacity=0.7] (\yrl + \lpml,3.4,\xrr) -- (\yrl + \lpml,3.4,\xrl) ; 
\draw[black, dashed,opacity=0.7] (\yrr - \lpml,3.4,\xrl) -- (\yrr - \lpml,3.4,\xrr) ;

\draw[black, dashed,opacity=0.7] (\yrl,2.8 + \lpml,\xrr - \lpml) -- (\yrr,2.8 + \lpml,\xrr - \lpml) ; 
\draw[black, dashed,opacity=0.7] (\yrl,2.8 + \lpml,\xrl + \lpml) -- (\yrr,2.8 + \lpml,\xrl + \lpml) ;

\draw[black, dashed,opacity=0.7] (\yrl + \lpml,2.8 + \lpml,\xrr) -- (\yrl + \lpml,2.8 + \lpml,\xrl) ; 
\draw[black, dashed,opacity=0.7] (\yrr - \lpml,2.8 + \lpml,\xrl) -- (\yrr - \lpml,2.8 + \lpml,\xrr) ;

\draw[black, dashed,opacity=0.7] (\yrl,2.8,\xrr - \lpml) -- (\yrr,2.8,\xrr - \lpml) ; 
\draw[black, dashed,opacity=0.7] (\yrl,2.8,\xrl + \lpml) -- (\yrr,2.8,\xrl + \lpml) ;

\draw[black, dashed,opacity=0.7] (\yrl + \lpml,2.8,\xrr) -- (\yrl + \lpml,2.8,\xrl) ; 
\draw[black, dashed] (\yrr - \lpml,2.8,\xrl) -- (\yrr - \lpml,2.8,\xrr) ;

\draw[black, dashed,opacity=0.7] (\yrl + \lpml,2.8,\xrr - \lpml) -- (\yrl + \lpml,3.4,\xrr - \lpml);
\draw[black, dashed,opacity=0.7] (\yrl,2.8,\xrr - \lpml) -- (\yrl,3.4,\xrr - \lpml);
\draw[black, dashed,opacity=0.7] (\yrl + \lpml,2.8,\xrr) -- (\yrl + \lpml,3.4,\xrr);

\draw[black, dashed,opacity=0.7] (\yrl + \lpml,2.8,\xrl + \lpml) -- (\yrl + \lpml,3.4,\xrl + \lpml);
\draw[black, dashed,opacity=0.7] (\yrl,2.8,\xrl + \lpml) -- (\yrl,3.4,\xrl + \lpml);
\draw[black, dashed,opacity=0.7] (\yrl + \lpml,2.8,\xrl) -- (\yrl + \lpml,3.4,\xrl);

\draw[black, dashed,opacity=0.7] (\yrr - \lpml,2.8,\xrr - \lpml) -- (\yrr - \lpml,3.4,\xrr - \lpml);
\draw[black, dashed,opacity=0.7] (\yrr,2.8,\xrr - \lpml) -- (\yrr,3.4,\xrr - \lpml);
\draw[black, dashed,opacity=0.7] (\yrr - \lpml,2.8,\xrr) -- (\yrr - \lpml,3.4,\xrr);

\draw[black, dashed,opacity=0.7] (\yrr - \lpml,2.8,\xrl + \lpml) -- (\yrr - \lpml,3.4,\xrl + \lpml);
\draw[black, dashed,opacity=0.7] (\yrr,2.8,\xrl + \lpml) -- (\yrr,3.4,\xrl + \lpml);
\draw[black, dashed,opacity=0.7] (\yrr - \lpml,2.8,\xrl) -- (\yrr - \lpml,3.4,\xrl);

 \node[circle, draw] at (2.6,3.2,3.2) () {S};

 \node[draw] at (2.6+0.7348,3.4,3.2- 0.7348) () {6};

\end{scope}
\end{scope}

\draw[dotted, gray] (sAr) -- (Ar) ;
\draw[dotted, gray] (sBr) -- (Br) ;
\draw[dotted, gray] (sFr) -- (Fr) ;
\draw[dotted, gray] (sEr) -- (Er) ;

\draw[dotted, gray] (sArb) -- (Arb) ;
\draw[dotted, gray] (sBrb) -- (Brb) ;
\draw[dotted, gray] (sFrb) -- (Frb) ;
\draw[dotted, gray] (sErb) -- (Erb) ;

\end{tikzpicture}
    \caption{LOH1 problem setup (to scale) with the upper and lower block separated by the yellow interface. 
    In blue is the suggested computational domain, in red is the computational domain we use with the PML. 
    The PML regions are sectioned along the boundary of the enlarged red block. 
    The point source is labelled S, and station 6 is marked on the Earth surface. The red region occupies approximately $2.1 \%$ of the volume of the blue region. 
Due to the efficient absorption properties of the PML, the computational load for this problem is significantly reduced.}
    \label{fig:loh1_setup}
\end{figure}
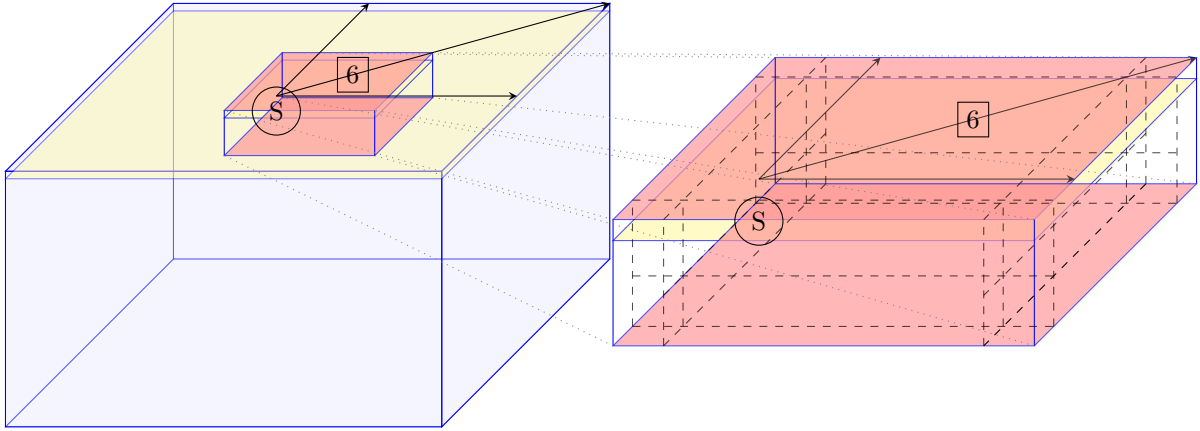

To evaluate the results of our schemes, simulated data is compared at the receiver, labelled station $6$ and $9$ in the LOH1 documentation \cite{Seismowine}, located on the face $F_{q,0}$ at the locations $(0,7.348,7.348)$ and $(0,8.647,8.647)$ relative to the epicentre $(0,0,0)$. 

In Figure \ref{fig:LOH1_6_all}, the numerical solutions are compared with the exact solution for upwind operators of order $2,3,4,5,6,7,8,9$. 
Note that the numerical solutions converge to the exact solution as the order of accuracy increases. At this mesh resolution, the accuracy of the solution becomes nearly optimal for the 6th order accurate upwind SBP operator. It would have been expected that the highest order accurate operator, upwind order 9, would yield the smallest error when compared with the exact solution. However, this is not the case. Beyond the 6th order accurate upwind SBP operator, the higher order accurate SBP operators do not necessarily improve the accuracy of numerical solution for the upwind SBP operators. This is consistent with the numerical dispersion relation analysis perform in Section \ref{sec:dsipersion_relation}. See also the numerical dispersion plots displayed in Figure \ref{fig:dispersion_all}.
%\newpage
\begin{figure}[H]
    \centering
    \includegraphics[width=\textwidth]{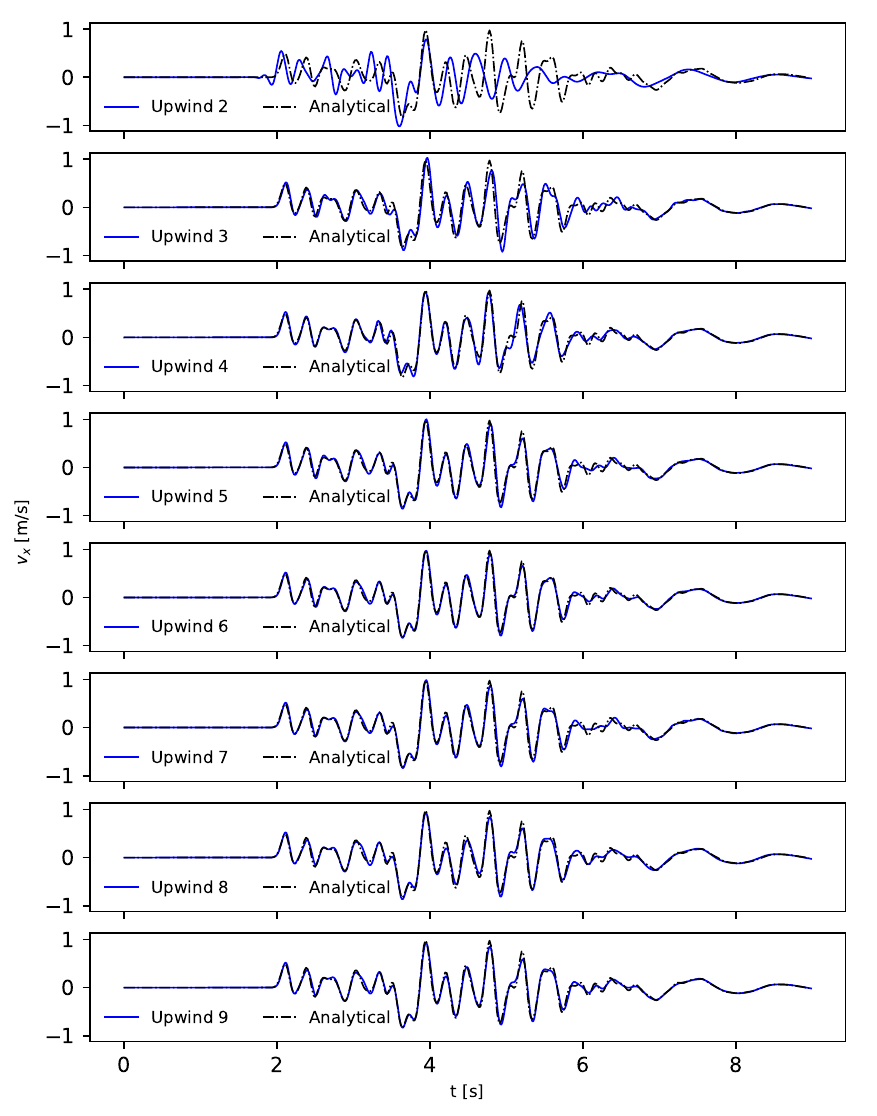}
     \vspace{-1em}
    \caption{Time history of the particle velocity vector at ``Station 6" $(x_r, y_r, z_r)=(0,7.348,7.348)$ upwind SBP FD operators of order $2,3,4,5,6,7,8,9$ compared with the analytical (exact) solution.}
    \label{fig:LOH1_6_all}
\end{figure}

In Figures \ref{fig:LOH1_6}--\ref{fig:LOH1_6_50m} we compare the analytical solution with specifically the numerical seismograms for the 6th order upwind SBP operator and the 6th order traditional SBP operator, at two mesh resolutions.
%%%%
In Figure \ref{fig:LOH1_6}, it can be seen that both the upwind and traditional schemes have relatively comparable accuracy for the benchmark at $100$ m grid resolution. 
Both of these schemes approximate the analytical solution, with certain peaks in the data containing discrepancies between the numerical and analytical solutions, such as the peak at $\approx 6s$ in the $V_x$ data at station 6. We note however that the upwind SBP numerical seismogram matches the analytical solution better than the traditional SBP numerical seismogram in most parts of the solution. 

Next we refine the mesh by doubling the number of grid points in each direction. 
That is, in the thin low velocity (soft) upper-layer, block1, we use the grid size $h_x = 31.25$ m, in the $x$-direction and set the uniform grid size $h_y = h_z = h= 50$ m in the $y$- and $z$-direction. In the the hard lower-half-space, block2, we use the  uniform grid size $h_x=h_y = h_z = h = 50$ m in all directions.
Upon refinement, as seen in Figure \ref{fig:LOH1_6_50m},  the numerical solutions have converged and the numerical seismograms are visually identical with the analytical seismogram for both upwind SBP FD operator and the traditional SBP FD operator.  However, doubling the number of grid points in each direction increases the computational cost by a factor of $16$ for the 3D numerical simulation. A similar result is seen in Figure \ref{fig:LOH1_9} in the appendix.
The quantitative envelop misfit and phase misfit \cite{Kristekova_etal2009, Kristekova_etal2006} for these stations are below $\% 0.1$ for upwind operator and below $\% 0.5$ for the traditional operators. 
\begin{figure}[H]
    \centering
    \includegraphics{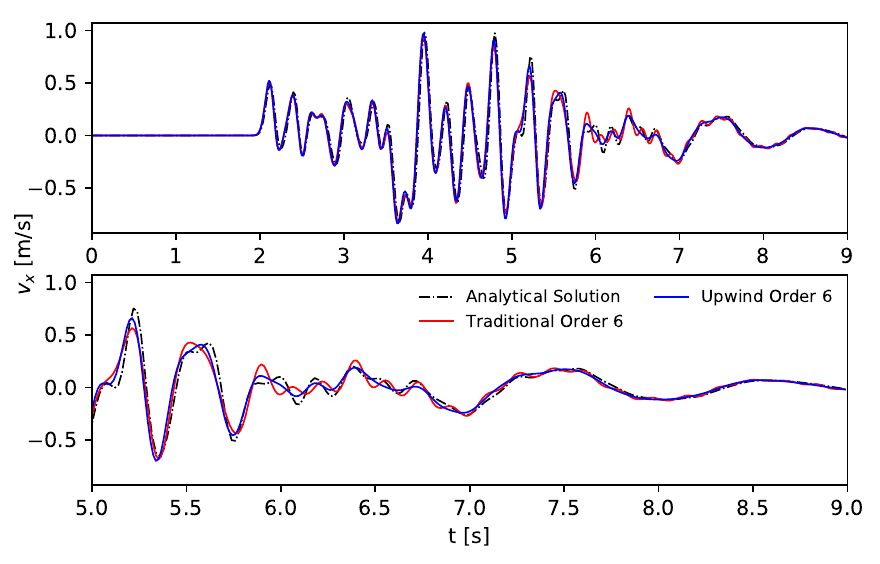}
   \caption{Time history of the particle velocity vector at ``Station 6" $(x_r, y_r, z_r)=(0,7.348,7.348)$ upwind and traditional operators of order $6$ compared with the analytical (exact) solution at $100m$ resolution.}
    \label{fig:LOH1_6}
\end{figure}

\begin{figure}[H]
    \centering
    \includegraphics{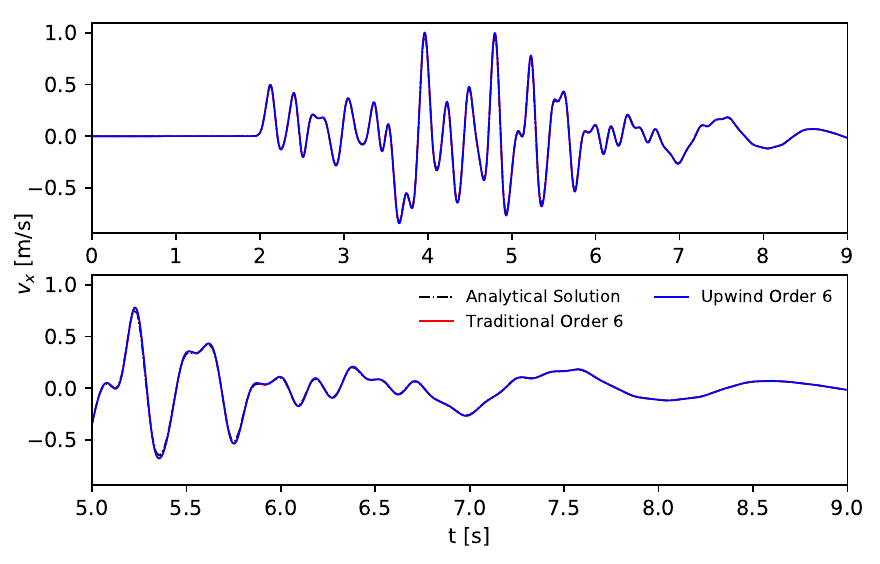}
    \caption{Time history of the particle velocity vector at ``Station 6" $(x_r, y_r, z_r)=(0,7.348,7.348)$ upwind and traditional operators of order $6$ compared with the analytical (exact) solution at $50m$ resolution.}
    \label{fig:LOH1_6_50m}
\end{figure}
%\todo{Will it be possible to redo the higher resolution simulations by doubling the grid points in lower resolution set-up? That is Block1, we use grid size $h_x = 1/32=62.5/2=31.25$ m, in the $x$-direction and set the uniform grid size $h_y = h_z = h= 50$ m in the $y$- and $z$-direction. In the the hard lower-half-space we use the  uniform grid size $h_x=h_y = h_z = h = 50$ m in all directions.}
%%%

% \begin{figure}[H]
%     \centering
%     \includegraphics[width=0.8\textwidth]{LOH1_6_100_100.pdf}
%     \includegraphics[width=0.8\textwidth]{LOH1_6_50_50.pdf}
%      %\includegraphics{LOH1_9_100_100.pdf}
%      \vspace{-2em}
%     \caption{Time history of the particle velocity vector at ``Station 6" $(x_r, y_r, z_r)=(0,7.348,7.348)$ with two levels of uniform mesh refinements $h = 100$ m and $h =50$ m}
%     \label{fig:LOH1_6}
% \end{figure}

% \begin{figure}[H]
%     \centering
%     \includegraphics{LOH1_6_50_50.pdf}
%     \includegraphics{LOH1_9_50_50.pdf}
%     \caption{5050}
%     \label{fig:LOH1_50}
% \end{figure}

% There are two boundary conditions in problem. 
% Firstly there is a free surface on the upper crust, that is for every $(x_0,y_0,z_0) \in F_{q,1}$ we have 
% \begin{align}
%     \sigma_{ \xi \eta  } (x_0,y_0,z_0) = 0, && \eta, \xi \in \{x,y,z \} . 
% \end{align}
% Secondly, 

\subsection{Large-scale numerical simulations in a 3D complex geometry}
%%%
We will now present numerical simulations in a complex geometry, with a geologically constrained complex non-planar free-surface topography.  Zugspitze is the tallest mountain in Germany, lying in the Wetterstein mountain range. The topography of this region is complex, with large variations in altitude across the Earths surface. We extracted the topography data of the Zugspitze region from the high resolution Alpine topography \cite{Copernicus}. 
Accurate and stable simulation of seismic wave propagation in this region is a computationally expensive task, the main reason being the high frequency wave modes generated by scattering  from the complex non-planar topography.  Effective numerical simulations will require an efficient HPC and parallel numerical elastic wave solver that scales perfectly with increasing supercomputing resources. We will present first some scaling tests to demonstrate efficient parallel implementation of the upwind SBP FD operators for the numerical simulations of 3D elastic waves in complex geometries, and will proceed later to the numerical simulations of seismic waves in the Zugspitze region.
%%%%

\paragraph{The Zugspitze setup}
%%%%
The modelling domain is $\Omega = \bigcup_{y,z \in [-5,85] } [\widehat{X}(x,y),80] \times [-5,85]^2 $ with the $x$-co-ordinate being positive in-towards the Earth, like our previous example, and $\widehat{X}(y,z)$ parameterising the Earth's surface. 
Our primary objective is to understand how different upwind SBP FD operators and the traditional SBP FD operator will resolve high frequency scattered wave modes generated by the complex free-surface topography without introducing spurious wave modes in the solution. To isolate scattering from complex geometries,  the material parameters of  the region is assumed to be constant and given by $\rho = 2700 ~\ kg/m^3$, $c_p = 6000 ~\ m/s$ and $c_s = 3464 ~\ m/s$. Therefore scattered wave fields seen in our simulations are primarily the effects of complex non-planar topography of the Zugspitze region.
 The location of the moment tensor point source is at $(10,10,10)$, that is $10$ km at depth.  As per the LOH1 experiment, we have free surface boundary conditions at the complex topography $\widehat{X}(y, z)$ and use the PML  to prevent artificial reflections  from the computational boundaries from contaminating the solution.

\subsubsection{Scaling tests}
For large scale numerical simulations of PDEs, it is imperative that the parallel numerical software is efficient and scalable. 
%Due to the large-size of the problem, it is advisable to gain heuristic estimates on the resources required for a full product run. 
Here, we will perform strong scaling tests to verify the efficiency of our parallel implementation of high order upwind SBP operators for large-scale elastic wave simulations in 3D  geometrically complex elastic solids. We will demonstrate strong perfect scaling  with increasing supercomputing resources. 
%%%
To conduct these scaling tests we consider the Zugspitze setup as described above with $h=200$~m uniform resolution of the complex topography and the model volume. 
%%%
At this mesh resolution, the discritisation generates about $1$~billion DoF for the evolving unknown vector field, and about $1.45$~billion DoF needed to store the mesh, material parameters, the Jacobian and metric parameters. We run the simulation for $1$~s using the same time step for all order of accuracy. In these runs, the I/O is disabled as to only evaluate the performance of the numerical solver and the time-to-solution. 

The HPC implementation of WaveQLab is parallelised with MPI\footnote{The software is compiled on Gadi with Intel Fortran Compiler 2019.5.281 and is linked to Intel MPI 2019.5.281}.   The time-to-solution can be split into the CPU-time used for floating point operations to evaluate the FD stencil and the MPI-time used to communicate ghost nodes shared between adjacent processors. The CPU-time will be proportional to the width of the finite difference stencil, $(r+q + 1)$ for upwind operator and $(2q + 1)$ for the traditional operator, and the MPI-time which is proportional to the number of ghost $q$, where $r$ and $q$ are listed in Tables \ref{tab:upwind}-\ref{tab:traditional}. Note in particular, while the traditional and upwind FD SBP operators of order $6$ have the same stencil width, $7$, at any inter-processor boundary the traditional FD SBP operator has $q=3$ ghost nodes and  the  upwind FD SBP operator has $q=4$ ghost nodes.

In the scaling tests, we keep the problem size constant (about $1$~billion DoF for the evolving unknown vector field, about $1$~billion DoF for the right hand side and about $1.45$~billion DoF needed to store the mesh, material parameters, the Jacobian and metric parameters) while increasing the number of nodes, from $1,2,4,8, 16$ to $32$ nodes.  On Gadi super-computing infrastructure there are $48$ compute processor cores per node. So the number of processors increases from $48, 96, 192, 384, 768$ to $1536$ CPU cores.
In Figure \ref{fig:speedup}, we present the scaling and speed-up results for each of the upwind operators. The scaling plot shows the wall-clock time (time-to-solution) against number of nodes. A perfect strong scaling pattern is observed as the wall-clock time is halved at each time doubling the number of nodes. In Figure \ref{fig:speedup_o6} we compare the wall-clock time and the scaling plots for the traditional SBP FD operator of order 6 and the upwind SBP FD  of order 6. Note that the scaling plots as well as the supercomputing resources consumed by both operators are  very comparable.

%This is presented and modelled on a $\log-\log$ scale plot. The speed-up ratio is the ratio of one compute node to $n$ compute nodes and represents the speed gain factor with increasing resources.  The speed-up ratio is plotted and modelled with a linear trend.  Ideally, the scaling plot follows a linear trend with slope one implying that time scales almost equally with resources used.  We present and model the scaling data as linear trends with zero intercept, and compare this with the ideal linear trend.   Figure \ref{fig:speedup} shows these results. 

 \begin{figure}[H]
    \centering
    \includegraphics[width=\textwidth]{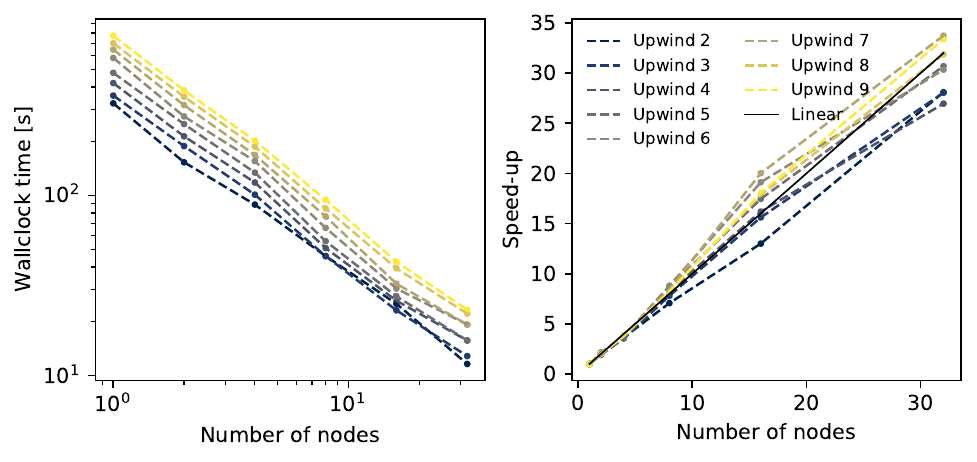}
    \caption{Scaling and speed-up plots for the upwind SBP FD operators of order $2,3,4,5,6,7,8,9$. For the scaling plot, note the log-scale in both axes.} 
    \label{fig:speedup}
\end{figure}

 \begin{figure}[H]
    \centering
    \includegraphics[width=\textwidth]{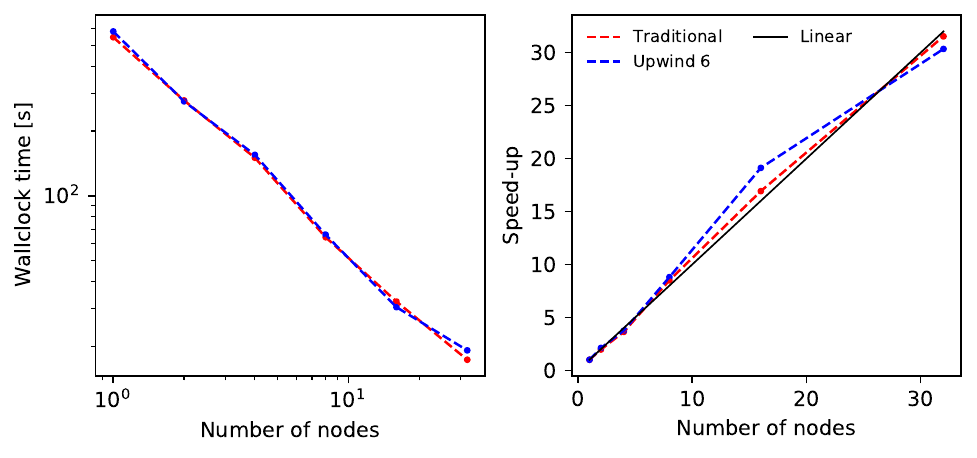}
    \caption{Scaling and speed-up plots for the order 6 traditional and upwind SBP FD operators. In red is the traditional SBP FD operator of order 6, and in blue is the upwind SBP FD  of order 6. For the scaling plot, note the log-scale in both axes.} 
    \label{fig:speedup_o6}
\end{figure}

%\todo{the y-label in the speed-up and scaling plots is swapped. The first figure should be the scaling plot and the second figure is the speed-up plot. You may want to use first plot: ylabel('wall-clock time [s]'), xlabel('Number of nodes') and for the second plot: ylabel('Speed-up'), xlabel('Number of nodes')}

We will now formulate a simple theoretical model for the scaling tests and speed-up.
Let $T_{N}$ denote the wall-clock time it takes to reach the final time for $N$ number of compute  nodes, and $S_N$ denotes the achieved speed-up. Ideally we have
\begin{align}
T_{N} := \frac{T_1}{N} \iff \log_{10}(T_{N}) = - \log_{10}(N) + \log_{10}(T_{1}), \qquad S_{N} := \frac{T_1}{T_{N}} = N,
\end{align}
where $T_1$ is the wall-clock time it takes to reach the final time on single node.
Thus, the experimental results are modelled by
\begin{align}
    \log_{10}(T_{N}) = m_1 \log_{10}(N) + b_1, \qquad S_N = m_2 N,
\end{align}
where $b_1 = \log_{10}(T_{1})$, and the model parameters $m_1$ and $m_2$ are determined experimentally. Note that in the theoretically ideal case we have $m_1 = -1$ and $m_2 = 1$. 
The  experimentally determined  constant parameters $m_1$ and $m_2$ are given in Table \ref{tab:scaling} for each of the upwind operators. 
%Ideally, we have that $W = W_1/N$, where $W_1 \coloneqq W(1)$ and so the ideal $m_1$ is negative one, and the $b_1$ values are the wall-clock time for one node to complete the task (on the $\log$ scale). 

\begin{table}[!h]
\begin{center}
 \begin{tabular}{||l c c c c c c c c||} 
 \hline
 & Upwind 2 & Upwind 3 & Upwind 4 & Upwind 5 & Upwind 6 & Upwind 7 & Upwind 8 & Upwind 9  \\ [0.5ex] 
 \hline\hline
 $m_1$ &
%  -0.91338 &
% -0.99596 &
% -1.00968 &
% -1.04132 &
% -1.05714 &
% -1.07016 &
% -1.03803 &
% -1.03712 \\
-0.93857 &
-0.97933 &
-0.97357 &
-1.01432 &
-1.01087 &
-1.0402 &
-1.01757 &
-1.02553 \\
$b_1$ &
% 2.49118 &
% 2.57124 &
% 2.6373 &
% 2.70496 &
% 2.77582 &
% 2.83025 &
% 2.86248 &
% 2.90115 \\
2.50129 &
2.56456 &
2.62281 &
2.69412 &
2.75725 &
2.81822 &
2.85427 &
2.8965 \\
 \hline
 $m_2$ &
% 0.83396 &
% 0.97129 &
% 1.0098 &
% 1.07785 &
% 1.16268 &
% 1.19789 &
% 1.09178 &
% 1.09862 \\
0.86685 &
0.90015 &
0.88412 &
0.98859 &
1.00183 &
1.08999 &
1.01965 &
1.05694 \\
 [1ex] 
 \hline
\end{tabular}
\end{center}
    \caption{Regression co-efficients for speed-up and parallel scaling for the upwind SBP FD operators of order $2,3,4,5,6,7,8,9$.}
    \label{tab:scaling}
\end{table}

\begin{table}[!h]
\begin{center}
 \begin{tabular}{||l c c ||} 
 \hline
 & Upwind 6 & Traditional 6  \\ [0.5ex] 
 \hline\hline
 $m_1$ &
-1.01087 &
-1.01241 \\
$b_1$ &
2.75725 &
2.74747 \\
 \hline
 $m_2$ &
1.00183 &
1.00162 \\
 [1ex] 
 \hline
\end{tabular}
\end{center}
    \caption{Regression co-efficients for sixth order speed-up and parallel scaling for the traditional SBP FD operator of order $6$ and the upwind SBP FD  of order $6$.}
    \label{tab:scaling_o6}
\end{table}

From Table \ref{tab:scaling}-\ref{tab:scaling_o6}, note that $m_1 \approx -1$ and $m_2 \approx 1$, which indicate nearly perfect strong scaling.
%In the scaling plot, we see that each of the operators are follow parallel linear trends indicating that each of the orders scale at the same rate to one-another. 
%The intercepts of these plots represent the initial compute load on a single node, where this load increases with order.
%This increase is well-modelled with a $\log$ trend, indicating that computational effort between higher order stencils becomes increasingly negligible. 
%
%
In particular, the speed-up plot shows a nearly linear scaling with problem size for upwind SBP operators of order $2, 3, 4, 5$, For higher order operators,of order $ 6, 7, 8, 9$, super-linear scaling is observed, with a maximal scaling rate found for the 6th and 7th order operators.  %This shows that higher order stencils achieve better speed-up with increasing computational resources, with the 7th order stencil being the most efficient in our tests. 
 %%%%
 
 \subsubsection{The Zugspitze simulation}
For the Zugspitze simulation, we  run the simulation until the final time $t = 30$ s such that the elastic waves propagate through the media and leave the computational domain. As the waves propagate through the media, they interact with the complex topography and generate  high frequency scattered wave modes. Because of the complex non-planar topography, the Zugspitze model has no analytical solutions. To verify accuracy we have generated a reference solution using the upwind SBP FD  operator of order 6 on a resolved grid ($h = 100$~m). We note that the reference solution has been benchmarked against numerical solutions produced by ExaHyPE \cite{ExaHyPE2019,Duru_exhype_2_2019}, a DG solver, and both numerical solutions are in perfect agreement at sufficiently high frequencies. Numerical simulations using the traditional and upwind SBP FD  operator of order 6 on a resolved grid ($h = 100$~m) require about $33000$ CPU-hours, including I/O. At $h = 100$~m resolution, the discritisation generates about $8$~billion DoF for the evolving unknown vector field, about $8$~billion DoF for the right hand side and about $11.6$~billion DoF needed to store the mesh, material parameters, the Jacobian and metric parameters.  Our goal is to achieve similar accuracy with less computational resources by utilising upwind SBP operators.

 %with a wall-clock time of $\approx 0.9$ hours.
 
 We consider the lower grid-spacing ($ h= 200$~m)  mesh resolution,  the final time $t = 30$~s. Note again that at $ h= 200$~m mesh resolution the discritisation generates only $1$~billion DoF for the evolving unknown vector field, and about $1.45$~billion DoF needed to store the mesh, material parameters, the Jacobian and metric parameters. We run the simulation until the final time using $480$ CPU-cores  with a wall-clock time of $\approx 0.83$ hours. In total, the $100$ m grid-spacing experiments required about $\approx 33000$ CPU-hours, whilst the $200$ m grid-spacing only needs $\approx 400$ CPU-hours. These CPU-hours also include time and resources needed for I/O, to output the wave fields on the entire topography. The required supercomputing resources are listed in Table \ref{tab:supercomputing_o6}. It is also significantly important to note that the $ h= 200$~m mesh resolution experiments use only $\%1.21$ of the CPU-hours  required by the $100$ m grid-spacing experiments, thus saving about $\%98.78$ of supercomputing resources.

\begin{table}[!h]
\begin{center}
 \begin{tabular}{||l c c ||} 
 \hline
 & $h=100$~m & $h=200$~m  \\ [0.5ex] 
 \hline\hline
 Memory [GB] &
220.8 &
27.6 \\
CPU-hours &
33000 &
400 \\
%  \hline
%  $m_2$ &
% 1.00183 &
% 1.00162 \\
 [1ex] 
 \hline
\end{tabular}
\end{center}
    \caption{Supercomputing resources required by the traditional SBP FD operator of order $6$ and the upwind SBP FD  of order $6$.}
    \label{tab:supercomputing_o6}
\end{table}

We present the computational results in Figures \ref{fig:zugs_movie}, \ref{fig:zugs_22_22}, \ref{fig:zugs_all}, \ref{fig:zugs_40_40}. 
Figure \ref{fig:zugs_movie} shows the snapshots of the numerical solution propagating through time on the surface $\widehat{X}(y, z)$, at $h = 100$~m resolution, illustrating the propagation of the surface elastic waves through the complex free-surface topography. 
Here we can see the scattering of high frequency waves conforming to the complex free-surface topography present.
We have placed two stations, at near source $(y=22.4,z=22.4)$ and at the peak of Mount Zugspitze $(y=40,z=40)$ where the numerical solutions are sampled and compared for different SBP FD operators and at different resolutions.

Figure \ref{fig:zugs_22_22} shows the numerical solutions at the near-source station data at $(22.4,22.4)$ on the Earth surface.  In this figure, the numerical solutions of the upwind and traditional SBP FD operator of order $6$  are compared at $200$~m and $100$~m resolution. 
At   $h = 200$~m resolution the upwind SBP FD operator resolves the waveform sufficiently accurate without introducing spurious unresolved wave modes in the numerical solution. The traditional SBP operator resolves some of the important features in the waveform but it also introduces additional large amplitude spurious oscillations which can potentially destroy the accuracy of numerical simulations.
%At $h = 200m$ resolution, the numerical solution for the  upwind 6 operator performs equally well to
At $h=100$~m resolution the numerical solutions for traditional SBP operator  converge to the reference solution. 
Although, the spurious oscillation diminishes with increasing mesh resolutions, that is for $h=100$ m grid spacing, however, doubling the mesh resolution by using $h=100$ m grid spacing, increases the computational resources for the 3D problem by several orders of magnitude.

\begin{figure}[H]
    \centering
    \includegraphics[width=0.46\textwidth]{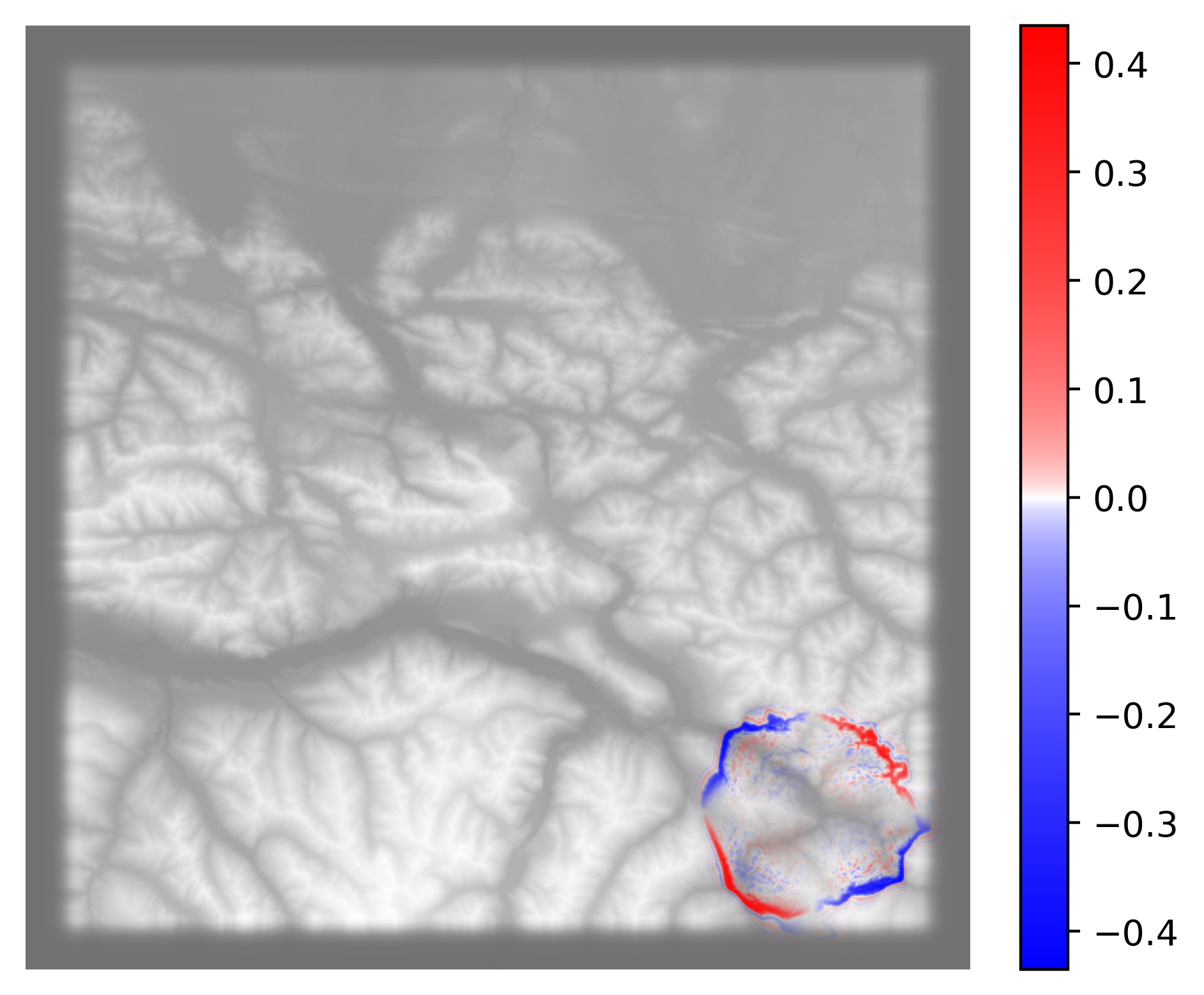}
    \includegraphics[width=0.46\textwidth]{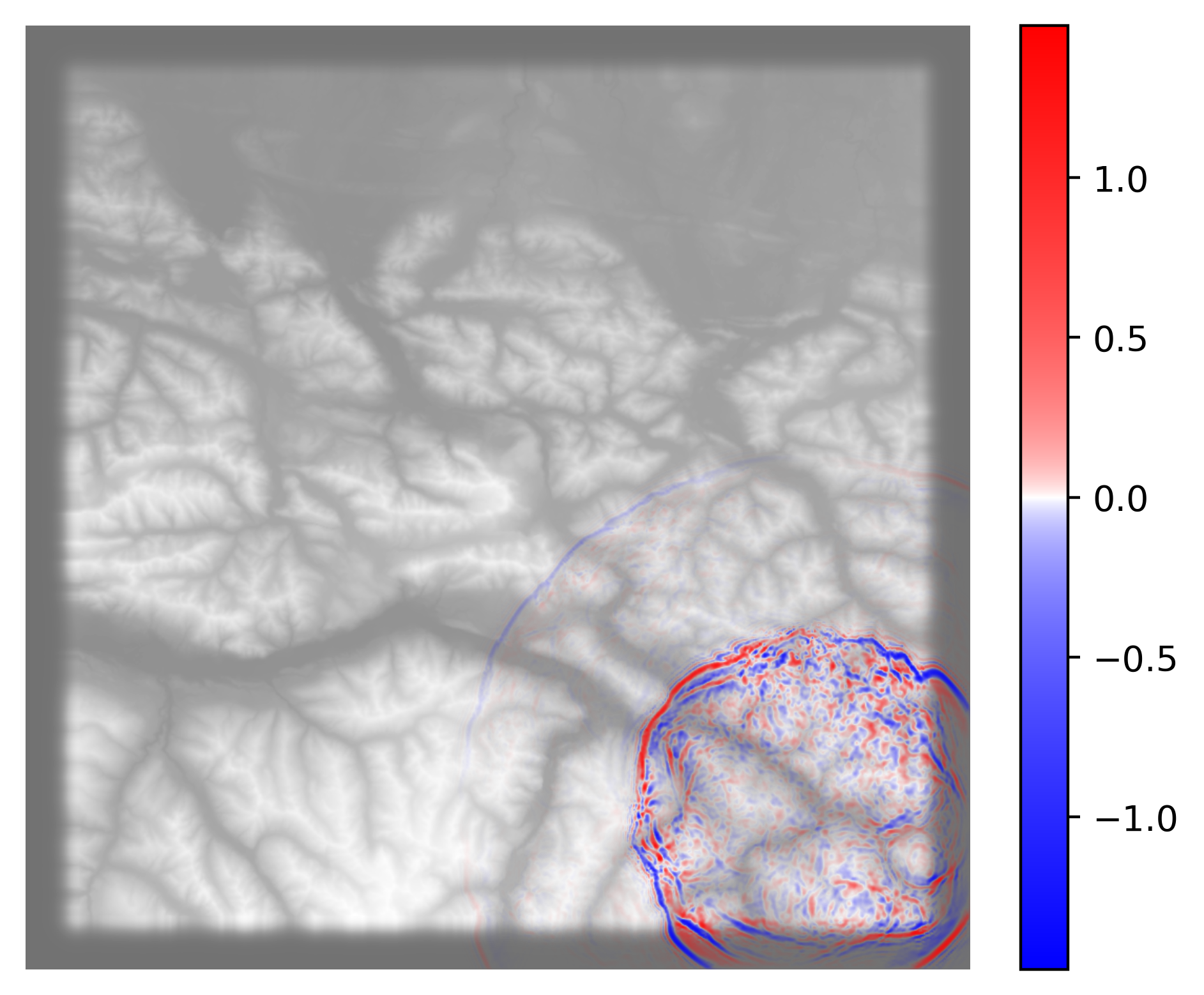}
    \includegraphics[width=0.46\textwidth]{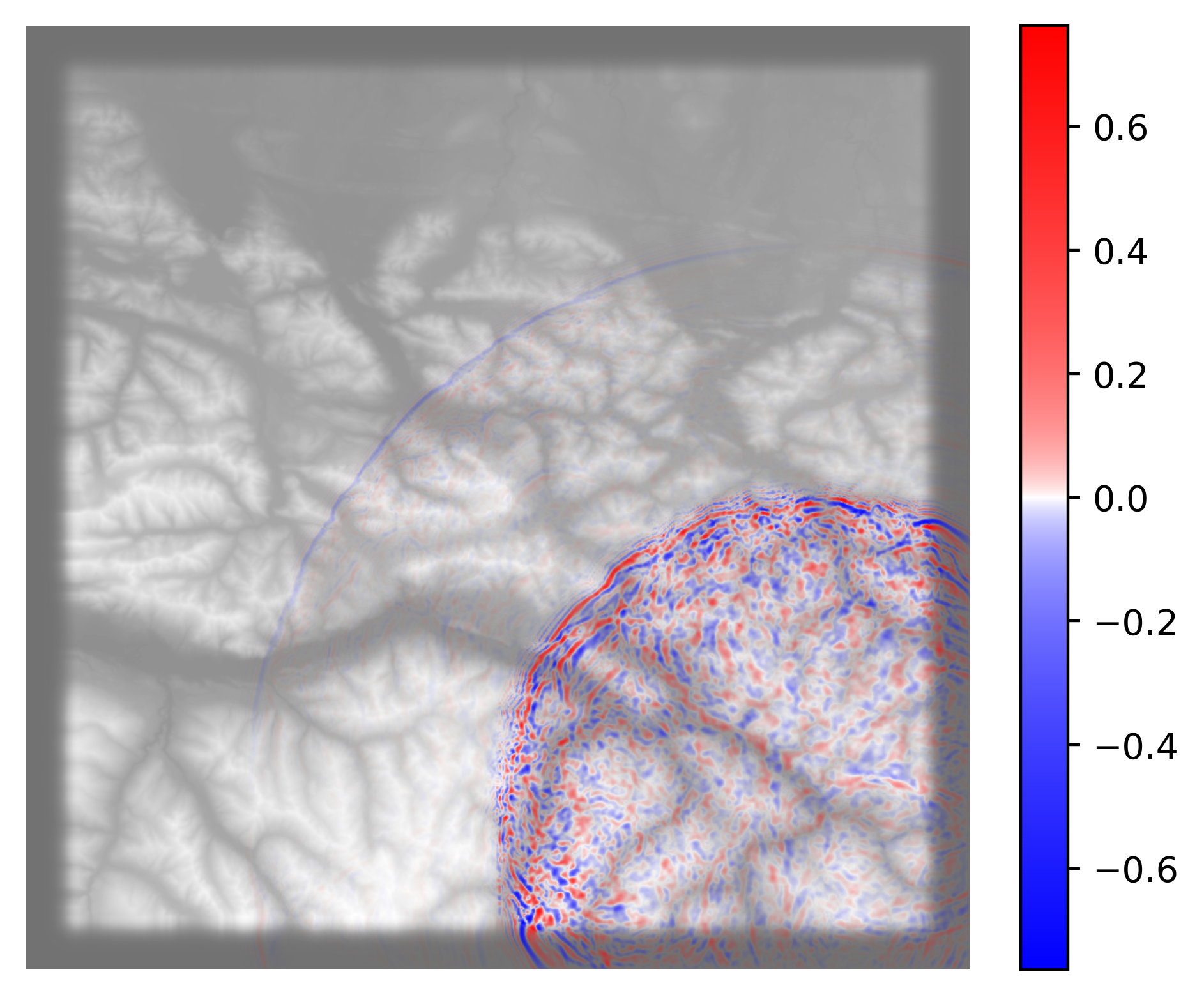}
    \includegraphics[width=0.46\textwidth]{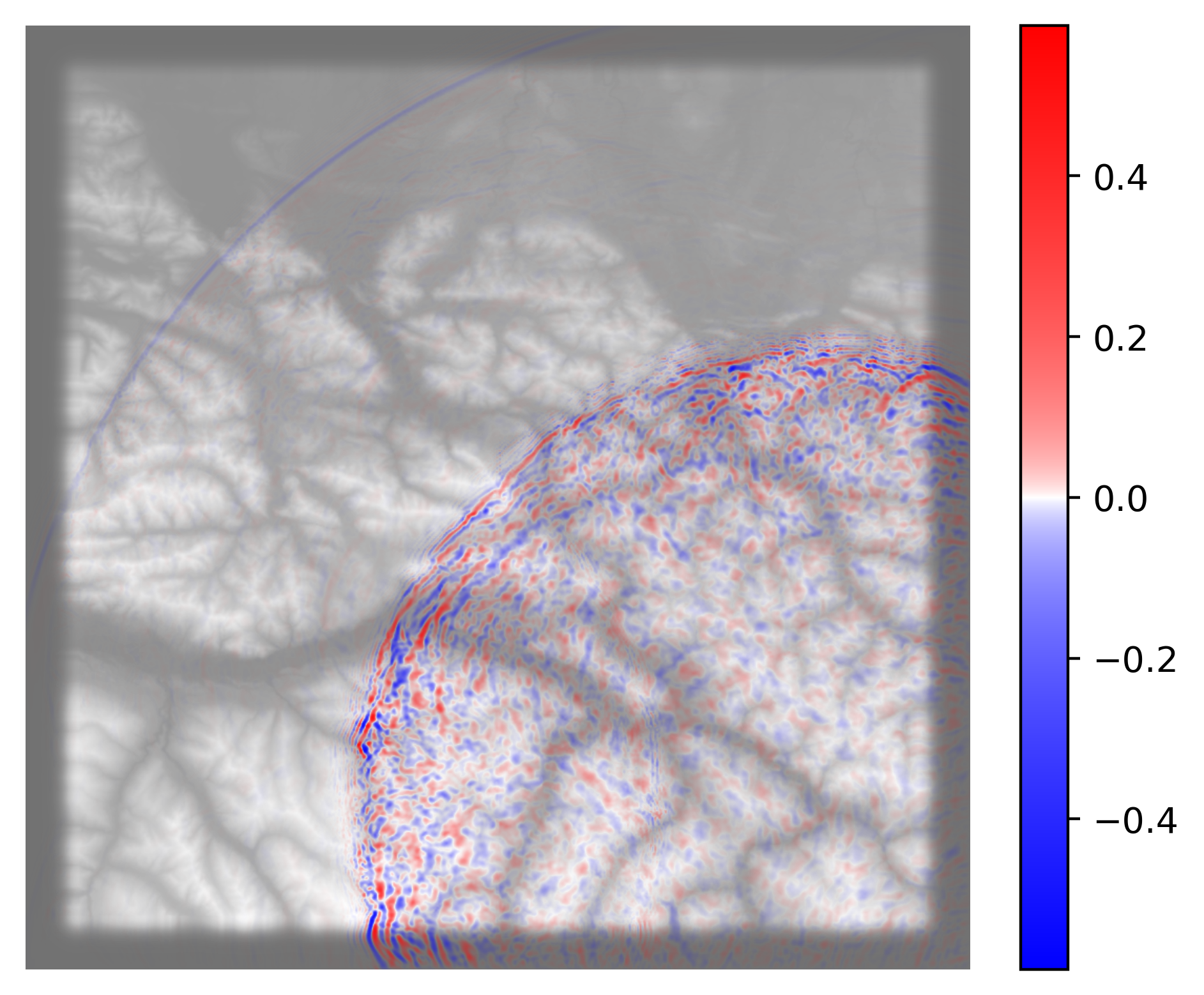}
    \caption{Snapshots of the simulated particle velocity on the geometrically complex free-surface topography   at $t \in \{2.56, 5.95, 9.18, 13\}$ seconds. The simulation is performed with the upwind SBP FD operator of order $6$. Here, the background grey represents the altitude given from the underlying topography. } 
    \label{fig:zugs_movie}
\end{figure}

\begin{figure}[H]
    \centering
    \includegraphics[width=\textwidth]{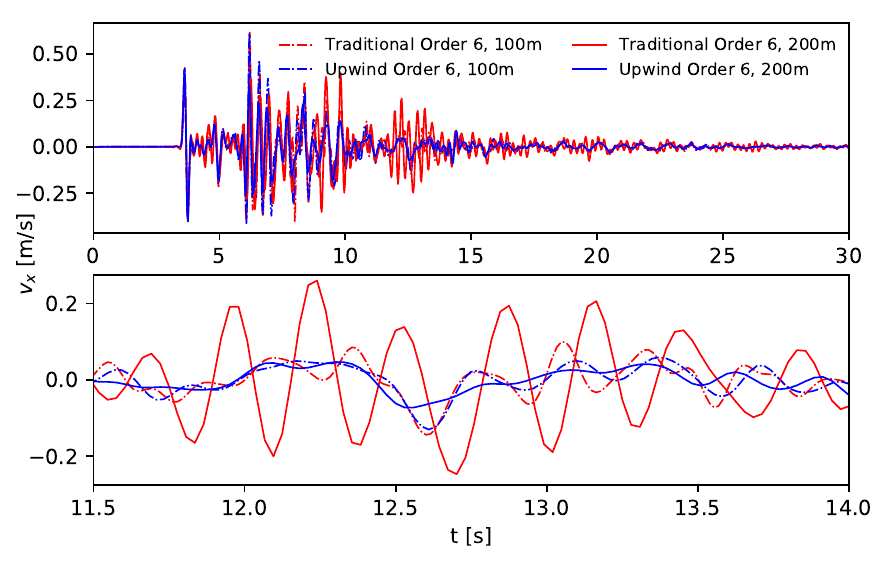}
    \caption{Seismograph from a near-source station placed at $(22.4,22.4)$ on the Earths surface. The lower plot is a zoomed section highlighting the spurious oscillations generated by the traditional SBP FD operator. }
    \label{fig:zugs_22_22}
\end{figure}

Figure \ref{fig:zugs_all} shows all of the upwind schemes, of order $2,3,4,5,6,7,8,9$, at $200$~m resolution compared to the reference $100$~m resolution solution, at the station placed at the peak of Mount Zugspitze $(y=40,z=40)$ on the Earths surface. Note that the numerical solutions converge to the reference solution as the order of accuracy increases. Again, for the marginally resolved mesh the accuracy of the solution becomes nearly optimal for the 6th order accurate upwind SBP operator.  Beyond the 6th order accurate upwind SBP operator, the higher order accurate SBP operators do not necessarily improve the accuracy of numerical solution for the upwind SBP operators without introducing spurious oscillations. We also identify that the odd-order operators have more spurious high frequency modes in their numerical solutions than their even-order operator counterparts.  As above, this is consistent with the numerical dispersion relation analysis of Section \ref{sec:dsipersion_relation} and numerical dispersion plots displayed in Figure \ref{fig:dispersion_all}.

%This again agrees with 

Finally, in Figure \ref{fig:zugs_40_40} we compare the traditional and upwind order 6 computations on a marginally resolved grid to the reference solution at the station placed at the peak of Mount Zugspitze $(y=40,z=40)$ on the Earths surface. At this resolution, the traditional operator solution has spurious high frequency wave modes which are not present in the upwind solution. The upwind 6 solution computed at $h = 200m$ resolution is then compared with the traditional order 6 solution computed at $h = 100m$ resolution against the reference solution. When compared to the reference, the upwind 6 on a marginally resolved mesh ($h = 200$~m) and the  traditional operators on a well resolved mesh ($h = 100$~m) capture the key features of the waveform. Much of the spurious oscillation in the traditional scheme has diminished with increasing mesh resolution, that is for $h=100$ m grid spacing. However, doubling the mesh resolution by using $h=100$ m grid spacing, increases the computational resources for the 3D problem by several orders of magnitude.

\begin{figure}[H]
    \centering
    \includegraphics[width=0.85\textwidth]{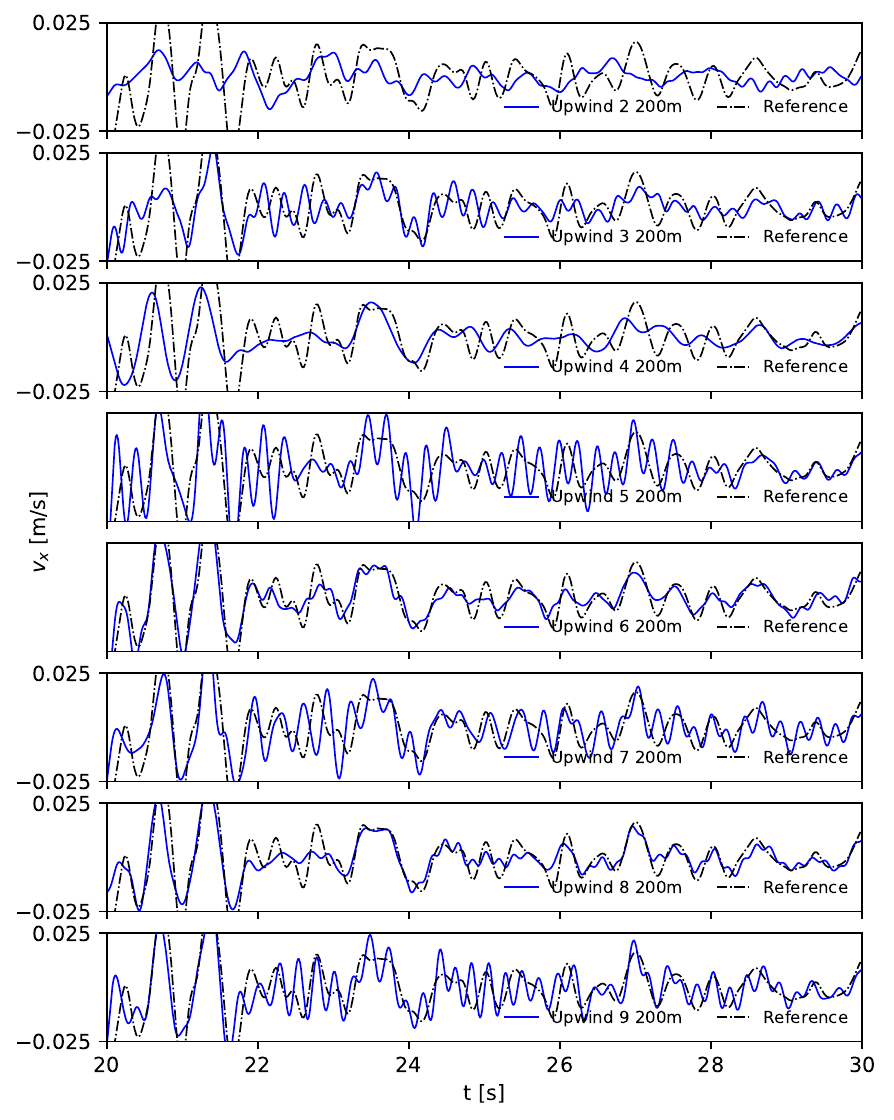}
    \caption{Seismograph from a station placed at the peak of Mount Zugspitze $(y=40,z=40)$ on the Earth's surface, truncated to see the coda waves. All odd-order $(3, 5, 7, 9)$ accurate SBP FD operators generate high frequency spurious modes.}
    \label{fig:zugs_all}
\end{figure}

% \begin{figure}[H]
%     \centering
%     %\includegraphics[width=0.85\textwidth]{paper_Station_at_22_4_22_4.pdf}
%     \includegraphics[width=\textwidth]{zugs6_40_40_x.pdf}
%     \caption{Seismograph from a station placed at $(40,40)$ on the Earths surface.}
%     \label{fig:zugs_40_40}
% \end{figure}

\begin{figure}[H]
    \centering
    \includegraphics[width=0.85\textwidth]{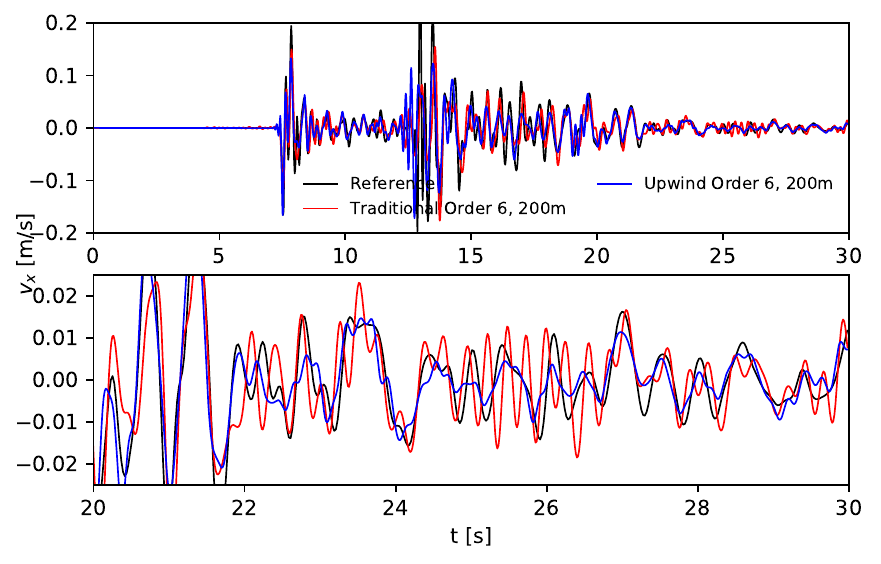}
    \includegraphics[width=0.85\textwidth]{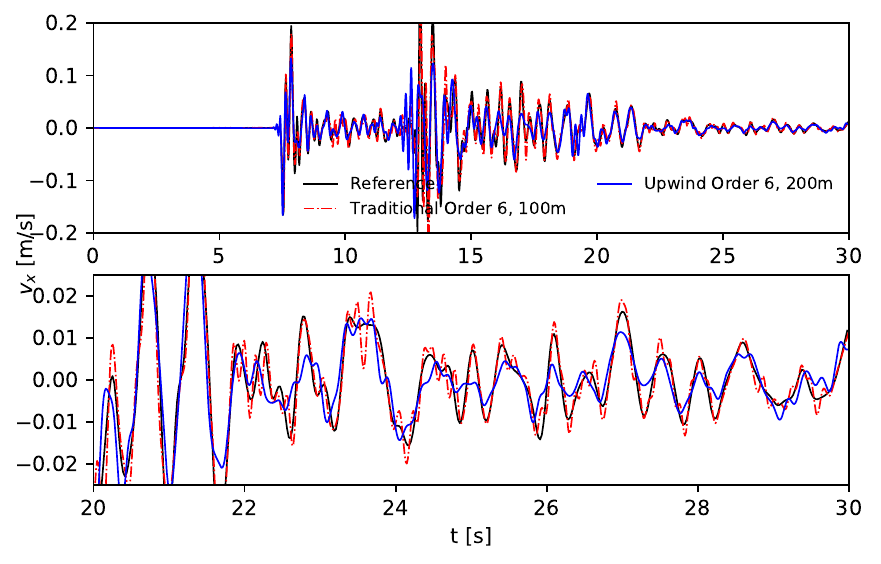}
    \caption{Seismograph's from numerical solutions of order 6 upwind and traditional operators from a station placed at the peak of Mount Zugspitze $(y=40,z=40)$ on the Earths surface. 
    In the top panel we compare the traditional and upwind solutions on marginal grids to a high resolution reference solution.
    In the lower panel we compare a upwind computation on a marginal grid to a high-resolution traditional computation. Note that in each panel the second plot is a zoomed section of the seismogram focusing on coda waves.}
    \label{fig:zugs_40_40}
\end{figure}

% \begin{figure}[H]
%     \centering
%     %\includegraphics[width=0.85\textwidth]{paper_Station_at_22_4_22_4.pdf}
%     \includegraphics[width=\textwidth]{40_40_6_x_ref.pdf}
%     \caption{Comparison of marginally resolved order 6 upwind and high-resolution order 6 traditional operator from a station placed at $(40,40)$ on the Earths surface. Reference solution is given from a high-resolution simulation.}
%     \label{fig:zugs_40_40_ref}
% \end{figure}

%
% \begin{figure}[H]
%     \centering
%     \includegraphics{paper_Station at (27,27).pdf}
%     \includegraphics{paper_Station at (40,40).pdf}
%     \caption{Plot to show low resolution ocillations}
%     \label{fd}
% \end{figure}

% \begin{figure}[H]
%     \centering
%     \includegraphics{paper_Station at (40,40).pdf}
%     \caption{Plot to show low resolution ocillations}
%     \label{fd}
% \end{figure}
%     \begin{figure}[h]
%     \centering
%     \includegraphics[width=\textwidth]{presentation/slides/Trad_6_100m_noerr.png}
% \end{figure}

%     \begin{figure}[h]
%     \centering
%     \includegraphics[width=\textwidth]{presentation/slides/Upw_5_100m_noerr.png}
% \end{figure}

\section{Summary and outlook}\label{sec:conclusions}
 High-order accurate SBP FD methods are attractive for efficient  large-scale numerical  simulation of hyperbolic wave propagation problems. 
% %%%%%
 Traditional SBP FD operators that approximate first-order spatial derivatives with central-difference stencils often have spurious unresolved wave-modes in their numerical solutions. 
% %%%%%
 On marginally resolved computational grids, these spurious wave-modes have the potential to destroy the accuracy of numerical solutions for a first-order hyperbolic partial differential equation, such as the elastic wave equation.  
% %%%%%
 Recently derived high order accurate upwind SBP operators \cite{Mattsson2017} based non-central (upwind) FD stencils have the potential to suppress  these poisonous spurious wave-modes on marginally resolved computational grids.
% %%%%%

 We have demonstrated in this paper that not all high order upwind SBP FD operators are applicable. Numerical dispersion relation analysis shows that odd-order upwind SBP FD operators also support spurious unresolved high frequency wave modes on marginally resolved meshes. Even-order upwind SBP FD operators (of order $2, 4, 6$) do not support spurious unresolved wave modes and have better numerical dispersion properties than traditional SBP FD operators and odd-order upwind SBP FD operators.
% %%%%

 To ensure the accuracy of numerical solutions of the three space dimensional (3D) elastic wave equation in complex geometries, we discretise the 3D elastic wave equation with a pair of upwind SBP operators, on boundary-conforming curvilinear meshes.
% %%%%%
 Using  the energy method we prove that the numerical method is stable, and energy conserving. We derive a priori error estimates and prove the convergence of the numerical error.
% %%%%%
% %Furthermore, computational results show the robustness of the scheme.
% %%%%%

 We presented numerical simulations of the 3D elastic wave equation in heterogeneous media with complex non-planar free surface topography, including numerical simulations of community developed seismological benchmark problems. 
% %%%%%
 Our results show that even-order upwind SBP  FD operators of order are more robust and less prone to numerical dispersion errors on marginally resolved meshes when compared to odd-order upwind SBP  FD operators and traditional SBP FD operators, thereby increasing efficiency.
% %%%%

 We performed  strong  scaling  tests, demonstrating nearly perfect strong scaling and   verifying  the  efficiency  of  our  parallel implementation  of  high  order  upwind  SBP  operators  for large scale elastic  wave  simulations  in  3D  complex geometries. We believe that the method and the software will increase the efficiency of numerical simulations of elastic waves in many applications such as earthquake engineering, natural minerals and energy resources exploration, as well in strong-ground motion analysis and underground fluid injection monitoring. 
% %%%%%
% Upwind SBP finite difference methods are derived and implemented efficiently for the simulation of large-scale elastic wave propagation problems in complex geometries. 
% Using the energy method we show that our implementation is numerically stable. 
% We demonstrate numerically that for certain upwind SBP operators the method effectively resolves scattered high-frequency elastic waves in 3D complex geometries with free surface boundary conditions. 
% When compared to traditional SBP operators based on central finite difference stencils, specific upwind counterparts at the same grid resolution are less prone to high-frequency poisonous wave modes. Furthermore, on marginally resolved meshes these upwind finite difference schemes have comparable accuracy to traditional SBP operators on finer mesh resolutions. 
% Thus, the upwind SBP operators give a significant gain in computational efficiency for large scale wave propagation problems in 3D-complex geometries.

Our preliminary 3D dynamic earthquake rupture simulations show promise in increasing efficiency through using upwind SBP operators with good dispersion relation properties for simulating nonlinear friction laws in elastic solids,  and earthquake source modelling. 

We have used the PML  to enable efficient domain truncation and prevent artificial numerical reflections, from the computational boundaries, from contaminating the numerical simulations. Because of the asymmetric properties of the PML and the upwind SBP operators, a stable implementation of the PML for the 3D IBVP elastic wave equation  using the upwind SBP operators is a non-trivial task. The details of the numerical treatment of the PML using upwind SBP operators will be reported in a forthcoming paper.

\appendix
%%%%%%%%%%%%%%%%%
\section{Hat-variables}\label{sec:hat_variables}
%%%%%%%%%%%%%%%%%%%%%%%
 The hat-variables encode the solution of the IBVP on the boundary/interface. The hat-variables  are constructed such that they preserve the amplitude of the outgoing waves and exactly satisfy the physical boundary conditions ~\cite{DuruGabrielIgel2017}. To be more specific, the hat-variables are solutions of the Riemann problem constrained against physical boundary conditions \eqref{eq:BC_General2}. We refer the reader to  ~\cite{DuruGabrielIgel2017,ExaHyPE2019} for more detailed discussions. Once the hat-variables are available, we construct physics based numerical flux fluctuations by penalizing data against the incoming characteristics \eqref{eq:characteristics} at the element faces.
%%%%%%%%%%%%%%%%%
%The difficulty lies in constructing stable discrete approximations and deriving discrete energy estimates, analogous to \eqref{eq:energy_estimate_fault_30} and \eqref{eq:energy_estimate_fault_40}, without introducing additional numerical stiffness.
%To succeed, we construct a subspace of functions (solutions) satisfying the elastic wave equation, the interface condition \eqref{eq:Interface_condition} and boundary conditions \eqref{eq:BC_General2},  and project the particle velocities and tractions on the boundary to that subspace.
% %%%%%%%%%%%%%%%%%%%
 
 %\subsection{Boundary data}\label{subsec:boundary_hat_variables}
 % %%%%%%%%%%%%%%%%%%%
For $Z_\eta  > 0$, we define the characteristics
% %%%%%%%%%%%%%%%%%%%
% %%%%%%%%%%%%%%%%%%%
\begin{align}\label{eq:charac}
q_\eta  = \frac{1}{2}\left(Z_\eta  v_\eta  + T_\eta \right), \quad
p_\eta  = \frac{1}{2}\left(Z_\eta  v_\eta  -  T_\eta \right), \quad  \eta \in \{ x, y, z \}.
\end{align}
 % %%%%%%%%%%%%%%%%%%%
 Here, $q_\eta$ are the left going waves, and $p_\eta$ are the right going waves.
% %%%%%%%%%%%%%%%%%%%
We will  construct boundary data which satisfy the physical boundary conditions \eqref{eq:BC_General2} exactly and preserve the amplitude of the outgoing waves  $q_\eta $ \text{at}   $\xi \equiv 0$, and  $p_\eta $ at $\xi \equiv 1$.
% %%%%%%%%%%%%%%%%%%%
% %%%%%%%%%%%%%%%%%%%
That is  introduce $\widehat{v}_\eta$ and $\widehat{T}_\eta$ such that 
% %%%%%%%%%%%%%%%%%%%
% %%%%%%%%%%%%%%%%%%%
{%\small
\begin{align}\label{eq:BC_hat_1bc}
&{q}_\eta \left(\widehat{v}_\eta, \widehat{T}_\eta , Z_\eta \right) = {q}_\eta \left({v}_\eta , {T}_\eta , Z_\eta \right),  \quad \text{at} \quad \xi \equiv 0,\\
\nonumber
&{p}_\eta \left(\widehat{v}_\eta , \widehat{T}_\eta , Z_\eta \right) = {p}_\eta \left({v}_\eta , {T}_\eta , Z_\eta \right), \quad \text{at} \quad \xi \equiv 1.
\end{align}
}
%%%%%%%%%%%%%%%%%%%%%%%
%%%%%%%%%%%%%%%%%%%%%%%
The variables $\widehat{v}_\eta$ and $\widehat{T}_\eta$ should also satisfy the physical boundary condition \eqref{eq:BC_General2}, and we have
{
\begin{align}\label{eq:BC_hat_2bc}
&\frac{Z_{\eta}}{2}\left({1-\gamma_\eta }\right)\widehat{v}_\eta  -\frac{1+\gamma_\eta }{2} \widehat{T}_\eta  = 0, \quad \text{at} \quad \xi \equiv 0,
\\
\nonumber
&\frac{Z_{\eta}}{2} \left({1-\gamma_\eta }\right)\widehat{v} _\eta + \frac{1+\gamma_\eta }{2}\widehat{T}_\eta  = 0, \quad \text{at} \quad \xi \equiv 1.
\end{align}
}
%%%%%%%%%%%%%%%%%%%%%%%
%%%%%%%%%%%%%%%%%%%%%%%
The algebraic problem defined by equations \eqref{eq:BC_hat_1bc} and \eqref{eq:BC_hat_2bc},  has a unique solution, namely
%%%%%%%%%%%%%%%%%%%%%%%
%%%%%%%%%%%%%%%%%%%%%%%
{
\begin{align}\label{eq:data_hat}
&\widehat{v}_\eta   = \frac{(1+\gamma_\eta )}{Z_{\eta}}q_\eta , \quad \widehat{T}_\eta   = {(1-\gamma_\eta )}q_\eta ,  \quad \text{at} \quad \xi \equiv 0,
\\
\nonumber
%%%%
&\widehat{v}_\eta  = \frac{(1+\gamma_\eta )}{Z_{\eta}}p_\eta ,  \quad \widehat{T}_\eta   = {-(1-\gamma_\eta )}p_\eta , \quad \text{at} \quad \xi \equiv 1.
\end{align}
}
%%%%
The expressions in \eqref{eq:data_hat} define a rule to update particle velocity vector and traction vector on the boundaries $\xi = 0, 1$. That is 
{%\small
\begin{align}\label{eq:boundary_data_hat}
v_\eta   = \widehat{v}_\eta  , \quad {T}_\eta  = \widehat{T}_\eta  , \quad \text{at} \quad \xi \equiv 0, 1.
\end{align}
}
%%%%%%%%%%%%%%%%%%%
  
 %%%%%%%%%%%%%%%%%%%
 %%%%%%%%%%%%%%%%%%%
 Note in particular that the hat-variables $\widehat{v}_\eta , \widehat{T}_\eta $,  satisfy the following  inequalities 
%%%%%%%%%%%%%%%%%
% \begin{subequations}\label{eq:identity_bc_app}
% %\small
% \begin{equation}\label{eq:identity_1_bc}
% {q}_\eta \left(\widehat{v}_\eta , \widehat{T}_\eta , Z_\eta \right) = {q}_\eta \left({v}_\eta , {T}_\eta , Z_\eta \right), \quad \text{at} \quad \xi = -1,
% %%%%%%%%%%%%%%%%%
% %%%%%%%%%%%%%%%%%
%  \quad
%  %%%%%%%%%%%%%%%%%
%  %%%%%%%%%%%%%%%%%
% {p}_\eta \left(\widehat{v}_\eta , \widehat{T}_\eta , Z_\eta \right) = {p}_\eta \left({v}_\eta , {T}_\eta , Z_\eta \right) , \quad \text{at} \quad \xi = 1,
% \end{equation}
% \begin{equation}\label{eq:identity_2_bc}
% {q}_\eta ^2\left({v}_\eta , {T}_\eta , Z_\eta \right)-{p}_\eta ^2\left(\widehat{v}_\eta , \widehat{T}_\eta , Z_\eta \right) = Z_\eta \widehat{T}_\eta \widehat{v}_\eta ,  \quad \text{at} \quad \xi = -1, \quad {p}_\eta ^2\left({v}_\eta , {T}_\eta , Z_\eta \right) -{q}_\eta ^2\left(\widehat{v}_\eta , \widehat{T}_\eta , Z_\eta \right) = -Z_\eta \widehat{T}_\eta \widehat{v}_\eta ,  \quad \text{at} \quad \xi = 1,
% \end{equation}
\begin{equation}\label{eq:identity_3_bc}
\begin{split}
 &\widehat{T}_\eta \widehat{v}_\eta  = \frac{1-\gamma_\eta ^2}{Z_{\eta}}{q}_\eta ^2\left({v}_\eta , {T}_\eta , Z_\eta \right) \ge 0,  \quad \text{at} \quad \xi \equiv 0, \\
 &\widehat{T}_\eta \widehat{v}_\eta  = -\frac{1-\gamma_\eta ^2}{Z_{\eta}}{p}_\eta ^2\left({v}_\eta , {T}_\eta , Z_\eta \right) \le 0 , \quad \text{at} \quad \xi \equiv 1.
%%$
\end{split}
\end{equation}
%\end{subequations}
%%%%%%%%%%%%%%%%%
%%%%%%%%%%%%%%%%%
The inequalities \eqref{eq:identity_3_bc} will be crucial in proving numerical stability.
%%%%%%%%%%%%%%%%%
Please see also  ~\cite{DuruGabrielIgel2017} for  more details.

\section{LOH1 Station 9} 
\begin{figure}[H]
    \centering
    \includegraphics[width=\textwidth]{LOH1_9x_100m_.pdf}
    \includegraphics[width=\textwidth]{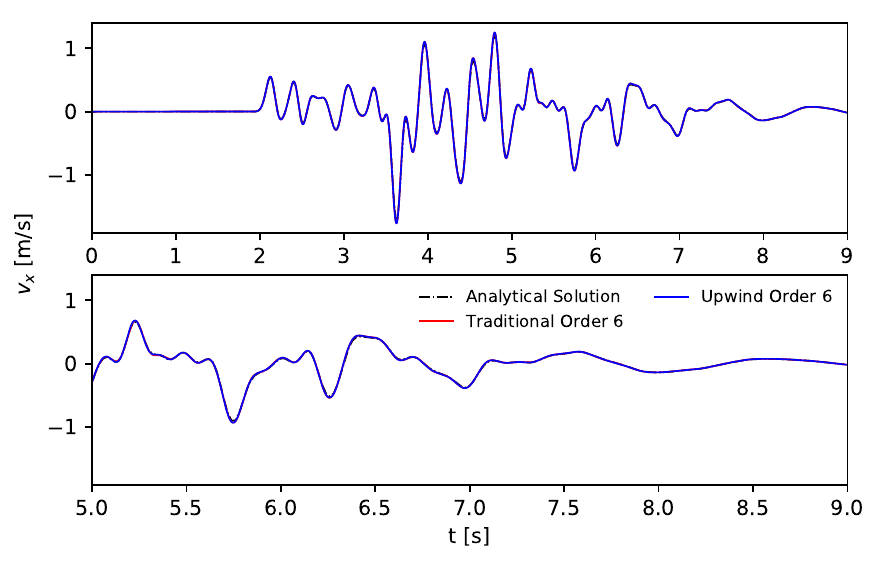}
    \caption{Time history of the particle velocity vector at ``Station 9" $(x_r, y_r, z_r)=(0, 8.647,8.647)$ with two levels of mesh refinements $h = 100$ m and $h =50$ m.}
    \label{fig:LOH1_9}
\end{figure}

%\end{theorem}

\section*{Acknowledgments}
This research was undertaken with the assistance of resources and services from the National Computational Infrastructure (NCI), which is supported by the Australian Government. The authors also gratefully acknowledge the Gauss Centre for Supercomputing e.V. \footnote{www.gauss-centre.eu} for funding this project by providing computing time on the GCS Supercomputer SuperMUC-NG at Leibniz Supercomputing Centre \footnote{www.lrz.de}. Frederick Fung and Christopher Williams acknowledge support from the Australian Government Research Training Program Scholarship.

\bibliographystyle{plain}
\bibliography{refs}

\end{document}